\documentclass[12pt]{amsart}
\usepackage{amsmath}
\usepackage{amsfonts}
\usepackage{amssymb}
\usepackage[all,cmtip]{xy}           %xypic macro for latex2.09

\usepackage{bbm}
\usepackage{bbding}
\usepackage{txfonts}
\usepackage[shortlabels]{enumitem}
\usepackage{ifpdf}
\ifpdf
\usepackage[colorlinks,final,backref=page,hyperindex]{hyperref}
\else
\usepackage[colorlinks,final,backref=page,hyperindex,hypertex]{hyperref}
\fi
\usepackage{tikz-cd}
\usepackage[active]{srcltx}
\usepackage{tikz}
\usepackage{amscd}
\usepackage{bm}
\usepackage{mathrsfs}
\usepackage{bbding}

\makeatletter

\newtheorem{thm}{Theorem}[section]
\newtheorem{prop}[thm]{Proposition}
\newtheorem{lem}[thm]{Lemma}

\newtheorem{prop-def}[thm]{Proposition-Definition}

\theoremstyle{definition}
\newtheorem{defn}[thm]{Definition}

\newtheorem{remark}[thm]{Remark}
\newtheorem{exam}[thm]{Example}
\newcommand{\nc}{\newcommand}

\nc{\delete}[1]{{}}

\nc{\mlabel}[1]{\label{#1}}  % Use this to suppress names
\nc{\mcite}[1]{\cite{#1}}  % Use this to suppress names
\nc{\mref}[1]{\ref{#1}}  % Use this to suppress names
\nc{\meqref}[1]{\eqref{#1}}  % Use this to suppress names
\nc{\mbibitem}[1]{\bibitem{#1}} % Use this to show number
\delete{
	\nc{\mlabel}[1]{\label{#1} {{\emph{{\ }\ (#1)}}}}				 % Use this lines to show names
	\nc{\mcite}[1]{\cite{#1}{{\emph{{\ }(#1)}}}}  % Use this lines to show names
	\nc{\mref}[1]{\ref{#1}{{\emph{{\ }(#1)}}}}  % Use this lines to show names
	\nc{\meqref}[1]{\eqref{#1}{{\emph{{\ }(#1)}}}}  % Use this lines to show names
	\nc{\mbibitem}[1]{\bibitem[\bf #1]{#1}} % Use this to show name
}

\nc{\mrm}[1]{{\rm #1}}
\nc{\name}[1]{{\bf #1}}
\nc{\tforall}{\ \text{for all }}

\nc{\la}{\longrightarrow}
\nc{\ot}{\otimes}
\nc{\rar}{\rightarrow}

\nc{\ac}{\mathrm{\textup{!`}}}

\nc{\Alg}{{\mathrm{Alg}}}
\nc{\bfk}{{\bf k}}
\nc{\C}{{\mathrm{C}}}
\nc{\DA}{{\mathsf{DA}_\lambda}}
\nc{\Dif}{{{}_\lambda\!\mathfrak{Dif}}}
\nc{\Difinfty}{{{}_\lambda\!\mathfrak{Dif}_\infty}}
\nc{\DO}{{\mathsf{DO}_\lambda}}
\nc{\End}{\mrm{End}}
\nc{\Ext}{\mrm{Ext}}
\nc{\Fil}{\mrm{Fil}}
\nc{\Fr}{\mrm{Fr}}
\nc{\Frob}{\mrm{Frob}}
\nc{\Gal}{\mrm{Gal}}
\nc{\GL}{\mrm{GL}}
\nc{\Hom}{\mrm{Hom}}
\nc{\Hoch}{\mrm{Hoch}}
\nc{\hsr}{\mrm{H}}
\nc{\hpol}{\mrm{HP}}
\nc{\im}{\mrm{im}}
\nc{\Id}{\mrm{Id}}
\nc{\id}{\mrm{Id}}
\nc{\ID}{\mrm{ID}}
\nc{\Irr}{\mrm{Irr}}
\nc{\incl}{\mrm{incl}}
\nc{\length}{\mrm{length}}
\nc{\NLSW}{\mrm{NLSW}}
\nc{\Lie}{\mrm{Lie}}
\nc{\mchar}{\rm char}
\nc{\mpart}{\mrm{part}}
\nc{\ql}{{\QQ_\ell}}
\nc{\qp}{{\QQ_p}}
\nc{\rank}{\mrm{rank}}
\nc{\rcot}{\mrm{cot}}
\nc{\rdef}{\mrm{def}}
\nc{\rdiv}{{\rm div}}
\nc{\rmH}{ {\mathrm{H}}}
\nc{\rtf}{{\rm tf}}
\nc{\rtor}{{\rm tor}}
\nc{\res}{\mrm{res}}
\nc{\Sh}{{\mathrm{Sh}}}
\nc{\SL}{\mrm{SL}}
\nc{\Spec}{\mrm{Spec}}
\nc{\sgn}{{\mathrm{sgn}}}
\nc{\tor}{\mrm{tor}}
\nc{\Tr}{\mrm{Tr}}
\nc{\tr}{\mrm{tr}}
\nc{\wt}{\mrm{wt}}
\nc{\op}{\mrm{op}}
\nc{\rmS}{\mrm{S}}

\nc{\BA}{{\mathbb A}}   \nc{\CC}{{\mathbb C}}
\nc{\DD}{{\mathbb D}}   \nc{\EE}{{\mathbb E}}
\nc{\FF}{{\mathbb F}}   \nc{\GG}{{\mathbb G}}
\nc{\HH}{ \mathrm{HH}}   \nc{\LL}{{\mathbb L}}
\nc{\NN}{{\mathbb N}}   \nc{\PP}{{\mathbb P}}
\nc{\QQ}{{\mathbb Q}}   \nc{\RR}{{\mathbb R}}
\nc{\TT}{{\mathbb T}}   \nc{\VV}{{\mathbb V}}
\nc{\ZZ}{{\mathbb Z}}   \nc{\TP}{\widetilde{P}}
\nc{\m}{{\mathbbm m}}

\nc{\cala}{{\mathcal A}}    \nc{\calc}{{\mathcal C}}
\nc{\cald}{\mathcal{D}}     \nc{\cale}{{\mathcal E}}
\nc{\calf}{{\mathcal F}}    \nc{\calg}{{\mathcal G}}
\nc{\calh}{{\mathcal H}}    \nc{\cali}{{\mathcal I}}
\nc{\call}{{\mathcal L}}    \nc{\calm}{{\mathcal M}}
\nc{\caln}{{\mathcal N}}    \nc{\calo}{{\mathcal O}}
\nc{\calp}{{\mathcal P}}    \nc{\calr}{{\mathcal R}}
\nc{\cals}{{\mathcal S}}    \nc{\calt}{{\mathcal T}}
\nc{\calv}{{\mathcal V}}    \nc{\calw}{{\mathcal W}}
\nc{\calx}{{\mathcal X}}

\nc{\fraka}{{\mathfrak a}}
\nc{\frakb}{\mathfrak{b}}
\nc{\frakg}{{\frak g}}
\nc{\frakl}{{\frak l}}
\nc{\fraks}{{\frak s}}
\nc{\frakB}{{\frak B}}
\nc{\frakm}{{\frak m}}
\nc{\frakM}{{\frak M}}
\nc{\frakp}{{\frak p}}
\nc{\frakW}{{\frak W}}
\nc{\frakX}{{\frak X}}
\nc{\frakS}{{\frak S}}
\nc{\frakA}{{\frak A}}
\nc{\frakC}{{\frak{C}}}
\nc{\frakx}{{\frakx}}
\nc{\frakt}{{\mathfrak{T}}}

\nc{\lir}[1]{\textcolor{red}{\underline{Li:}#1 }}

\nc{\rev}[1]{\textcolor{red}{#1 }}

\begin{document}

\title[Homotopy differential algebras with weight]{Koszul duality, minimal model and $L_\infty$-structure for   differential algebras  with   weight
}

\author{Jun Chen, Li Guo, Kai Wang and Guodong Zhou}
\address{Jun Chen,  Kai Wang and Guodong Zhou, School of Mathematical Sciences, Key Laboratory of Mathematics and Engineering Applications (Ministry of Education), Shanghai Key laboratory of PMMP,
  East China Normal University,
 Shanghai 200241,
   China}
   \email{632159784@qq.com, wangkai@math.ecnu.edu.cn, gdzhou@math.ecnu.edu.cn}

\address{Li Guo,
	Department of Mathematics and Computer Science,
	Rutgers University,
	Newark, NJ 07102, US}
\email{liguo@rutgers.edu}

\date{\today}

\begin{abstract} A differential algebra with weight is an abstraction of both the derivation (weight zero) and the forward and backward difference operators (weight $\pm 1$). In 2010 Loday established the Koszul duality for the operad of differential algebras of weight zero. He did not treat the case of nonzero weight, noting that new techniques are needed since the operad is no longer quadratic. This paper continues Loday's work and establishes the Koszul duality in the case of nonzero weight. In the process, the minimal model and the Koszul dual homotopy cooperad of the operad governing differential algebras with weight are determined. As a consequence, a notion of homotopy differential algebras with weight is obtained and the deformation complex as well as its $L_\infty$-algebra structure  for differential algebras with weight are deduced.
\end{abstract}

\subjclass[2010]{
18M70  %Algebraic operads, cooperads, and Koszul duality
12H05   %differential algebra
18M65  %Non-symmetric operads, multicategories, generalized multicategories
12H10   %difference algebra
16E40   %(co)homology of rings and algebras
16S80   %deformations of rings
%16W25   %derivations
%16S70   %extensions of rings by ideas
18M60 %Operads (general)
}

\keywords{cohomology, deformation, differential algebra, homotopy cooperad,  homotopy differential algebra,   Koszul dual, $L_\infty$-algebra,  minimal model, operad}

\maketitle

\tableofcontents

\allowdisplaybreaks

\section{Introduction}
In this paper we develop the Koszul duality theory for  differential algebras with nonzero weight by determining  the minimal model and the  Koszul homotopy dual cooperad of  the corresponding operad as well as the $L_\infty$-structure on the deformation complex, thereby solving a problem raised by Loday~\mcite{Lod}.

\subsection{Differential algebras with weight}\

For a fixed scalar $\lambda$, a \name{differential operator of weight $\lambda$}~\cite{GK} (called $\lambda$-derivation in~\cite{Lod}) is a linear operator $d=d_\lambda$ on an associative algebra $A$ satisfying the operator identity

\begin{equation} \mlabel{eq:diffwt}
	d(uv)=d(u)v+ud(v)+\lambda d(u)d(v), \quad u, v\in A.
\end{equation}
When $\lambda=0$, this is the Leibniz rule of the derivation in analysis. When $\lambda\neq 0$, this is the operator identity satisfied by the differential quotient $d_\lambda(f)(x):=(f(x+\lambda)-f(x))/\lambda$ in the definition of the derivation. The special cases of $\lambda=\pm 1$ gives the forward and backward difference operators in numerical analysis and other areas. An associative algebra with a differential operator of weight $\lambda$ is called a  \name{differential algebra of weight $\lambda$}.

Traditionally, a differential algebra is referring to a differential commutative algebra or a field of weight zero, originated nearly a century ago from the pioneering work of Ritt~\mcite{Rit1} in his algebraic study of differential equations.
Through the works of Kaplansky, Kolchin, Magid, Singer and many others~\mcite{Kap,Kol,Mag,PS}, the subject has evolved into a well-established area in mathematics. In theory its directions include differential Galois theory, differential algebraic geometry and differential algebraic groups. Its importance is also reflected by its connections with broad areas from arithmetic geometry and logic to computer science and mathematical physics, such as noncommutative Geometry, quantum Fields and motives~\mcite{CM}, differential algebraic geometry ~\mcite{FLS}, dynamical systems ~\mcite{MS}, and mechanical theorem-proving ~\mcite{Wu2}.

 A commutative differential algebra naturally gives rise to a Novikov algebra by the well-known work of Gelfand and Dorfman~\mcite{GD}, see also~\mcite{KSO}. Differential algebras are also closed related  to newly emerged structures including  Novikov-Poisson algebras~\mcite{Xu}, transposed Poisson algebras~\mcite{BBGW},  anti pre-Lie (Poisson) algebras~\mcite{LB}, and Novikov bialgebras~\mcite{HBG}.

More generally, differential algebras without the commutativity condition have attracted the interest of many recent studies. Derivations on path algebras and universal enveloping algebras of differential Lie algebras were studied in~\mcite{GL,Poi,Poi1}.
From the categorical viewpoint, differential categories were introduced and then studies in~\mcite{BCLS,CL,Lem}.
In~\mcite{ATW}, the notion of a differential algebra was generalized to non(anti)commutative superspace by deformation in the study of instantons in string theory.

\subsection{Cohomology and deformations of differential algebras}\

Beginning with the pioneering works of Gerstenhaber for associative algebras and of Nijenhuis and Richardson for Lie algebras~\mcite{Ge1,Ge2,NR}, deformation, cohomology and homotopy theory have been fundamental in understanding algebraic structures.
A general philosophy, as evolved from ideas of Gerstenhaber, Nijenhuis, Richardson, Deligne, Schlessinger, Stasheff, Goldman,   Millson etc,  is that  the deformation theory of any given
mathematical object can be described  by
a certain differential graded (dg) Lie algebra or more generally a $L_\infty$-algebra associated to the
mathematical object  (whose underlying complex is called the deformation complex).  This philosophy has been made into a theorem in characteristic zero by Lurie \mcite{Lur}   and Pridham \mcite{Pri10}, expressed in terms of infinity categories. It is an important problem to explicitly construct the dg Lie algebra or $L_\infty$-algebra governing the deformation theory of the given mathematical object. For algebraic structures (operads) with binary operations, especially the binary quadratic operads, there is a general theory in which the deformations and homotopy of an algebraic structure are controlled by  a dg Lie algebra naturally coming from the operadic Koszul duality for the given algebraic structure;  see~\cite{GJ94,GK94} for the original literature and~\mcite{LV,MSS} for general treatises.

Studies of cohomology and homotopy for algebra structures with linear operators have become very active recently. The structures include differential operators and Rota-Baxter operators on associative and Lie algebras.   The developments began with~\mcite{TBGS1} which defined the controlling cohomologie as well as the  dg Lie algebra for relative Rota-Baxter operators on Lie algebras. Subsequently,
the controlling cohomologie as well as the $L_\infty$-structure    for relative Rota-Baxter   Lie algebras were found in~\mcite{LST21};
derivations (=differential operators of weight zero) on Lie algebras were considered and their cohomology theory was introduced in~\mcite{TFS};
a cohomology theory of crossed homomorphisms (=differential operators of weight $1$) on Lie algebras was established in~\mcite{PSTZ}, together with a Maurer-Cartan characterisation.
The reference~\mcite{GLSZ} introduced a cohomology theory of differential associative algebras of arbitrary weight by an ad hoc method. The paper is a predecessor of the present one and whose results will be recalled in Subsection~\ref{sec:cohomologyda}.

Due to the complexity of the algebraic structures, the $L_\infty$-algebras in these studies so far have been obtained by direct constructions or by using the derived bracket technique  without using the language of operads.

The first operadic study of algebras with linear operators is the 2010 work of Loday~\mcite{Lod} on differential algebras of weight zero. Since this operad is quadratic, it can be treated applying the above-mentioned general theory of Koszul duality~\mcite{GJ94,GK94}. In particular, Loday obtained the corresponding minimal model and cohomology theory. For further studies, see works~\mcite{DL16,KS06,LT13}.
For the case of nonzero weight, Loday made the following observation in \mcite{Lod}, where the operad $\lambda$-AsDer is just the operad of differential algebras of weight $\lambda$.

\begin{quote}
If the parameter $\lambda$ is different from 0, then the operad $\lambda$-AsDer is not a quadratic operad since the term $d(a)d(b)$ needs three generating operations to
be defined. So one needs new techniques to extend Koszul duality to this case.
\end{quote}

The goal of this paper is to address this problem observed by Loday.

\subsection{Our approach   and layout of the paper}\

One of the most important consequences of the general Koszul duality theory for a quadratic operad is that the theory directly gives the minimal model \mcite{GK94, GJ94}, that is, the cobar construction of the Koszul dual cooperad. Then the deformation complex as well as its $L_\infty$-algebra structure  (in fact a dg Lie algebra in the Koszul case) could be deduced from the minimal model \mcite{KS00}.

With the absence of the quadratic property for differential algebras with nonzero weight, the existing general Koszul duality theory does not apply, as noted by Loday as above. The purpose of this work is provide such a theory for differential algebras with nonzero weight, in hope to shed light on a general theory beyond the quadratic operads.

To get started, we use a previous work~\mcite{GLSZ}, in which a cohomology theory of differential algebras with nonzero weight was constructed from a specifically constructed cochain complex, and was shown to control the formal deformations as well as abelian extensions of the differential algebras. We first determine the  $L_\infty$-algebra structure on this deformation complex,  in an ad hoc way, such that differential algebras with weight are realized as Maurer-Cartan elements for the $L_\infty$-algebra. This suggests that the $L_\infty$-algebra should come from the minimal model of the operad $\Dif$ of differential algebras with weight. Thus instead of verifying the $L_\infty$ property directly, we put this $L_\infty$-algebra in a suitable operadic context that is comparable to the weight zero case obtained in Loday's work~\mcite{Lod}.
We introduce homotopy differential algebras with nonzero weight  and show that their operad is indeed the minimal model of the operad $\Dif$.   As a byproduct, we obtain  the Koszul dual homotopy cooperad of $\Dif$.

\smallskip

Her is an outline of the paper.

In \S\mref{sec:difinf},   after recalling in \S\mref{sec:cohomologyda} the needed notions and background on the cohomology theory developed in \mcite{GLSZ}, we propose an $L_\infty$ structure on the deformation complex for differential algebras with nonzero weight (Theorem~\mref{th:linfdiff})  in \S\mref{ss:difinf}. In \S\mref{ss:mce}, we show that this $L_\infty$-algebra structure is the right one that realizes differential algebras with weight as the Maurer-Cartan elements (Proposition~\mref{pp:difopcochain}), while holding off the verification of the $L_\infty$ property to the end of paper, until the operadic tools are developed.

In \S\mref{sec:dualhomocoop}, once the  required basics on homotopy (co)operads are collected in \S\mref{ss:homocood},  we construct in \S\mref{ss:dual} a homotopy cooperad $\Dif^\ac$,  called  the  Koszul dual homotopy cooperad of  the operad $\Dif$ of differential algebras with weight because its cobar construction is shown to be the minimal model of $\Dif$  (Theorem~\mref{thm:difmodel}).

In \S\mref{sec:homomodel}, we first introduce the dg operad  $\Difinfty$ of homotopy   differential algebras with weight  in \S\mref{ss:operad}. We  establish in \S\mref{sec:model} that the dg operad $\Difinfty$ is the minimal model of   $\Dif$  (Theorem~\mref{thm:difmodel}). We give the notion of  homotopy   differential algebras and its explicit description (Definition~\mref{de:homodifalg} and Proposition~\ref{de:homodifalg2}) in \S\mref{ss:homo}.

In \S\mref{ss:modelinf},  we determine in \S\mref{ss: Linifty} the the $L_\infty$-algebra coming from the Koszul dual homotopy cooperad and show that it coincides with the $L_\infty$-algebra of differential algebras with weight  proposed in \S\mref{ss:difinf}. We also give equivalent descriptions of homotopy   differential algebras  (Propositions~\mref{de:homodifalg3} and \ref{de:homodifalg4}) in \S\mref{ss: another definition}.

\smallskip

\noindent
{\bf Notations.} Throughout this paper, let $\bfk$ be a field of characteristic $0$. Unless otherwise stated,  vector spaces, linear maps and tensor products are taken over $\bfk$.
We will use the homological grading. For two graded spaces $V$ and $W$, let $V\otimes W$ (resp. $\Hom(V, W)$) denote the graded tensor space (resp. the space of graded linear maps).

\section{The $L_\infty$-algebra  controlling deformations of differential algebras}   \mlabel{sec:difinf}

In this section, we first recall the cohomology theory of differential algebras with weight, as developed  in \mcite{GLSZ} with some modification in sign convention. We then introduce an $L_\infty$-algebra such that the cochain complex  of differential algebras can be realized as the shift of the underlying complex of the $L_\infty$-algebra obtained by the twisting procedure. As we will see in \S\mref{ss:modelinf}, this is exactly the $L_\infty$-algebra deduced from the minimal model of the operad governing differential algebras with weight.

\subsection{Differential algebras and their cohomology}\mlabel{sec:cohomologyda}\

Let $A=(A, \mu_A, d_A)$ be a differential algebra of weight $\lambda$, with a multiplication $\mu_A$ and a differential operator $d_A$ of weight $\lambda$, as defined in Eq.~\meqref{eq:diffwt}.
A  bimodule  $M $ over the  associative algebra $(A, \mu_A)$  is called a \name{differential bimodule  over the  differential algebra} $(A, \mu_A,  d_A)$  if $M$ is further endowed with a linear transformation
   $d_M: M\to M$  such that for all $a,b\in A$ and $x\in M,$
   \begin{eqnarray*}
     d_M (ax)&=&d_A(a)x+ad_M (x)+\lambda d_A(a)d_M (x),\\
      d_M (xa)&=&x d_A(a)+d_M (x)a+\lambda d_M (x)d_A(a).
   \end{eqnarray*}
   The regular bimodule $A$ is obviously a differential bimodule over itself, called the \name{regular differential bimodule} of the differential algebra $A$.

Given a differential bimodule   $M =(M,   d_M )$   over the differential algebra $A=(A, \mu_A, d_A)$,  a new differential bimodule structure on $M$, with the same derivation $d_M$, is given by
\begin{equation}\mlabel{eq:newdifbim}
a\vdash  v=(a+\lambda d_A(a))x,\quad x\dashv  a=x(a+\lambda d_A(a)),\quad  a \in A, x\in M.
\end{equation}

For distinction, we let ${}_\vdash M _{\dashv}$ denote this new differential  bimodule structure over  $A=(A, \mu_A, d_A)$; for details, see \cite[Lemma 3.1]{GLSZ}.

Recall that the  Hochschild cochain complex  $\C^{\bullet}_\Alg(A, M )$ of an associative algebra $A$ with coefficients in a bimodule $M$  is the cochain complex
 $$(\C^{\bullet}_\Alg(A, M ):=\bigoplus_{n=0}^\infty \C^n_\Alg(A,M ),\partial_{\Alg}^{\bullet}),$$ where  for $n\geqslant 0$, $\C^n_\Alg(A,M )=\Hom(A^{\otimes n}, M )$ (in particular, $\C^0_\Alg(A,M )=M $) and  the coboundary operator $$\partial_{\Alg}^n: \C^n_\Alg(A, M )\longrightarrow  \C^{n+1}_\Alg(A, M ), n\geqslant 0,$$
\[
\partial_{\Alg}^n(f)(a_{1, n+1}):=(-1)^{n+1} a_1 f(a_{2, n+1})+\sum_{i=1}^n(-1)^{n+1-i}f(a_{1, i-1}\ot a_ia_{i+1}\ot a_{i+2, n+1})
+f(a_{1, n}) a_{n+1}
\]
 for $f\in \C^n_\Alg(A, M ),~a_1,\dots, a_{n+1}\in A$, where for $1\leqslant i\leqslant j\leqslant n+1$, we write
 $a_{i, j}:=a_i\ot \cdots \ot a_j$ (by convention, $a_{i, j}=\emptyset$ for $i>j$).
The corresponding  Hochschild cohomology is denoted by  $\rmH^{\bullet}_\Alg(A, M )$.
We write $\rmH^{\bullet}_\Alg(A):=\rmH^{\bullet}_\Alg(A, M)$ when $M $ is taken to be the regular bimodule $A$,

Let $M =(M,  d_M )$ be a differential bimodule over a differential algebra $A=(A, \mu_A, d_A)$.  The Hochschild cochain complex $\C^{\bullet}_\Alg(A, {}_\vdash M _{\dashv})$ of the associative  algebra $A=(A, \mu_A)$ with coefficients in the new bimodule ${}_\vdash M _{\dashv}$ is called the \name{cochain complex of the differential operator $d_A$} with coefficients in the  differential  bimodule $M $, denoted by  $(\C^{\bullet}_\DO(A, M ), \partial_{\DO}^{\bullet})$.  More precisely, for $g\in \C^n_\DO(A, M )\coloneqq \Hom(A^{\ot n}, M )$ and $a_1,\dots,a_{n+1}\in A$,
we have
\[\begin{split}
\partial_{\DO}^n g(a_{1, n+1})\coloneqq &(-1)^{n+1} a_1 \vdash g(a_{2, n+1})+ \sum_{i=1}^n(-1)^{n+1-i}g(a_{1, i-1}\ot a_ia_{i+1} \ot a_{i+2, n+1})
+ g(a_{1, n})\dashv a_{n+1}\\
=&(-1)^{n+1}(a_1+\lambda d_A(a_1))g(a_{2, n+1})+\sum_{i=1}^n(-1)^{n+1-i}g(a_{1, i-1}\ot a_ia_{i+1} \ot a_{i+2, n+1})\\
&+ g(a_{1, n})  (a_{n+1}+\lambda d_A(a_{n+1})).
\end{split}\]
We write $(\C^{\bullet}_\DO(A), \partial_{\DO}^{\bullet})\coloneqq (\C^{\bullet}_\DO(A, M ), \partial_{\DO}^{\bullet})$ by taking the  differential  bimodule $M $ to be the regular differential bimodule $A$.
 The cohomology of the cochain complex  $(\C^{\bullet}_{\DO}(A, M ), \partial_{\DO}^{\bullet})$, denoted by $\rmH^{\bullet}_\DO(A, M )$, is called the  \name{cohomology of the differential operator  $d_A$} with coefficients in the  differential  bimodule $M$. When the  differential  bimodule $M $ is the regular differential bimodule $A$,    write $\rmH^{\bullet}_\DO(A):=\rmH^{\bullet}_\DO(A, M)$.

\smallskip

As in \cite[Proposition 3.3]{GLSZ}, define a chain map  $\Phi^{\bullet}:  \C^{\bullet}_\Alg(A,M ) \to  \C^{\bullet}_\DO (A, M )$ as follows:
for $f\in \C^n_\Alg(A,M )$ with $  n\geqslant 1$, define
\[\begin{split}
 \Phi^n(f)(a_{1, n})\coloneqq &\sum_{k=1}^n\lambda^{k-1}\sum_{1\leqslant i_1<\cdots<i_k\leqslant n}f(a_{1, i_1-1}\ot d_A(a_{i_1})\ot a_{i_1+1, i_2-1}\ot d_A(a_{i_2})\ot \cdots \ot d_A(a_{i_k}) \ot a_{i_k+1, n})\\
 &-d_M  (f(a_{1, n})),
\end{split}\]
and $$ \Phi^0(x)\coloneqq - d_M  (x), \quad  \ x\in \C^0_\Alg(A,M )=M .$$

Let $(C_{\DA}^{\bullet}(A,M ), \partial_{\DA}^{\bullet})$ be the negative shift of the mapping cone of $\Phi^{\bullet}$. More precisely,
$$
C_{\DA}^n(A,M )\coloneqq
\begin{cases}
\C^n_\Alg(A,M )\oplus \C^{n-1}_\DO(A, M ),&n\geqslant 1,\\
\C^0_\Alg(A,M )=M ,&n=0,
\end{cases}
$$
and the  differentials $\partial_{\DA}^n: C_{\DA}^n(A,M )\rar C_{\DA}^{n+1}(A,M )$ are defined by
$$
\partial_{\DA}^n(f,g) :=  ( \partial^n_{\Alg} (f), -\partial_\DO^{n-1} (g)- \Phi^n(f)),  \quad  f\in \C^n_\Alg(A, M ),\,g\in \C^{n-1}_\DO(A,  M ), n\geqslant 1,$$
$$\partial_{\DA}^0 (x) := ( \partial^0_{\Alg} (x), -\Phi^0( x)), \quad
 x\in \C^0_\Alg(A,M )=M. $$
  The complex $(C_{\DA}^{\bullet}(A,M ), \partial_{\DA}^{\bullet})$ is called the \name{cochain complex of the differential algebra $A=(A, \mu_A, d_A)$} with coefficients in  the differential bimodule $M$. When the  differential  bimodule $M $ is the regular differential bimodule $A$, we write $(\C^{\bullet}_\DA(A), \partial_{\DA}^{\bullet}):=(\C^{\bullet}_\DA(A, M ), \partial_{\DA}^{\bullet})$.

The cohomology of the cochain complex $(C_{\DA}^{\bullet}(A, M ),\partial_{\DA}^{\bullet})$, denoted by $\rmH_{\DA}^{\bullet}(A, M )$, is called the  \name{cohomology of the differential algebra}  $A=(A, \mu_A, d_A)$ with coefficients in the differential  bimodule $M =(M,  d_M )$.

It is shown in~\mcite{GLSZ} that this cohomology theory controls formal deformations as well as abelian extensions of differential algebras.
We will further clarify this connection theoretically by using the $L_\infty$-algebra structure on the deformation complex in \S\mref{ss:difinf}; see Propositions~\mref{pp:infcohcomp} and \mref{pp:difopcochain}.

\subsection{ $L_\infty$-algebra on the deformation complex from differential algebras}\
\mlabel{ss:difinf}

We first recall some basics on $L_\infty$-algebras; for more details, see \mcite{Get09,LS93,LM,Sta92}.

Let $V=\oplus_{n\in \mathbb{Z}} V_n$ be a graded vector space. The graded symmetric algebra $S(V)$ of $V$ is defined to be the quotient of the tensor algebra $T(V)$ by   the two-sided ideal $I$   generated by
$x\ot y -(-1)^{|x||y|}y\ot x$ for all homogeneous elements $x, y\in V$. For $x_1\ot\cdots\ot x_n\in V^{\ot n}\subseteq T(V)$, write $ x_1\odot x_2\odot\dots\odot x_n$ for its image in $S(V)$.
For homogeneous elements $x_1,\dots,x_n \in V$ and $\sigma\in \rmS_n$, the Koszul sign $\epsilon(\sigma;  x_1,\dots, x_n)$ is defined by the equality
$$x_1\odot x_2\odot\dots\odot x_n=\epsilon(\sigma;  x_1,\dots,x_n)x_{\sigma(1)}\odot x_{\sigma(2)}\odot\dots\odot x_{\sigma(n)}\in S(V).$$
We also define
$$ \chi(\sigma;  x_1,\dots,x_n)= \sgn(\sigma)\ \epsilon(\sigma;  x_1,\dots,x_n),$$
where $\sgn(\sigma)$ is the sign of the permutation $\sigma$.

\begin{defn}\mlabel{Def:L-infty}
An \name{$L_\infty$-algebra} is a graded space $L=\bigoplus\limits_{i\in\mathbb{Z}}L_i$ endowed with a family of linear operators $l_n:L^{\ot n}\rightarrow L, n\geqslant 1$ with  degree $|l_n|=n-2$ subject to  the following conditions. For $n\geqslant 1$  and homogeneous elements $x_1,\dots,x_n\in L$,
\begin{enumerate}
\item (generalized anti-symmetry)  $$l_n(x_{\sigma(1)}\ot \cdots \ot x_{\sigma(n)})=\chi(\sigma; x_1,\dots,   x_n)\ l_n(x_1 \ot \cdots  \ot x_n), \  \sigma\in \rmS_n;$$
		
\item (generalized Jacobi identity)
$$\sum\limits_{i=1}^n\sum\limits_{\sigma\in \Sh(i,n-i)}\chi(\sigma; x_1,\dots,x_n)(-1)^{i(n-i)}l_{n-i+1}(l_i(x_{\sigma(1)}\ot \cdots \ot x_{\sigma(i)})\ot x_{\sigma(i+1)}\ot \cdots \ot x_{\sigma(n)})=0.$$
Here $\Sh(i,n-i)$ is the set of $(i,n-i)$
		shuffles, consisting of $\sigma\in \rmS_n$ such that
$$\sigma(1)<\cdots<\sigma(i)\ \  \mathrm{and}\ \  \sigma(i+1)<\cdots<\sigma(n).$$
\end{enumerate}
\end{defn}

In particular,  if $l_n=0$ for all  $n\geqslant 3$, then $(L,l_1,l_2)$ is called a \name{differential graded (=dg) Lie algebra}; if $l_n=0$ except $n=2$, then $(L, l_2)$ is called a \name{graded Lie algebra}.

Let $V=\bigoplus_{i\in \mathbb{Z}}V_i$ be a graded space. Let $sV$ denote the suspension of $V$ and $s^{-1}V$ the  desuspension.
Denote
$$ \mathfrak{C}_{\Alg}(V):=\Hom(T(sV), sV).$$  Given  homogeneous elements  $f\in \Hom((sV)^{\ot m},V)$ with $m\geqslant 1$   and $ g \in \Hom((sV)^{\ot n},V)$ with  $n\geqslant 0$, for  each $1\leqslant i\leqslant m$, write
$$sf\circ_i sg:=sf\circ (\Id^{\otimes (i-1)}\ot sg \ot \Id^{\otimes (m-i)}),$$
and  $sf\{sg\}:=\sum_{i=1}^m sf\circ_i sg$; when $m=0$, $sf\{sg\}$ is defined to be $0$. The  Gerstenhaber bracket~\mcite{Ge1} of $sf$ and $sg$
  is  defined by
\begin{eqnarray}\mlabel{eq:gers} [sf,sg]_G:=sf\{sg\}-(-1)^{(|f|+1)(|g|+1)}sg\{sf\}.\end{eqnarray}
 A celebrated result of Gerstenhaber  \mcite{Ge1} states that
 the graded space $\mathfrak{C}_{\Alg}(V)$ endowed with the Gerstenhaber bracket is  a graded Lie algebra.
 The following facts are quoted for later applications.

\begin{enumerate}
\item
Let $V$ be an ungraded space considered as a graded space concentrated in degree 0. Then there is a bijection between the set of Maurer-Cartan elements in the graded Lie algebra $\mathfrak{C}_{\Alg}(V)$ and the set of associative algebra structures on the space $V$, where the correspondence is  induced  by
\begin{eqnarray}\mlabel{eq:iso1} \Hom((sV)^{\ot n}, sV) \simeq \Hom(V^{\ot n}, V),  f\mapsto \tilde{f}:=    s^{-1}\circ  f \circ s^{\ot n}, \quad f\in \Hom((sV)^{\ot n},  sV).
\end{eqnarray}

\item Let $(A, \mu)$ be an associative algebra and $\alpha$ be the corresponding Maurer-Cartan element in the graded Lie algebra $\mathfrak{C}_{\Alg}(A)$. Then the underlying complex of the twisted dg Lie algebra $(\mathfrak{C}_{\Alg}(A), l_1^\alpha, l_2^\alpha)$ is  exactly    $s\C^{\bullet}_\Alg(A)$, the shift of the Hochschild cochain complex of  the  associative algebra $A$. 	
\end{enumerate}

Let $V$ be a graded space. Write
\[  \mathfrak{C}_{\DO}(V):=\Hom(T(sV),V) \]  and define the  graded space
\begin{equation}\mlabel{eq:linfdiff}
\mathfrak{C}_{\DA}(V):=\mathfrak{C}_{\Alg}(V)\oplus \mathfrak{C}_{\DO}(V)=\Hom(T(sV),sV)\oplus \Hom(T(sV),V).
\end{equation}

Now  we introduce  an $L_\infty$-algebra structure on  the  graded space $\mathfrak{C}_{\DA} (V)$ by the following process:
\begin{enumerate}
	\item For $sf, sg \in \mathfrak{C}_{\Alg}(V)$, define
	\[l_2(sf\ot sg)=[sf,sg]_G\in \mathfrak{C}_{\Alg}(V),\]
where $[-,-]_G$ is the Gerstenhaber Lie bracket in Eq.~\meqref{eq:gers}.
\item Let $sf \in  \mathfrak{C}_{\Alg}(V)$ and $g\in \mathfrak{C}_{\DO}(V)$. Define
	$$l_2(sf\ot g)=(-1)^{|f|+1}s^{-1}[sf,sg]_G\in \mathfrak{C}_{\DO}(V).$$
	\item Let $sf\in \Hom((sV)^{\ot n}, sV)\subseteq\mathfrak{C}_{\Alg}(V)$ and $g_1,\dots,g_m\in \mathfrak{C}_{\DO}(V)$ with $2\leqslant m\leqslant n$. Define
	\[l_{m+1}(sf\ot g_1\ot \cdots \ot g_m)=\lambda^{m-1}\sum_{\sigma\in \rmS_m}(-1)^\xi s^{-1}\big(sf\{sg_{\sigma(1)},\dots,sg_{\sigma(m)}\}\big)\in \mathfrak{C}_{\DO}(V),\]
	where $(-1)^\xi=\chi(\sigma;g_1,\dots,g_m)(-1)^{m(|f|+1)+\sum\limits_{k=1}^{m-1}\sum\limits_{j=1}^k|g_{\sigma(j)}|}$.
\item Let $sf\in  \mathfrak{C}_{\Alg}(V)$ and $g_1,\dots,g_m\in \mathfrak{C}_{\DO}(V)$ with $  m \geqslant 1$. For $1\leqslant k\leqslant m$, define
	\[l_{m+1}(g_1\ot \cdots\ot g_k \ot sf\ot g_{k+1}\ot \cdots \ot g_m )\in\mathfrak{C}_{\DO}(V)\]
	to be
	\[l_{m+1}(g_1\ot \cdots\ot g_k \ot sf\ot g_{k+1}\ot \cdots \ot g_m ):=(-1)^{(|f|+1)(\sum\limits_{j=1}^k|g_j|)+k}l_{n+1}(sf\ot g_1\ot \cdots \ot g_m),\]
	where the right hand side has been defined in (ii) and (iii).
	\item All other components of the  operators $\{l_n\}_{n\geqslant 1}$ vanish.
\end{enumerate}

Then we can state one of our main results in this paper. \begin{thm}\mlabel{th:linfdiff}
	Given a graded space $V$ and $\lambda\in \bfk$, the graded space $\mathfrak{C}_{\DA}(V)$ endowed with the operations $\{l_n\}_{n\geqslant 1}$ defined above is an $L_\infty$-algebra.
\end{thm}
The theorem can be proved by a direct but lengthy verification. Instead, we  will present an operadic proof in the spirit of~\mcite{KS00,VdL02,VdL03} once we obtain the minimal model. We first sketch the proof below, with details supplied later. We then give an application in the next subsection, realizing differential algebras with weight as the Maurer-Cartan elements in this $L_\infty$-algebra.

\begin{proof} (A sketch)
The minimal model of the differential algebra operad $\Dif$ is obtained in Theorem~\mref{thm:difmodel}. It is shown  in \S\mref{ss:modelinf} that  $\mathfrak{C}_{\DA}(V)$ is exactly the $L_\infty$-structure resulting from  this minimal model.
\end{proof}

\subsection{Realising differential algebra structures as Maurer-Cartan elements}\
\mlabel{ss:mce}

We give an application of Theorem~\mref{th:linfdiff}. \begin{defn}\mlabel{de:linfmc}
	An element $\alpha$ of degree $-1$ in an  $L_\infty$-algebra $(L,\{l_n\}_{n\geqslant1})$   is called a \name{Maurer-Cartan element} if it satisfies the \name{Maurer-Cartan equation}:
	\begin{eqnarray}\mlabel{eq:mce}\sum_{n=1}^\infty\frac{1}{n!}(-1)^{\frac{n(n-1)}{2}} l_n(\alpha^{\ot n})=0\end{eqnarray}
(in particular, this infinite sum needs to be well defined).
\end{defn}

For a dg Lie algebra $(L,l_1,l_2)$, Eq.~\meqref{eq:mce} reduces to
\begin{eqnarray}\mlabel{eq:dglamc} l_1(\alpha)-\frac{1}{2}l_2(\alpha\ot \alpha)=0.\end{eqnarray}

\begin{prop}[twisting procedure\cite{Get09}]\mlabel{pp:deflinf}
	Let $\alpha$ be a Maurer-Cartan element in an  $L_\infty$-algebra $(L,\{l_n\}_{n\geqslant1})$. A new $L_\infty$-structure $\{l_n^\alpha: L^{\ot n}\rightarrow L\}_{n\geqslant 1}$ on the graded space $L$ is defined by
	\begin{equation}\mlabel{eq:twlinf}
			l^\alpha_n(x_1\ot \cdots\ot x_n):=\sum_{i=0}^\infty\frac{1}{i!}(-1)^{in+\frac{i(i-1)}{2}}l_{n+i}(\alpha^{\ot i}\ot x_1\ot \cdots\ot x_n),\  x_1, \dots, x_n\in L,
	\end{equation}
provided the infinite sums exist. The resulting $L_\infty$-algebra $(L, \{l_n^\alpha\}_{n\geqslant 1})$ is called the \name{twisted $L_\infty$-algebra} by the Maurer-Cartan element $\alpha$.
\end{prop}

For a dg Lie algebra $(L,l_1,l_2)$   and a Maurer-Cartan element $\alpha\in L_{-1}$,  Eq.~\meqref{eq:twlinf}  reduces to
\begin{eqnarray}\mlabel{Eq: twist dgla} l_1^\alpha(x) = l_1(x)-l_2(\alpha\ot x), \  x\in L,  \ \mathrm{and}\  l_2^\alpha = l_2.	
\end{eqnarray}

We fix the isomorphism
 \begin{eqnarray}\mlabel{eq:iso2}
 	\Hom((sV)^{\ot n},  V) \simeq \Hom(V^{\ot n}, V), g\mapsto \check{g}:=     g \circ  s^{\ot n}, \,  g\in \Hom((sV)^{\ot n},  V).
\end{eqnarray}
The following proposition follows from general facts about $L_\infty$-structures obtained from minimal models; see the discussion at the end of \S\mref{ss:homocood}.

\begin{prop}\mlabel{pp:damc} Let $V$ be an ungraded space considered as a graded space concentrated in degree $0$. Then a differential algebra structure of weight $\lambda$ on $V $ is equivalent to a  Maurer-Cartan element in the $L_\infty$-algebra $\mathfrak{C}_{\DA}(V)$ in Theorem~\mref{th:linfdiff}.
\end{prop}
The following result justifies the cohomology theory introduced in \mcite{GLSZ}  and recalled in \S\mref{sec:cohomologyda}.
 \begin{prop}\mlabel{pp:infcohcomp}
	Let $(A,\mu, d)$ be a differential  algebra of weight $\lambda$. Twisting the $L_\infty$-algebra $\mathfrak{C}_{\DA}(A)$ by the Maurer-Cartan element corresponding to the differential  algebra structure $(A,\mu,d)$, then its  underlying complex is exactly $s\C^\bullet_{\DA}(A)$, the suspension of the cochain complex of the differential  algebra $(A, \mu, d)$.
\end{prop}

\begin{proof}
	By Proposition~\mref{pp:damc}, the differential  algebra structure $(A,\mu, d)$ is equivalent to a Maurer-Cartan element $\alpha=(m,\tau)$ in the $L_\infty$-algebra $\mathfrak{C}_{\DA}(A)$ with
$$m=-s\circ\mu\circ (s^{-1})^{\ot 2}: (sV)^{\ot 2}\rightarrow sV  \ \mathrm{and}\ \tau=d\circ s^{-1}: sV\rightarrow V.$$
 By Proposition~\mref{pp:deflinf}, the Maurer-Cartan element induces a new $L_\infty$-algebra structure $\{l_n^\alpha\}_{n\geqslant 1}$ on the graded space $\mathfrak{C}_{\DA}(A)$.

Given $(sf, g)\in\Hom((sA)^{\ot n}, sA)\oplus \Hom((sA)^{\ot n-1}, sA) \subset \mathfrak{C}_{\Alg}(A)\oplus \mathfrak{C}_{\DO}(A)=\mathfrak{C}_{\DA}(A)$, note that by the isomorphisms  \meqref{eq:iso1}  and \meqref{eq:iso2},
we get $(\widetilde{sf}, \check{g})\in \C_{\DA}^{n}(A)$.
By definition, for  $sf\in\Hom((sA)^{\ot n}, sA)\subset \mathfrak{C}_{\Alg}(A)$, we have
 {\small$$l_1^\alpha(sf)=\sum_{i=0}^\infty\frac{1}{i!}(-1)^{i+\frac{i(i-1)}{2}}l_{i+1}(\alpha^{\ot i}\ot sf)
=\Big(-l_2(m\ot sf),\  \sum_{i=1}^n\frac{1}{i!}(-1)^{ \frac{i(i+1)}{2}}l_{i+1}(\tau^{\ot i}\ot sf)\Big).	
$$}
The definition of $\{l_n\}_{n\geqslant 1}$ on $\mathfrak{C}_{\DA}(A)$ gives $-l_2(m\ot sf)=-[m,sf]_G$.
On the other hand, we have {\small\begin{align*}
		& \sum_{i=1}^n\frac{1}{i!}(-1)^{ \frac{i(i+1)}{2}}l_{i+1}(\tau^{\ot i}\ot sf)\\
		=& -l_{2}(\tau \ot sf)+\sum_{i=2}^{n}\frac{1}{i!}(-1)^{\frac{i(i+1)}{2}}l_{i+1}(\tau^{\ot i}\ot sf)\\
		\stackrel{\rm (iv)}{=}&-(-1)^{|f|} l_{2}(sf\ot \tau)+\sum_{i=2}^{n}\frac{1}{i!}(-1)^{\frac{i(i+1)}{2}} (-1)^{i|f| }l_{i+1}(sf\ot \tau^{\ot i}  )\\
		\stackrel{\rm (ii)(iii)  }{=} & s^{-1}[sf, s\tau]_G+\sum_{i=2}^n \frac{1}{i!}(-1)^{\frac{i(i+1)}{2}}(-1)^{i|f| } i! \lambda^{i-1} (-1)^{i(|f|+1)+\frac{i(i-1)}{2}}
s^{-1}(sf\{\underbrace{s\tau,\dots,s\tau}_{i}\})\\
		 = &s^{-1}[sf, s\tau]_G+ \sum_{i=2}^n   \lambda^{i-1}
s^{-1}(sf\{\underbrace{s\tau,\dots,s\tau}_{i}\}).
	\end{align*}}
	For  $g\in \Hom((sA)^{\ot (n-1)},A)\subset \mathfrak{C}_{\DO}(A)$, we have
	{\small\begin{align*}
		l_1^{\alpha}(g)=&\sum_{i=0}^\infty\frac{1}{i!}(-1)^{\frac{i(i+1)}{2}}l_{i+1}(\alpha^{\ot i}\ot  g)\\
		=&-l_2(m\ot g)-\frac{1}{2!}\Big(l_3(m\ot \tau\ot g)+l_3(\tau\ot m\ot g)\Big)\\
		=& s^{-1}[m, sg]+\lambda s^{-1} m\circ (s\tau \ot  sg)+\lambda s^{-1} m\circ (sg \ot  s\tau). \end{align*}}
We obtain
 {\small$$\begin{array}{rcl} l_1^\alpha (sf, g)
&=&\Big(-[m,sf]_G, s^{-1}[sf, s\tau]_G+ \sum_{i=2}^n   \lambda^{i-1}
s^{-1}(sf\{\underbrace{s\tau,\dots,s\tau}_{i}\})\\
& &  \quad +s^{-1}[m, sg]+\lambda s^{-1} m\circ (s\tau \ot  sg)+\lambda s^{-1} m\circ (sg \ot  s\tau)\Big),\end{array}$$}
	which is easily seen to be in
	 correspondence with  $- \partial_{\DA}^n(\widetilde{sf}, \check{g})$ via the  fixed isomorphisms \meqref{eq:iso1} and  \meqref{eq:iso2}. This completes the proof.
\end{proof}

Although $\mathfrak{C}_{\DA}(A)$ is an $L_\infty$-algebra,   once the associative algebra structure $\mu$ over $A$ is fixed, the graded space  $\mathfrak{C}_{\DO}(A)$ is indeed a differential graded Lie algebra, as shown by the following result. Moreover, this result justifies the cohomology theory of differential operators.

\begin{prop}\mlabel{pp:difopcochain}
	Let $(A,\mu)$ be an associative algebra. Then the graded space  $\mathfrak{C}_{\DO}(A)$ can be endowed with a dg Lie algebra structure, of which the Maurer-Cartan elements are in bijection with the differential operators   of weight $\lambda$ on $(A,\mu)$.

Given   a differential operator   $d$  of weight $\lambda$  on  the associative algebra $(A, \mu)$, the  dg Lie algebra $\mathfrak{C}_{\DO}(A)$ twisted by the Maurer-Cartan element corresponding to $d$ has its underlying complex being the cochain complex $\C^{\bullet}_\DO(A)$ of the differential operator $d$.
\end{prop}
\begin{proof} Regard $A$ as a graded space concentrated in degree 0. Define $m=- s\circ \mu\circ (s^{-1}\ot s^{-1}): (sA)^{\ot 2}\rightarrow sA$. Then it is easy to see that   $\beta=(m,0)$ is naturally a Maurer-Cartan element in the $L_\infty$-algebra $\mathfrak{C}_{\DA}(A)$.  We apply the twisting procedure by $\beta$ on $\mathfrak{C}_{\DA}(A)$.
	By the construction of   $\{l_n\}_{\geqslant 1}$ on $\mathfrak{C}_{\DA}(A)$,    the graded subspace $\mathfrak{C}_{\DO}(A)$ is closed under the action of the operators $\{l_n^\beta\}_{n\geqslant 1}$.

 Since the arity of $m$ is 2,  the restriction of $l_n^\beta$ on $\mathfrak{C}_{\DO}(A)$ is $0$ for $n\geqslant 3$. Thus $(\mathfrak{C}_{\DO},\{l_n^\beta\}_{n=1,2})$ forms a dg Lie algebra which can be made explicit as follows. For $g\in \Hom((sA)^{\ot n},A)$ and $h\in \Hom((sA)^{\ot k},A)$, we have
	{\small\begin{align*}
		l_1^\beta(g)=&-l_2(m\ot g)=  s^{-1} [m, sg]_G,\\
		l_2^\beta(g\ot h)=&l_3(m\ot g\ot h)\\
		=& \lambda (-1)^{|g|} s^{-1} m\circ (sg\ot sh)+\lambda (-1)^{|g||h|+1+|h|} s^{-1} m\circ (sh\ot sg)\\
=& \lambda (-1)^{n} s^{-1} m\circ (sg\ot sh)+\lambda (-1)^{nk+k+1} s^{-1} m\circ (sh\ot sg).
	\end{align*}}
	
Since $A$ is concentrated in degree 0, we have $\mathfrak{C}_{\DO}(A)_{-1}=\Hom(sA,A)$. Take an element  $\tau\in \Hom(sA,A)$.  Then $\tau$ satisfies the  Maurer-Cartan equation:
$$l_1^\beta(\tau)-\frac{1}{2}l_2^{\beta}(\tau\ot \tau)=0$$
if and only if
$$   s^{-1}[m, s\tau]_G+\lambda\   s^{-1} m\circ (s\tau\ot s\tau)=0.$$
Define $d=\tau\circ s:A\rightarrow A$. Under the isomorphisms \meqref{eq:iso1} and   \meqref{eq:iso2},  the above equation is exactly the statement  that $d$ is a differential operator of weight $\lambda$ on the associative algebra $(A,\mu)$.

Now let $d$ be a differential operator of weight $\lambda$ on the associative algebra $(A,\mu)$ and $\tau=d\circ s^{-1}$ be the corresponding Maurer-Cartan element in the dg Lie algebra $(\mathfrak{C}_{\DO},\{l_n^\beta\}_{n=1,2})$. For $g\in \Hom((sA)^{\ot n},A)$, a direct computation shows that
$\partial_\DO(\check{g})$ corresponds to $(l_1^\beta)^\tau(g)$ via the isomorphism \meqref{eq:iso2}. This proves the last statement.
\end{proof}

\begin{remark} In the course of proving the above proposition, the Lie bracket $(l_2^\beta)^\tau$ on  $\C^{\bullet}_\DO(A)$ can also be given explicitly:
    for $f\in \C^n_\DO(A)$ and $g\in \C^k_\DO(A)$,
$$[f, g](a_{1, n+k})=  (-1)^n \lambda f(a_{1, n})g(a_{n+1, n+k})+ (-1)^{nk+k+1} \lambda g(a_{1, k})f(a_{k+1, n+k}),\quad a_1, \dots, a_{n+k}\in A.$$
  Note that when $\lambda=0$, the cochain complex $\C^{\bullet}_\DO(A)$ has the trivial Lie bracket.
\end{remark}

\section{The Koszul dual homotopy cooperad}\mlabel{sec:dualhomocoop}

In this section, we will  construct a homotopy cooperad $ \Dif^\ac$, which will be called the Koszul dual homotopy cooperad of the operad $\Dif$ of differential algebras with weight $\lambda$. In fact,   we will show that  the  cobar construction of $ \Dif^\ac$  is  the minimal model of $\Dif$  (Theorem~\mref{thm:difmodel}) which justifies the name.

\subsection{Homotopy  (co)operads and $L_\infty$-algebras}\mlabel{ss:homocood}\

We first collect needed background on nonsymmetric homotopy  (co)operads that are scattered in several references \mcite{DP16, Mar96, MV09a, MV09b}. We  also  explain how $L_\infty$-structures can be obtained from homotopy operads, in particular from convolution homotopy operads.

A graded collection is a family $\calo=\{\calo(n)\}_{n\geqslant 1}$ of graded space indexed by positive integers. So each $\calo(n)$ itself is a graded space. Elements in $\calo(n)$ are said to have arity $n$. The suspension of $\calo$, denoted by $s\calo$ is defined to be the graded collection $\{s\calo(n)\}_{n\geqslant 1}$. The desuspension $s^{-1}\calo$ of the graded collection $\calo$ is defined in the same way.

We also recall preliminaries about trees. As we  only consider planar reduced rooted trees, we will simply call them trees.
For a tree $T\in \mathfrak{T}$,  denote the weight (=number of vertices)  and arity (=number of leaves)  of $T$ by $\omega(T)$ and $\alpha(T)$  respectively. Let $\mathfrak{T}$ be the set of all   trees with weight $\geqslant 1$ and, for  $n\geqslant 1$, let $\mathfrak{T}^{(n)}$ be the set of trees of weight $n$.
Since the trees are planar, the inputs of each vertex in a tree have total order by going clockwise.
Furthermore by the existence of the root, there is a naturally induced total order on the set of all    vertices of the tree $T$, given by counting the   vertices  starting from the root clockwise along the tree. We call this order the \name{planar order}.

Let $T'$ be a divisor of $T$. Let $T/T'$ denote the tree obtained from $T$ in replacing $T'$ by a corolla (a tree with only one vertex)  of arity $\alpha(T')$. There is a natural permutation $\sigma=\sigma(T, T')\in \rmS_{\omega(T)}$ associated to the pair  $(T, T')$ defined as follows.  Let $\{v_1<\dots<v_n\}$ be the sequence of all vertices of $T$ listed in the planar order and let $\omega(T')=j$. Let $v'$ be the vertex in $T/T'$ corresponding to the divisor $T'$ of $T$ and let $i$ be  the serial number of $v'$ in $T/T'$ in the planar order, that is, there are $i-1$ vertices before $T'$. Then define $\sigma=\sigma(T, T')\in \rmS_{n}$ to be the unique permutation with the properties that it does not permute the vertices $v_1, \dots, v_{i-1}$, the  ordered set $\{v_{\sigma(i)}<\dots<v_{\sigma(i+j-1)}\}$ is exactly the planar ordered set of all vertices  of $T'$, and the ordered set $\{v_{1}<\dots<v_{i-1}<v'<v_{\sigma(i+j)}<\dots< v_{\sigma(n)}\}$ is exactly the  planar ordered   set of all vertices  in the tree $T/T'$.

Let $\calp=\{\calp(n)\}_{n\geqslant 1}$ be a graded collection, let  $T\in \mathfrak{T}$ and let $\{v_1<\dots<v_n\}$ be the sequence of all vertices  of $T$ in the planar order. Let $\calp^{\otimes T}$ denote   $\calp(\alpha(v_1))\ot \cdots \ot \calp(\alpha(v_n))$. Its element can be pictured as a linear combination of  the tree $T$ with each vertex  $v_i$ decorated by an element  of  $\calp(\alpha(v_i))$.

\begin{defn} \mlabel{de:homood}
A \name{homotopy operad structure} on a graded collection $\calp=\{\calp(n)\}_{n\geqslant 1}$ is a family of operations $$\{m_T: \calp^{\ot T}\rightarrow \calp(\alpha(T))\}_{T\in \mathfrak{T}}$$ with $|m_T|=\omega(T)-2$, such that, for all $T\in \mathfrak{T}$, the equation	
\[\sum\limits_{T'\subset{T}}(-1)^{i-1+jk}\sgn(\sigma(T,T'))\ m_{T/T'}\circ(\id^{\ot {i-1}}\ot m_{T'}\ot \id^{\ot k})\circ r_{\sigma(T,T')}=0\]
holds. Here $T'$ runs through subtrees of $T$,  $i$ is the serial number of the vertex $v'$ in $T/T'$ as above, $j=\omega(T')$, $k=\omega(T)-i-j$, and
	  $r_\sigma$ denotes the right action by $\sigma=\sigma(T,T')$:
$$r_\sigma(x_1\ot \cdots\ot x_n)=\varepsilon(\sigma; x_1,\dots, x_{n})x_{\sigma(1)}\ot \cdots  \ot x_{\sigma(n)}.$$	
\end{defn}

A \name{strict morphism} between two homotopy operads is a morphism of graded collections that is compatible with all operations $m_T, T\in \mathfrak T$.

 Let $\cali$ be the collection with $\cali(1)=k$ and $\cali(n)=0$ for $n\ne1$. The collection $\cali$ can be naturally endowed with a homotopy operad structure: when $T$ is the tree with two vertices and a unique leaf, then $m_T: \cali(1)\ot \cali(1)\to \cali(1)$ is given by the identity; for other trees $T$, $m_T$ vanish.

A homotopy operad $\calp$ is called \name{strictly unital} if there exists a strict morphism of homotopy operads $\eta: \cali\rightarrow \calp$ such that, for $n\geqslant 1$,  the compositions  $$\calp(n)\cong \calp(n)\ot \cali(1)\xrightarrow{\id\ot \eta}\calp(n)\ot \calp(1)\xrightarrow{m_{T_{1, i}}}\calp(n)$$ and
$$\calp(n)\cong \cali(1)\ot \calp(n)\xrightarrow{\eta\ot \id}\calp(1)\ot \calp(n)\xrightarrow{m_{T_2}} \calp(n)$$ are identity maps on $\calp(n)$, where for $1\leqslant i\leqslant n$, $T_{1,i}$ is the tree of weight $2$, arity $n$, with its second vertex having arity $1$ and connecting to the first vertex on its $i$-th leaf; and $T_2$ is the tree of weight $2$, arity $n$ whose first vertex has arity $1$. Moreover,
for a tree $T$ of weight at least three, $m_T(\Id^{\ot i-1}\ot \eta\ot \Id^{\ot \omega(T)-i})=0$ for each $1\leqslant i\leqslant \omega(T)$.

Also a homotopy operad $\calp$ is called \name{augmented} if there is a strict morphism of homotopy operads $\varepsilon: \calp\rightarrow \cali$ such that $\varepsilon \circ \eta=\id_{\cali}$.

Given a dg operad $\calp$, for each tree $T$, one can define  the composition $m_\calp^T:\calp^{\ot T}\to \calp(\alpha(T))$ in $\calp$ along  $T$ as follows:  for $\omega(T)=1$, set $m_\calp^T:=\Id$; for  $\omega(T)=2$, set $m_\calp^T:=m_T$;
 for $\omega(T)\geqslant 3$, write $T$ as the grafting  of a subtree $T'$,   whose vertex set is that of $T$ except the last one, with the corolla whose unique vertex is just the last vertex of $T$ in the planar order, then set
 $m_\calp^T:=m_{T/T'}\circ (m_\calp^{T'}\ot \Id) $, where $m_\calp^{T'}$ is obtained by induction.

Dualizing the definition of homotopy operads, one has the notion of homotopy cooperads.

\begin{defn}Let $\calc=\{\calc(n)\}_{n\geqslant 1}$ be a graded collection. A \name{homotopy cooperad structure} on $\calc$ consists of a family of operations $\{\Delta_T: \calc(\alpha(T))\rightarrow \calc^{\ot T }\}_{T\in\mathfrak{T}}$ with $|\Delta_T|=\omega(T)-2$ such that for $c\in \calc$, $\Delta_T(c)=0$ for all but finitely many $T\in\frakt$, and  the family of operations $\{\Delta_T\}_{T\in\frakt}$ satisfies the identity
$$\sum_{T'\subset T}\sgn(\sigma(T,T')^{-1})(-1)^{i-1+jk} r_{\sigma(T,T')^{-1}}\circ (\id^{\ot i-1}\ot \Delta_{T'}\ot \id^{\ot k})\circ  \Delta_{T/T'}=0$$
for all $T\in\frakt$, where  $T', i, j, k$ have the same meaning as for  homotopy operads.
\end{defn}

The graded collection   $\cali$ has a natural  homotopy cooperad structure. It is obtained by taking $\Delta_T: \cali(1)\to  \cali(1)\ot \cali(1)$ to be the identity when $T$ is the tree with two vertices and one unique leaf, and $\Delta_T=0$ for other $T$.
Dualizing the notions of being strictly unital and being augmented gives the notions of being \name{strictly counital and coaugmented}.
For a coaugmented homotopy cooperad $\calc$, the graded collection $\overline{\calc}=\ker(\varepsilon)$ endowed with operations $\{\overline{\Delta}_T\}_{T\in\frakt}$ is naturally a homotopy cooperad, where $\overline{\Delta}_T$ is the  the restriction of operation $\Delta_T$ on $\overline{\calc}$.

A homotopy cooperad  $\cale=\{\cale(n)\}_{n\geqslant 1}$  such that   $\{\Delta_T\}$ vanish for all $\omega(T)\geqslant 3$ is  exactly a noncounital dg cooperad in the sense of Markl \mcite{Mar08}.

For a (noncounital) dg cooperad $\cale$, one can define  the cocomposition $\Delta^{T}_\cale:\cale(\alpha(T))\to \cale^{\ot T}$ along a tree $T$ in duality of the composition $m^T_\calp$ along $T$ for a dg operad $\calp$.

\begin{prop-def} Let $\calc$ be a homotopy cooperad and $\cale$ be a dg cooperad. Then the graded collection $\calc\ot \cale$ with $(\calc\ot\cale)(n):=\calc(n)\ot \cale(n), n\geqslant 1$, has a natural structure of homotopy cooperad as follows:
	\begin{enumerate}
		\item For a tree $T\in \frakt$ of weight $1$ and arity $n$, define  $$\Delta_T^H(c\ot e):=\Delta_T^\calc(c)\ot e+(-1)^{|c|}c\ot d_{\cale}e$$ for homogeneous elements $c\in \calc(n), e\in \cale(n)$;
		\item For a tree $T$ of weight $n\geqslant 2$, define
		$$\Delta_T^H(c\ot e):=(-1)^{\sum\limits_{k=1}^{n-1}\sum_{j=k+1}^n|e_k||c_j|}(c_1\ot e_1)\ot \cdots \ot (c_n\ot e_n)\in (\calc\ot \cale )^{\ot T},$$
		with $c_1\ot \cdots\ot c_n=\Delta_T^\calc(c)\in \calc^{\ot T}$  and $e_1\ot \cdots \ot e_n=\Delta^T_\cale(e)\in \cale^{\ot T}$, where $\Delta^T_\cale$ is the cocomposition  in $\cale$  along   $T$.
	\end{enumerate}
The new homotopy cooperad  is called the \name{Hadamard product} of $\calc$ and $\cale$,   denoted by $\calc\ot_{\rmH}\cale$.
	\end{prop-def}

Define $\cals:=\mathrm{End}_{\bfk s}^c$ to be the graded cooperad whose underlying graded collection is given by $\cals(n)=\Hom((\bfk s)^{\ot n},\bfk s),n\geqslant 1$. Denote $\delta_n$ to be the map in $\cals(n)$ sending $s^{\ot n}$ to $s$. The cooperad structure is given by
$$\Delta_T(\delta_n):=(-1)^{(i-1)(j-1)}\delta_{n-j+1}\ot \delta_j\in \cals^{\ot T} $$
 for a tree $T$ of weight $2$ whose second vertex is connected with the $i$-th  leaf of its first vertex. We also define $\cals^{-1}$ to be the graded cooperad whose underlying graded collection is $\cals^{-1}(n)=\Hom((\bfk s^{-1})^{\ot n},s^{-1})$ for all $n\geqslant 1$ and the cooperad structure is given by
 $$\Delta_T(\varepsilon_n):=(-1)^{(j-1)(n-i+1-j)}\varepsilon_{n-j+1}\ot \varepsilon_j\in (\cals^{-1})^{\ot T},$$
 where $\varepsilon_n\in \cals^{-1}(n)$ is the map which takes $(s^{-1})^{\ot n}$ to $s^{-1}$ and $T$ is the same as before.
 It is easy to see that $\cals\ot_{\mathrm{H}}\cals^{-1}\cong \cals^{-1}\ot_{\mathrm{H}}\cals=:\mathbf{As}^\vee$. Notice that for a homotopy cooperad $\calc$, we have $ \calc\ot_\rmH \mathbf{As}^\vee\cong\calc\cong \mathbf{As}^\vee\ot_\rmH \calc.$

Let $\calc$ be a homotopy cooperad. Define the \name{operadic suspension} (resp. \name{desuspension})  of $\calc$ to be the homotopy cooperad $\calc\ot _{\mathrm{H}} \cals$ (resp. $\calc\ot_{\mathrm{H}}\cals^{-1}$),  denoted as $\mathscr{S}\calc$ (resp. $\mathscr{S}^{-1}\calc$). For cooperads, we have

\begin{defn}\mlabel{Def: cobar construction}
	Let $\calc=\{\calc(n)\}_{n\geqslant 1}$ be a coaugmented homotopy cooperad. The \name{cobar construction} of $\calc$, denoted by $\Omega\calc$,  is the free graded operad generated by the graded collection $s^{-1}\overline{\calc}$, endowed with the differential $\partial$ which is lifted from
	$$\partial: s^{-1}\overline{\calc}\to \Omega\calc, \ \partial(s^{-1}f):=-\sum_{T\in \mathfrak{T} } (s^{-1})^{\otimes \omega(T)} \circ \Delta_T(f) \tforall f\in \overline{\calc}(n).$$
 \end{defn}

This provides an alternative definition for   homotopy cooperads. In fact, a   graded collection  $\overline{\calc}=\{\overline{\calc}(n)\}_{n\geqslant 1}$ is a     homotopy cooperad if and only if   the free graded operad generated by $s^{-1}\overline{\calc}$ (also called the cobar construction of $\calc=\overline{C}\oplus  \cali$) is endowed with a differential such that it becomes a dg operad.

A natural  $L_\infty$-algebra  is associated with a   homotopy operad $\calp=\{\calp(n)\}_{n\geqslant 1}$.
 Denote $\calp^{\prod}:= \prod\limits_{n=1}^\infty\calp(n)$. For each $n\geqslant 1$, define operations $m_n=\sum\limits_{T\in \mathfrak{T}^{(n)}}m_T: (\calp^{\prod})^{\ot n}\to \calp^{\prod} $
  and define $l_n$ to be the anti-symmetrization of $m_n$, that is, $$l_n(x_1\ot \cdots \ot x_n):=\sum\limits_{\sigma\in \rmS_n}\chi(\sigma; x_1,\dots, x_n)m_n(x_{\sigma(1)}\ot \cdots \ot x_{\sigma(n)}).$$
The we have the following conclusion.
\begin{prop}\mlabel{pp:linfhomood}
	Let $\calp$ be a homotopy operad. Then $(\calp^{\prod}, \{l_n\}_{n\geqslant 1})$ is an $L_\infty$-algebra.
\end{prop}

For a homotopy operad $\calp$ satisfying $m_T=0$ for all $T\in \frakt$ with $\omega(T)\geqslant 3$, the operad $\calp$ reduces to the nonunital dg operad in the sense of Markl \mcite{Mar08} and $\calp^{\prod}$ is just a dg Lie algebra. The operation $l_2$ in the dg Lie algebra $\calp^{\prod}$ recovers the Gerstenhaber bracket defined in \eqref{eq:gers}.

\begin{defn}\label{Deninition: brace operation}Let $\calp$ be a (nonunital) dg operad.
 For  $f\in \calp(m)$ and $g_1\in\calp(l_1),\dots,g_n\in \calp(l_n) $ with $1\leqslant n\leqslant m$, define
\[f\{g_1,\dots,g_n\}:=\sum_{i_j\geqslant l_{j-1}+i_{j-1},n\geqslant j\geqslant 2, i_1\geqslant 1}\Big(\big((f\circ_{i_1} g_1)\circ_{i_2}g_2\big)\dots\Big)\circ_{i_n}g_n,\]
called the \name{brace operation} on $\calp^{\prod}$. For $f\in \calp(m)$ and $g\in \calp(m)$, define
\[[f,g]_{G}:=f\{g\}-(-1)^{|f||g|}g\{f\}\in \calp(m+n-1).\]
It is called the \name{Gerstenhaber bracket} of $f$ and $g$.
\end{defn}

The brace operation in a dg operad $\calp$ satisfies the following pre-Jacobi identity:
\begin{prop}\mcite{Ge1, GV95, Get93}
For  homogeneous elements $f, g_1,\dots, g_m,  h_1,\dots,h_n$ in $\calp^{\prod}$, we have
\begin{eqnarray}
	\mlabel{Eq: pre-jacobi}	&&\Big(f \{g_1,\dots,g_m\}\Big)\{h_1,\dots,h_n\}=\\
			\notag &&\quad  \sum\limits_{0\leqslant i_1\leqslant j_1\leqslant i_2\leqslant j_2\leqslant \dots \leqslant i_m\leqslant  j_m\leqslant n}(-1)^{\sum\limits_{k=1}^m(|g_k|)(\sum\limits_{j=1}^{i_k}(|h_j|))}
			f\{h_{1, i_1},  g_1\{h_{i_1+1, j_1}\},\dots,   g_m\{h_{i_m+1, j_m} \}, h_{j_m+1,  n}\}.
		\end{eqnarray}	
\end{prop}
In particular, the following equation holds.
	\begin{eqnarray}
			\mlabel{Eq: pre-jacobi1}
(f \{g\})\{h\}=f\{g\{h\}\}+f\{g, h\}+(-1)^{|g||h|}f\{h, g\}.
\end{eqnarray}

 Next we introduce the notion of  convolution homotopy operads.
\begin{prop-def}\mlabel{Prop: convolution homotopy operad}
	Let $\calc$ be a homotopy cooperad and $\calp$ be a dg operad. Then the graded collection $\mathbf{Hom}(\calc,\calp)=\{\Hom(\calc(n), \calp(n))\}_{n\geqslant 1}$ has a natural homotopy operad structure as follows.
	 \begin{enumerate}
	 	\item For  $T\in \mathfrak{T}$ with $\omega(T)=1$, $$m_T(f):=d_{\calp}\circ f-(-1)^{|f|}f\circ \Delta_T^\calc, \  f\in\mathbf{Hom}(\calc,\calp)(n).$$
	 	\item For  $T\in \mathfrak{T}$ with $\omega(T)\geqslant 2$, $$m_T(f_1\ot \cdots \ot f_r):=(-1)^{\frac{r(r-1)}{2}+1+r\sum_{i=1}^r |f_i|}m_{\calp}^T\circ(f_1\ot \cdots\ot f_r)\circ \Delta_T^\calc,$$  where $m_\calp^T$ is the composition in $\calp$ along $T$.
	 \end{enumerate}

 This homotopy operad is called the \textbf{convolution homotopy cooperad}.
\end{prop-def}

 The following result explains the Maurer-Cartan elements of the  $L_\infty$-algebra associated to a convolution homotopy operad.
\begin{prop}\mlabel{Prop: Linfinity give MC}
	Let $\calc$ be a coaugmented homotopy cooperad and $\calp$ be a unital dg operad. Then there is a natural bijection
	\[\Hom_{udgOp}(\Omega\calc, \calp)\cong \calm\calc\Big(\mathbf{Hom}(\overline{\calc},\calp)^{\prod}\Big)\]
between the set of morphisms of unital dg operads from $\Omega C$ to $\calp$ and the set of Maurer-Cartan elements in the $L_\infty$-algebra $\mathbf{Hom}(\overline{\calc},\calp)^{\prod}$.
\end{prop}

At last we recall the notions of minimal models and Koszul dual homotopy cooperads  of operads and explain how they are related to deformation complexes and $L_\infty$-algebra structures on deformation complexes.

For a collection $M=\{M(n)\}_{n\geqslant 1} $ of (graded) vector spaces, denote by $ \mathcal{F}(M)$ the free  graded  operad generated by $M$. Recall that a dg operad is called {\bf quasi-free} if its underlying graded operad is free.
\begin{defn}\mcite{DCV13} \mlabel{de:model} A \name{minimal model} for an operad $\mathcal{P}$  is a quasi-free dg operad $ (\mathcal{F}(M),d)$ together with a surjective quasi-isomorphism of operads $(\mathcal{F}(M), \partial)\overset{\sim}{\twoheadrightarrow}\mathcal{P}$, where the dg operad $(\mathcal{F}(M),  \partial)$  satisfies the following conditions.
	\begin{enumerate}
		\item The differential $\partial$ is decomposable, i.e. $\partial$ takes $M$ to $\mathcal{F}(M)^{\geqslant 2}$, the subspace of $\mathcal{F}(M)$ consisting of elements with weight $\geqslant 2$; \label{it:min1}
		\item The generating collection $M$ admits a decomposition $M=\bigoplus\limits_{i\geqslant 1}M_{(i)}$  such that $\partial(M_{(k+1)})\subset \mathcal{F}\Big(\bigoplus\limits_{i=1}^kM_{(i)}\Big)$ for all $k\geqslant 1$. \label{it:min2}
\end{enumerate}
\end{defn}

By~\mcite{DCV13}, if an operad $\mathcal{P}$ admits a minimal model, then it is unique up to isomorphisms.

For an operad $\calp$, assume that its minimal model $\calp_\infty$ exists. Since $\calp_\infty$ is quasi-free,
there exists a coaugmented homotopy cooperad $\calc$ such that $\Omega\calc\cong \calp_\infty$.  So  $\calc$ is called the \textbf{Koszul dual homotopy cooperad} of $\calp$, denoted by $\calp^\ac$.

 Let $V$ be a complex. Denote by  $\End_V$ the dg endomorphism  operad. Then the underlying complex of the convolution homotopy cooperad $\mathbf{Hom}(\overline{\calp^\ac}, \End_V)$ is called the \textbf{deformation complex} of $\calp$ on the complex $V$.
An element of  $\Hom_{udgOp}( \calp_\infty,  \End_V)$ is exactly a homotopy $\calp$-structure in $V$. So
 Proposition~\ref{Prop: Linfinity give MC} gives a bijection between  the set of homotopy $\calp$-structures on $V$ and that of Maurer-Cartan elements in the $L_\infty$-algebra on the deformation complex.

\subsection{ Koszul dual homotopy cooperad of $ \Dif$}\
\mlabel{ss:dual}

Now we introduce  the operad of differential algebras with weight.

\begin{defn}\mlabel{de:difod}
	The \name{operad for differential algebras of weight $\mathbf \lambda$} is defined to be the quotient of the free graded operad  $\mathcal{F(M)}$ generated by a  graded collection  $\mathcal{M}$ by an  operadic ideal $I$, where  the   collection  $\mathcal{M}$ is given by $\mathcal{M}(1)=\bfk d, \mathcal{M}(2)=\bfk \mu $ and $\mathcal{M}(n)=0$ for $n\neq 1,2$, and  where  $I$ is   generated by
	\begin{equation}\mlabel{eq:difod}
		\mu\circ_1\mu-\mu\circ_2\mu \ \  \mathrm{and}\ \
		d\circ_1\mu-\big(\mu\circ_1d+\mu\circ_2 d+\lambda\ (\mu\circ_1 d)\circ_2 d\big).
	\end{equation}
	Denote    this operad by  $\Dif$.
\end{defn}

The  homotopy cooperad $\mathscr{S} (\Dif^\ac)$ is defined on the graded collection with arity-$n$ component
$$\mathscr{S} (\Dif^\ac)(n):=\bfk  \widetilde{m_n} \oplus \bfk  \widetilde{d_n}$$
with $|\widetilde{m_n}|=0$, $|\widetilde{d_n}|=1$ for $n\geqslant 1$.

The coaugmented homotopy cooperad structure on the graded collection $\mathscr{S} (\Dif^\ac)$ is defined as follows.
First consider the following two types of trees:
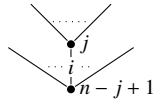
\begin{itemize}[keyvals]
	\item[(i)]
	\begin{eqnarray*}
		\begin{tikzpicture}[scale=0.6,descr/.style={fill=white}]
			\tikzstyle{every node}=[thick,minimum size=3pt, inner sep=1pt]
			\node(r) at (0,-0.5)[minimum size=0pt,circle]{};
			\node(v0) at (0,0)[fill=black, circle,label=right:{\tiny $ n-j+1$}]{};
			\node(v1-1) at (-1.5,1){};
			\node(v1-2) at(0,1)[fill=black,circle,label=right:{\tiny $\tiny j$}]{};
			\node(v1-3) at(1.5,1){};
			\node(v2-1)at (-1,2){};
			\node(v2-2) at(1,2){};
			\draw(v0)--(v1-1);
			\draw(v0)--(v1-3);
			\draw(v1-2)--(v2-1);
			\draw(v1-2)--(v2-2);
			\draw[dotted](-0.4,1.5)--(0.4,1.5);
			\draw[dotted](-0.5,0.5)--(-0.1,0.5);
			\draw[dotted](0.1,0.5)--(0.5,0.5);
			\path[-,font=\scriptsize]
			(v0) edge node[descr]{{\tiny$i$}} (v1-2);
		\end{tikzpicture}
	\end{eqnarray*}
	where $n\geqslant 1, 1\leqslant j\leqslant n, 1\leqslant i \leqslant n-j+1$.
	
	\item[(ii)]
	
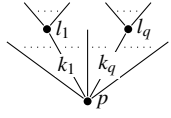
\begin{eqnarray*}
		\begin{tikzpicture}[scale=0.6,descr/.style={fill=white}]
			\tikzstyle{every node}=[thick,minimum size=3pt, inner sep=1pt]
	
			\node(v1-2) at(0,1.2)[circle,fill=black,label=right:{\tiny $p$}]{};

			\node(v2-1) at(-1.9,2.6){};
			\node(v2-2) at (-0.9, 2.8)[circle,fill=black,label=right:{\tiny $l_1$}]{};
			\node(v2-3) at (0,2.9){};
			\node(v2-4) at(0.9,2.8)[circle,fill=black,label=right:{\tiny $l_q$}]{};
			\node(v2-5) at(1.9,2.6){};
			\node(v3-1) at (-1.5,3.5){};
			\node(v3-2) at (-0.3,3.5){};
			\node(v3-3) at (0.3,3.5){};
			\node(v3-4) at(1.5,3.5){};
			\draw(v1-2)--(v2-1);
			\draw(v1-2)--(v2-3);
			\draw(v1-2)--(v2-5);
			\path[-,font=\scriptsize]
			(v1-2) edge node[descr]{{\tiny$k_1$}} (v2-2)
			edge node[descr]{{\tiny$k_{q}$}} (v2-4);
			\draw(v2-2)--(v3-1);
			\draw(v2-2)--(v3-2);
			\draw(v2-4)--(v3-3);
			\draw(v2-4)--(v3-4);
			\draw[dotted](-0.5,2.4)--(-0.1,2.4);
			\draw[dotted](0.1,2.4)--(0.5,2.4);
			\draw[dotted](-1.4,2.4)--(-0.8,2.4);
			\draw[dotted](1.4,2.4)--(0.8,2.4);
			\draw[dotted](-1.1,3.2)--(-0.6,3.2);
			\draw[dotted](1.1,3.2)--(0.6,3.2);
		\end{tikzpicture}
	\end{eqnarray*}
	where $  2\leqslant  p\leqslant n , 2\leqslant q \leqslant p, 1\leqslant k_1<\dots<k_{q}\leqslant p, l_1, \dots, l_q\geqslant 1,   l_1 + \cdots + l_q+p-q=n$.
\end{itemize}

Next, we define a family of operations $\{\Delta_T: \mathscr{S} (\Dif^\ac)\rightarrow \mathscr{S}( \Dif^\ac)^{\ot T}\}_{T\in \frakt}$ as follows:
\begin{itemize}
	\item[(i)] For the element $\widetilde{m_n}\in\mathscr{S} (\Dif^\ac)(n)$ and  a tree $T$ of type $\mathrm{(i)}$, define
	\begin{eqnarray*}
		\begin{tikzpicture}[scale=0.8,descr/.style={fill=white}]
			\tikzstyle{every node}=[thick,minimum size=5pt, inner sep=1pt]
			\node(r) at (0,-0.5)[minimum size=0pt,rectangle]{};
			\node(v-2) at(-2,0.5)[minimum size=0pt, label=left:{$\Delta_T(\widetilde{m_n})=$}]{};
			\node(v0) at (0,0)[draw,rectangle]{{\small $\widetilde{m_{n-j+1}}$}};
			\node(v1-1) at (-1.5,1){};
			\node(v1-2) at(0,1)[draw,rectangle]{\small$\widetilde{m_j}$};
			\node(v1-3) at(1.5,1){};
			\node(v2-1)at (-1,2){};
			\node(v2-2) at(1,2){};
			\draw(v0)--(v1-1);
			\draw(v0)--(v1-3);
			\draw(v1-2)--(v2-1);
			\draw(v1-2)--(v2-2);
			\draw[dotted](-0.4,1.5)--(0.4,1.5);
			\draw[dotted](-0.5,0.5)--(-0.1,0.5);
			\draw[dotted](0.1,0.5)--(0.5,0.5);
			\path[-,font=\scriptsize]
			(v0) edge node[descr]{{\tiny$i$}} (v1-2);
		\end{tikzpicture}
	\end{eqnarray*}
	\item[(ii)] For the element $\widetilde{d_n}\in\mathscr{S} (\Dif^\ac)(n)$ and  a tree $T$ of type $\mathrm{(i)}$, define
	\begin{eqnarray*}
		\begin{tikzpicture}[scale=0.8,descr/.style={fill=white}]
			\tikzstyle{every node}=[thick,minimum size=5pt, inner sep=1pt]
			\node(r) at (0,-0.5)[minimum size=0pt,rectangle]{};
			\node(va) at(-2,0.5)[minimum size=0pt, label=left:{$\Delta_T(\widetilde{d_n})=$}]{};
			\node(vb) at(-1.5,0.5)[minimum size=0pt ]{};
			\node(vc) at (0,0)[draw, rectangle]{\small $\widetilde{d_{n-j+1}}$};
			\node(v1) at(0,1)[draw,rectangle]{\small $\widetilde{m_j}$};
            \node(v0) at(-1.5,1){};
            \node(v2) at(1.5,1){};
			\node(v2-1)at (-1,2){};
			\node(v2-2) at(1,2){};
			\node(vd) at(1.5,0.5)[minimum size=0, label=right:$+$]{};
			\node(ve) at (3.5,0)[draw, rectangle]{\small $\widetilde{m_{n-j+1}}$};
			\node(ve1) at (3.5,1)[draw,rectangle]{\small $\widetilde{d_j}$};
			\node(ve2-1) at(2.5,2){};
            \node(ve0) at (2,1){};
            \node(ve2) at (5,1){};
			\node(ve2-2) at(4.5,2){};
			\draw(v1)--(v2-1);
			\draw(v1)--(v2-2);
			\draw[dotted](-0.4,1.5)--(0.4,1.5);
            \draw(vc)--(v0);
            \draw(vc)--(v2);
            \draw(ve)--(ve0);
            \draw(ve)--(ve2);
			\draw(ve1)--(ve2-1);
			\draw(ve1)--(ve2-2);
			\draw[dotted](3.1,1.5)--(3.9,1.5);
            \draw[dotted](-0.5,0.5)--(0,0.5);
            \draw[dotted](0,0.5)--(0.5,0.5);
            \draw[dotted](3,0.5)--(3.5,0.5);
            \draw[dotted](3.5,0.5)--(4,0.5);
            \path[-,font=\scriptsize]
			(vc) edge node[descr]{{\tiny$i$}} (v1);
            \path[-,font=\scriptsize]
			(ve) edge node[descr]{{\tiny$i$}} (ve1);
		\end{tikzpicture}
	\end{eqnarray*}
	
	\item[(iii)] For the element $\widetilde{d_n}\in\mathscr{S} (\Dif^\ac)(n)$ and a tree $T$ of type $\mathrm{(ii)}$, define
	\begin{eqnarray*}
		\begin{tikzpicture}[scale=0.8,descr/.style={fill=white}]
			\tikzstyle{every node}=[minimum size=4pt, inner sep=1pt]
			\node(v-2) at (-5,2)[minimum size=0pt, label=left:{$\Delta_T(\widetilde{d_n})=$}]{};
			\node(v-1) at(-5,2)[minimum size=0pt,label=right:{$(-1)^\frac{q(q-1)}{2}\lambda^{q-1}$}]{};
			\node(v1-1) at (-2,1){};
			\node(v1-2) at(0,1.2)[rectangle,draw]{\small $\widetilde{m_p}$};
			\node(v1-3) at(2,1){};
			\node(v2-1) at(-1.9,2.6){};
			\node(v2-2) at (-0.9, 2.8)[rectangle,draw]{\small$\widetilde{d_{l_1}}$};
			\node(v2-3) at (0,2.9){};
			\node(v2-4) at(0.9,2.8)[rectangle,draw]{\small $\widetilde{d_{l_q}}$};
			\node(v2-5) at(1.9,2.6){};
			\node(v3-1) at (-1.5,3.5){};
			\node(v3-2) at (-0.3,3.5){};
			\node(v3-3) at (0.3,3.5){};
			\node(v3-4) at(1.5,3.5){};
			\draw(v1-2)--(v2-1);
			\draw(v1-2)--(v2-3);
			\draw(v1-2)--(v2-5);
			\path[-,font=\scriptsize]
			(v1-2) edge node[descr]{{\tiny$k_1$}} (v2-2)
			edge node[descr]{{\tiny$k_{q}$}} (v2-4);
			\draw(v2-2)--(v3-1);
			\draw(v2-2)--(v3-2);
			\draw(v2-4)--(v3-3);
			\draw(v2-4)--(v3-4);
			\draw[dotted](-0.5,2.4)--(-0.1,2.4);
			\draw[dotted](0.1,2.4)--(0.5,2.4);
			\draw[dotted](-1.4,2.4)--(-0.8,2.4);
			\draw[dotted](1.4,2.4)--(0.8,2.4);
			\draw[dotted](-1.1,3.2)--(-0.6,3.2);
			\draw[dotted](1.1,3.2)--(0.6,3.2);
		\end{tikzpicture}
	\end{eqnarray*}
	\item[(iv)] $\Delta_T$ vanishes elsewhere.
\end{itemize}

Note that $\Delta_T(\widetilde{m_1})=\widetilde{m_1}\otimes \widetilde{m_1}\in  \mathscr{S} (\Dif^\ac)^{\otimes T}$ when $T$ is the tree of  type $(i)$ with $i=j=n=1$ and $\Delta_T(\widetilde{m_1})=0$ for every other tree $T$.  So $\bfk \widetilde{m_1}\cong \cali$ as homotopy cooperads and there is a  natural projection and injection $\varepsilon:\mathscr{S} (\Dif^\ac)\twoheadrightarrow \bfk \widetilde{m_1}\cong \cali$ and $\eta:\cali\cong \bfk \widetilde{m_1}\hookrightarrow \mathscr{S} (\Dif^\ac).$

\begin{prop}\mlabel{pp:dualcoop}
The graded collection $\mathscr{S} (\Dif^\ac)$ endowed with operations $\{\Delta_T\}_{T\in \frakt}$ forms a coaugmented homotopy cooperad with  strict counit  $\varepsilon:\mathscr{S} (\Dif^\ac)\twoheadrightarrow \bfk \widetilde{m_1}\cong \cali$  and the coaugmentation  $\eta:\cali\cong \bfk \widetilde{m_1}\hookrightarrow \mathscr{S} (\Dif^\ac).$ 	
\end{prop}

\begin{proof}
 First of all, it is easy to check that  $|\Delta_T|=\omega(T)-2$ for all $T\in \frakt$.

Next, we prove that $\{\Delta_T\}_{T\in \frakt}$ endows $\mathscr{S} (\Dif^\ac)$ with a homotopy cooperad structure. It suffices to prove that the induced derivation $\partial$ on the cobar construction $\Omega \mathscr{S} (\Dif^\ac)$, namely the free operad generated by $s^{-1}\overline{\mathscr{S} (\Dif^\ac)}$ is a differential:  $\partial^2=0$. To simplify notations, we denote the generators $s^{-1}\widetilde{m_n}, n\geqslant 2$ and $s^{-1}\widetilde{d_n}, n\geqslant 1$ by $\mu_n$ and $\nu_n$ respectively and note that $|\mu_n|=-1, n\geqslant 2$  and $|\nu_n|=0, n\geqslant 1$.

 	By the definition of $\mathscr{S} (\Dif^\ac)$ (Definition \ref{de:difod} $\mathrm {(i)-(iv)}$ ) and  $\Omega \mathscr{S} (\Dif^\ac)$ (Definition~\mref{Def: cobar construction}), the action of $\partial$ on generators $\mu_n=s^{-1}\widetilde{m_n}$ and $v_n=s^{-1}\widetilde{d_n}$ can be writted in terms of trees as:
 	\begin{eqnarray}\label{Eq: trees of differential of associative}
 		\begin{tikzpicture}[scale=0.33,baseline={(b0.base)}]
 			\tikzstyle{every node}=[thick,minimum size=3pt, inner sep=1pt]
 			\node(a) at (-2,-0.5){\begin{large}$\partial \ \ \ $\end{large}};
 			\node(b0)at (0,-1)[rectangle,draw]{\tiny{$\mu_n$}};
 			%  \node (1a) at (-0.5,1.5);
 			\node (b1) at (-2,1)  [minimum size=0pt,label=above:{\tiny{$1$}}]{};
 			\node (b2) at (0,1)  [minimum size=0pt,label=above:{}]{};
 			\node (b3) at (2,1)  [minimum size=0pt,label=above:{\tiny{$n$}}]{};
 			\draw        (b0)--(b1);
 			\draw        (b0)--(b2);
 			\draw        (b0)--(b3);
 			\draw [dotted,line width=0.5pt] (-1.25,0.25)--(1.25,0.25);
 		\end{tikzpicture}
 		&&
 		\begin{tikzpicture}[scale=0.4,descr/.style={fill=white},baseline={(b0.base)}]
 			\tikzstyle{every node}=[thick,minimum size=5pt, inner sep=1pt]
 			\node(a) at (2,-0.5)[minimum size=0pt, label=right:{$=-\sum\limits_{2\leqslant i\leqslant n-1\atop1\leqslant j\leqslant n-i+1}$}]{};
 			\node(r1) at(10,-1)[rectangle,draw]{\tiny{$\mu_{n-i+1}$}};
 			\node(r2) at(10,1.5)[rectangle,draw]{\tiny{$\mu_{i}$}};
 			\draw(r1)--(8,1);
 			\draw (r1)--(r2);
 			\draw(r1)--(12,1);
 			\draw(r2)--(8.5,3);
 			\path[-,font=\scriptsize]
 			(r1) edge node[descr]{{\tiny{$j$}}} (r2);
 			\draw(r2)--(11.5,3);
 			\draw [dotted,line width=0.5pt] (9,0.5)--(11,0.5);
 			\draw [dotted,line width=0.5pt] (9,2.5)--(11,2.5);
 		\end{tikzpicture}\\
 		\label{Eq: trees of differential of operators}
 		\begin{tikzpicture}[scale=0.35,baseline={(b0.base)}]
 			\tikzstyle{every node}=[thick,minimum size=3pt, inner sep=1pt]
 			\node(a) at (-2,0){\begin{large}$\partial \ \ \ $\end{large}};
 			\node(b0)at (0,-1)[rectangle,draw]{\tiny{$\nu_n$}};
 			% \node (1a) at (-0.5,1.5);
 			\node (b1) at (-2,1)  [minimum size=0pt,label=above:{\tiny{$1$}}]{};
 			\node (b2) at (0,1)  [minimum size=0pt,label=above:{}]{};
 			\node (b3) at (2,1)  [minimum size=0pt,label=above:{\tiny{$n$}}]{};
 			\draw        (b0)--(b1);
 			\draw        (b0)--(b2);
 			\draw        (b0)--(b3);
 			\draw [dotted,line width=0.5pt] (-1.25,0.25)--(1.25,0.25);
 		\end{tikzpicture}\!\!\!\!&&
 		\begin{tikzpicture}[scale=0.4,descr/.style={fill=white},baseline={(b0.base)}]
 			\tikzstyle{every node}=[thick,minimum size=5pt, inner sep=1pt]
 			\node(a) at (1,0)[minimum size=0pt, label=right:{$=\sum\limits_{2\leqslant i\leqslant n\atop 1\leqslant j\leqslant n-i+1}$}]{};
 			\node(r1) at(7,-1)[rectangle,draw]{\tiny{$\nu_{n-i+1}$}};
 			\node(r2) at(7,1.5)[rectangle,draw]{\tiny{$\mu_{i}$}};
 			\draw(r1)--(5,1);
 			\draw(r1)-- (r2);
 			\draw(r1)--(9,1);
 			\draw(r2)--(5.5,3);
 			\draw(r2)--(8.5,3);
 			\draw [dotted,line width=0.5pt] (6.5,0)--(7.5,0);
 			\draw [dotted,line width=0.5pt] (6.5,2.25)--(7.5,2.25);
 			\path[-,font=\scriptsize]
 			(r1) edge node[descr]{{\tiny{$j$}}} (r2);

 			\node(b) at(8.5,-0.5)[minimum size=0pt,label=right:{$\ \ \ \  -\sum\limits_{1\leqslant k_1<\dots<k_q\leqslant p\atop 1\leqslant q\leqslant p,2\leqslant p}\sum\limits_{l_1+\dots+l_q+p-q=n,\atop l_1,\dots,l_q \geqslant 1 }\!\!\!\!\lambda^{q-1}$}]{};
 			\node(l1) at(25,-1)[rectangle, draw]{\tiny{$\mu_p$}};
 			\node(l21) at(22,1.5)[rectangle, draw]{\tiny{$\nu_{l_1}$}};
 			\node(l22) at(25,1.5)[rectangle,draw]{\tiny{$\nu_{l_j}$}};
 			\node(l23) at(28,1.5)[rectangle,draw]{\tiny{$\nu_{l_q}$}};
 			\draw[dotted, line width=0.5](22.7,0.2)--(27.4,0.2);
 			\draw(l21)--(21,2.5);
 			\draw(l21)--(23,2.5);
 			\draw(l22)--(24,2.5);
 			\draw(l22)--(26,2.5);
 			\draw(l23)--(27,2.5);
 			\draw(l23)--(29,2.5);
 			\draw(l1)--(20.5,0.75);
 			\draw(l1)--(29.5,0.75);
 			\draw(l1)--(23.5,1.5);
 			\draw(l1)--(26.5,1.5);
 			\draw(l1)--(l21);
 			\draw(l1)--(l22);
 			\draw(l1)--(l23);
 			\draw[dotted,line width=0.5pt] (21.5,2.25)--(22.5,2.25);
 			\draw[dotted,line width=0.5pt] (24.5,2.25)--(25.5,2.25);
 			\draw[dotted,line width=0.5pt] (27.5,2.25)--(28.5,2.25);
 			\path[-,font=\scriptsize]
 			(l1) edge node[descr]{{\tiny{$k_1$}}} (l21)
 			(l1) edge node[descr]{{\tiny{$k_j$}}} (l22)
 			(l1) edge node[descr]{{\tiny{$k_q$}}} (l23);
 		\end{tikzpicture}
 	\end{eqnarray}
 	
 	As $\partial ^2$ is also a derivation in the free  operad $ \Omega \mathscr{S} (\Dif^\ac) $, to prove $\partial ^2=0$ is equivalent to proving that $\partial ^2=0 $ holds on the generators $\mu_n,n\geqslant 2$ and $\nu_n,n\geqslant 1$.
 	
 	For $\partial^2(\mu_n)$, we have
 	\begin{eqnarray*}
 	&&	\begin{tikzpicture}[scale=0.45,baseline={(b0.base)}]
 			\tikzstyle{every node}=[thick,minimum size=3pt, inner sep=1pt]
 			\node(a) at (-4,0.5){\begin{large}$\partial^2 \ \ \ $\end{large}};
 			\node(b0)at (-2,-0.5)[rectangle,draw]{\tiny{$\mu_n$}};
 			%  \node (1a) at (-0.5,1.5);
 			\node (b1) at (-3.5,1.5)  [minimum size=0pt,label=above:{\tiny{$1$}}]{};
 			\node (b2) at (-2,1.5)  [minimum size=0pt,label=above:{}]{};
 			\node (b3) at (-0.5,1.5)  [minimum size=0pt,label=above:{\tiny{$n$}}]{};
 			\draw        (b0)--(b1);
 			\draw        (b0)--(b2);
 			\draw        (b0)--(b3);
 			\draw [dotted,line width=1pt] (-3,1)--(-2.2,1);
 			\draw [dotted,line width=1pt] (-1.8,1)--(-1,1);
 		\end{tikzpicture}\\
 		&&
 		\begin{tikzpicture}[scale=0.45,descr/.style={fill=white},baseline={(b0.base)}]
 			\tikzstyle{every node}=[thick,minimum size=5pt, inner sep=1pt]
 			\node(a) at (0,-0.5)[minimum size=0pt, label=right:{$=\large{\partial}\ \Big(-\sum\limits_{2\leqslant i\leqslant n-1\atop 1\leqslant j \leqslant n-i+1}$}]{};
 			\begin{scope}[xshift=-2cm]
 \node(r1) at(10,-1)[rectangle,draw]{\tiny{$\mu_{n-i+1}$}};
 			\node(r2) at(10,1)[rectangle,draw]{\tiny{$\mu_{i}$}};
 			\node(r3) at(11.5,0)[minimum size=0pt,label=right:{$\Big)$}]{};
 			\draw(r1)--(8.5,1);
 			\draw (r1)--(r2);
 			\draw(r1)--(11.5,1);
 			\draw(r2)--(9,2.5);
 			\draw(r2)--(11,2.5);
 			\draw [dotted,line width=0.5pt] (9.5,0)--(10.5,0);
 			\draw [dotted,line width=0.5pt] (9.5,2.25)--(10.5,2.25);
 \path[-,font=\scriptsize]
 			(r1) edge node[descr]{{\tiny{$j$}}} (r2);
 \end{scope}
 		\end{tikzpicture}\\
 		&&\begin{tikzpicture}[scale=0.45,descr/.style={fill=white}]
 			\tikzstyle{every node}=[thick,minimum size=5pt, inner sep=1pt]
 			\node(a) at (0,-0.5)[minimum size=0pt, label=right:{$=\Big(-\sum\limits_{2\leqslant i\leqslant n-1\atop 1\leqslant j \leqslant n-i+1}$}]{};
 			\node(l1) at(8,-1)[rectangle,draw,label=left:{\tiny{$\partial$}}]{\tiny{$\mu_{n-i+1}$}};
 			\node(l2) at(8,1)[rectangle,draw]{\tiny{$\mu_{i}$}};
 			\node(l3) at(9.5,0)[minimum size=0pt,label=right:{$\Big)\ +\ \Big($}]{};
 			\node(b) at (13,-0.5)[minimum size=0pt, label=right:{$\sum\limits_{2\leqslant i\leqslant n-1\atop 1\leqslant j \leqslant n-i+1}$}]{};
 			\node(r1) at(18,-1)[rectangle, draw]{\tiny$\mu_{n-i+1}$};
 			\node(r2) at (18,1)[rectangle,draw,label=left:{\tiny{$\partial$}}]{\tiny$\mu_{i}$};
 			\node(r3) at(20,0)[minimum size=0pt, label=right:{$\Big)$}]{};
 			\draw(l1)--(7,0);
 			
 			\draw(l1)--(l2);
 			\draw(l1)--(9,0);
 			\draw(l2)--(7,2.5);
 			\draw(l2)--(9,2.5);
 			\draw [dotted,line width=0.5pt] (7.5,0)--(8.5,0);
 			\draw [dotted,line width=0.5pt] (7.5,2.25)--(8.5,2.25);
 			\draw(r1)--(17,0);
 			\draw(r1)--(19,0);
 			\draw(r1)--(r2);
 			\draw(r2)--(17,2.5);
 			\draw(r2)--(19,2.5);
 			\draw [dotted,line width=0.5pt] (17.5,0)--(18.5,0);
 			\draw [dotted,line width=0.5pt] (17.5,2.25)--(18.5,2.25);
  \path[-,font=\scriptsize]
            (l1)  edge node[descr]{{\tiny{$j$}}} (l2)
 			(r1) edge node[descr]{{\tiny{$j$}}} (r2);
 		\end{tikzpicture}\\
 		&&\begin{tikzpicture}[scale=0.45,descr/.style={fill=white}]
 			
 			\tikzstyle{every node}=[thick,minimum size=5pt, inner sep=1pt]
 			\tikzstyle{every node}=[thick,minimum size=5pt, inner sep=1pt]
 			\node(p) at (0,-0.5)[minimum size=0pt, label=right:{$=\Big(\sum\limits_{\tiny{\substack{i+r+s-2=n, \\ i,r,s \geqslant 2, \\ 1\leqslant j < k \leqslant r}}}$}]{};		\begin{scope}[xshift=1cm]
            \node(p1) at(6.5,-1)[rectangle,draw]{\tiny{$\mu_{r}$}};
 			\node(p21) at(5,1.5)[rectangle,draw]{\tiny{$\mu_{i}$}};
 			\node(p22) at(8,1.5)[rectangle,draw]{\tiny{$\mu_{s}$}};
 			\draw(p1)--(p21);
 			\draw(p1)--(p22);
 			\draw(p21)--(4,3);
 			\draw(p21)--(6,3);
 			\draw[dotted,line width=0.5pt](4.5,2.25)--(5.5,2.25);
 			\draw[dotted,line width=0.5pt](7.5,2.25)--(8.5,2.25);
 			\draw(p22)--(7,3);
 			\draw(p22)--(9,3);
 			\draw(p1)--(4,1);
 			\draw(p1)--(9,1);
 			\draw[dotted,line width=0.5pt](5,0.25)--(8,0.25);
 			\draw(p1)--(6.5,1.5);
 		 \path[-,font=\scriptsize]
            (p1)  edge node[descr]{{\tiny{$j$}}} (p21)
 			(p1) edge node[descr]{{\tiny{$k$}}} (p22);
 			\node(m) at(9,-0.5)[minimum size=0pt,label=right:{$\ \ -\!\!\!\sum\limits_{\tiny{\substack{i+r+s-2=n, \\ i,r,s\geqslant 2, \\ 1\leqslant k< j \leqslant r}}}$}]{};
  			\end{scope}
  \begin{scope}[xshift=1.5cm]
 			\node (m1) at (15.5,-1)[rectangle,draw]{\tiny{$\mu_r$}};
 			\node(m21) at (14,1.5)[rectangle,draw]{\tiny{$ \mu_s$}};
 			\node(m22) at (17,1.5)[rectangle,draw]{\tiny{$\mu_i$}};
 			\draw(m1)--(m21);
 			\draw(m1)--(m22);
 			\draw(m1)--(13,1);
 			\draw(m1)--(15.5,1);
 			\draw(m1)--(18,1);
 			\draw[dotted,line width=0.5pt](14,0.25)--(17,0.25);
 			\draw[dotted,line width=0.5pt](13.5,2.25)--(14.5,2.25);
 			\draw[dotted,line width=0.5pt](16.5,2.25)--(17.5,2.25);
 			\draw(m21)--(13,3);
 			\draw(m21)--(15,3);
 			\draw(m22)--(16,3);
 			\draw(m22)--(18,3);
  \path[-,font=\scriptsize]
            (m1)  edge node[descr]{{\tiny{$k$}}} (m21)
 			(m1) edge node[descr]{{\tiny{$j$}}} (m22);
 			\end{scope}
 \begin{scope}[xshift=2cm]
 			\tikzstyle{every node}=[thick,minimum size=5pt, inner sep=1pt]
 			\node(a) at (18,-0.5)[minimum size=0pt, label=right:{$+\sum\limits_{\tiny{\substack{i+r+s-2=n\\ i,r,s\geqslant 2}}} \sum\limits_{\tiny{\substack{1\leqslant k \leqslant r\\ 1\leqslant j' \leqslant s}}}$}]{};
 \end{scope}
 \begin{scope}[xshift=4.5cm]
 			\node(l1) at(23,-1)[rectangle,draw]{\tiny{$\mu_r$}};
 			\node(l2)at(23,0.5)[rectangle,draw]{\tiny{$\mu_{s}$}};
 			\node(l3) at(23,2)[rectangle,draw]{\tiny{$\mu_{i}$}};
 			\node(s31) at(24.5,0)[minimum size=0pt,label=right:{$\Big)$}]{};
 			\draw(l1)--(22,0);
 			\draw(l1)--(24,0);
 			\draw(l1)--(l2);
 			\draw(l2)--(l3);
 			\draw(l2)--(22,1.5);
 			\draw(l2)--(24,1.5);
 			\draw(l3)--(22,3);
 			\draw(l3)--(24,3);
 			\draw [dotted,line width=0.5pt] (22.5,-0.25)--(23.5,-0.25);
 			\draw [dotted,line width=0.5pt] (22.5,2.75)--(23.5,2.75);
 			\draw[dotted,line width=0.5pt] (22.5,1.25)--(23.5,1.25);
 \path[-,font=\scriptsize]
            (l1)  edge node[descr]{{\tiny{$k$}}} (l2)
 			(l2) edge node[descr]{{\tiny{$j'$}}} (l3);
 			\end{scope}
\end{tikzpicture}\\
&& \quad \begin{tikzpicture}[scale=0.45,descr/.style={fill=white}]
 			\tikzstyle{every node}=[thick,minimum size=5pt, inner sep=1pt]
 		\node(s32) at(24.5,0)[minimum size=0pt,label=right:{$\ -\ \Big($}]{};
 			\node(b) at (28,-0.5)[minimum size=0pt, label=right:{$\sum\limits_{\tiny{\substack{2\leqslant i \leqslant n-1\\ 1\leqslant j \leqslant n-i+1 }}} \sum\limits_{\tiny{\substack{2\leqslant s \leqslant i-1 \\ 1\leqslant k \leqslant i-s+1}}}$}]{};
 \begin{scope}[xshift=2.5cm]
 			\node(r1) at(33,-1)[rectangle, draw]{\tiny$\mu_{n-i+1}$};
 			\node(r2) at (33,0.5)[rectangle,draw]{\tiny$\mu_{i-s+1}$};
 			\node(r3) at(33,2)[rectangle,draw]{{\tiny{$\mu_s$}}};
 			\node(t3) at(35,0)[minimum size=0pt, label=right:{$\Big)$}]{};
 			\draw(r1)--(32,0);
 			\draw(r1)--(34,0);
 			\draw(r1)--(r2);
 			\draw(r2)--(r3);
 			\draw(r2)--(32,1.5);
 			\draw(r2)--(34,1.5);
 			\draw(r3)--(32,3);
 			\draw(r3)--(34,3);
 			\draw [dotted,line width=0.5pt] (32.5,1.25)--(33.5,1.25);
 			\draw [dotted,line width=0.5pt] (32.5,-0.25)--(33.5,-0.25);
 			\draw [dotted,line width=0.5pt] (32.5,2.75)--(33.5,2.75);
  \path[-,font=\scriptsize]
            (r1)  edge node[descr]{{\tiny{$j$}}} (r2)
 			(r2) edge node[descr]{{\tiny{$k$}}} (r3);
 \end{scope}
 		\end{tikzpicture}\\
 		&&=0,
 	\end{eqnarray*}
 	where we use the pre-Jacobi identity \meqref{Eq: pre-jacobi1} implicitly in the second-to-last equality.
 	
 	The computation of $\partial ^2(\nu_n)$ is more involved.
 	\begin{eqnarray*}
 		&&\begin{tikzpicture}[scale=0.43,baseline={(b0.base)}]
 			\tikzstyle{every node}=[thick,minimum size=3pt, inner sep=1pt]
 			\node(a) at (-2,0){\begin{large}$\partial^2 \ \ \ $\end{large}};
 \begin{scope}[xshift=1cm]
 			\node(b0)at (0,-1)[rectangle,draw]{\tiny{$\nu_n$}};
 			% \node (1a) at (-0.5,1.5);
 			\node (b1) at (-2.5,1.5)  [minimum size=0pt,label=above:{\tiny{$1$}}]{};
 			\node (b2) at (0,1.5)  [minimum size=0pt,label=above:{}]{};
 			\node (b3) at (2.5,1.5)  [minimum size=0pt,label=above:{\tiny{$n$}}]{};
 			\draw        (b0)--(b1);
 			\draw        (b0)--(b2);
 			\draw        (b0)--(b3);
 			\draw [dotted,line width=0.5pt] (-1.25,0.25)--(1.25,0.25);
 \end{scope}
 		\end{tikzpicture}\!\!\!\!\\
 		&&\begin{tikzpicture}[scale=0.43,descr/.style={fill=white},baseline={(b0.base)}]
 			\tikzstyle{every node}=[thick,minimum size=5pt, inner sep=1pt]
 			\node(a) at (0,0)[minimum size=0pt, label=right:{$=\large{\partial}\ \Big(\sum\limits_{2\leqslant i\leqslant n \atop 1\leqslant j \leqslant n-i-1 }$}]{};
 			\node(r1) at(7,-1)[rectangle,draw]{\tiny{$\nu_{n-i+1}$}};
 			\node(r2) at(7,1.5)[rectangle,draw]{\tiny{$\mu_{i}$}};
 			\draw(r1)--(6,0.5);
 			\draw(r1)-- (r2);
 			\draw(r1)--(8,0.5);
 			\draw(r2)--(6,2.5);
 			\draw(r2)--(8,2.5);
 			\draw [dotted,line width=0.5pt] (6.5,0)--(7.5,0);
 			\draw [dotted,line width=0.5pt] (6.5,2.25)--(7.5,2.25);
  \path[-,font=\scriptsize]
            (r1)  edge node[descr]{{\tiny{$j$}}} (r2);

 			\node(b) at(8,-0.5)[minimum size=0pt,label=right:{$\ \ \ \  -\sum\limits_{\tiny{\substack{l_1+\dots+l_q+p-q=n,\\ 1\leqslant q\leqslant p,p\geqslant 2 \\ 1\leqslant k_1 < \dots < k_q \leqslant p}}}\lambda^{q-1}$}]{};
 			\node(l1) at(21,-1)[rectangle, draw]{\tiny{$\mu_p$}};
 			\node(l21) at(18,1.5)[rectangle, draw]{\tiny{$\nu_{l_1}$}};
 			\node(l22) at(21,1.5)[rectangle,draw]{\tiny{$\nu_{l_j}$}};
 			\node(l23) at(24,1.5)[rectangle,draw]{\tiny{$\nu_{l_q}$}};
 			\draw[dotted, line width=0.5](18.7,0.2)--(23.4,0.2);
 			\draw(l21)--(17,2.5);
 			\draw(l21)--(19,2.5);
 			\draw(l22)--(20,2.5);
 			\draw(l22)--(22,2.5);
 			\draw(l23)--(23,2.5);
 			\draw(l23)--(25,2.5);
 			\draw(l1)--(16,1.3);
 			\draw(l1)--(26,1.3);
 			\draw(l1)--(19.5,1.5);
 			\draw(l1)--(22.5,1.5);
 			\draw(l1)--(l21);
 			\draw(l1)--(l22);
 			\draw(l1)--(l23);
 			\draw[dotted,line width=0.5pt] (17.5,2.25)--(18.5,2.25);
 			\draw[dotted,line width=0.5pt] (20.5,2.25)--(21.5,2.25);
 			\draw[dotted,line width=0.5pt] (23.5,2.25)--(24.5,2.25);
 			 \path[-,font=\scriptsize]
            (l1)  edge node[descr]{{\tiny{$k_1$}}} (l21)
            (l1)  edge node[descr]{{\tiny{$k_j$}}} (l22)
            (l1)  edge node[descr]{{\tiny{$k_q$}}} (l23);
 			\node(m) at(26.5,0)[minimum size=0pt, label=right:$\Big)$]{};

 		\end{tikzpicture}\\
 		&&\begin{tikzpicture}[scale=0.43,descr/.style={fill=white}]
 			\tikzstyle{every node}=[thick,minimum size=5pt, inner sep=1pt]
 			\node(a) at (0,0)[minimum size=0pt, label=right:{$=\sum\limits_{2\leqslant i\leqslant n\atop 1\leqslant j \leqslant n-i+1}$}]{};
 			\node(l1) at(6,-1)[rectangle,draw,label=left:{\tiny{$\partial$}}]{\tiny{$ \nu_{n-i+1}$}};
 			\node(l2) at(6,1.5)[rectangle,draw]{\tiny{$\mu_{i}$}};
 			\draw(l1)--(5,0.5);
 			\draw(l1)--(7,0.5);
 			\draw(l2)--(5,2.5);
 			\draw(l2)--(7,2.5);
 			\draw [dotted,line width=0.5pt] (5.5,0)--(6.5,0);
 			\draw [dotted,line width=0.5pt] (5.5,2.25)--(6.5,2.25);
 			\draw(l1)--(l2);
 			 \path[-,font=\scriptsize]
            (l1)  edge node[descr]{{\tiny{$j$}}} (l2);
 			
 			\node(b) at(7,0)[minimum size=0pt,label=right:{$\ \ +\ \ \sum\limits_{2\leqslant i\leqslant n\atop 1\leqslant j \leqslant n-i+1} $}]{};
 			\node (m1) at (14,-1)[rectangle,draw]{\tiny{$\nu_{n-i+1}$}};
 			\node(m2) at (14,1.5)[rectangle,draw,label=left:{\tiny{$\partial$}}]{\tiny{$ \mu_i$}};
 			\draw[dotted,line width=0.5pt](13.5,0)--(14.5,0);
 			\draw[dotted,line width=0.5pt](13.5,2.25)--(14.5,2.25);
 			\draw(m1)--(13,0.5);
 			\draw(m1)--(15,0.5);
 			\draw(m2)--(13,2.5);
 			\draw(m2)--(15,2.5);
 			\draw(m1)--(m2);
 			\path[-,font=\scriptsize]
            (m1)  edge node[descr]{{\tiny{$j$}}} (m2);

 			\node(c) at(17,0)[minimum size=0pt,label=right:{$-\sum\limits_{\tiny{\substack{l_1+\dots+l_q+p-q=n,\\ 1\leqslant q\leqslant p,p\geqslant 2 \\ 1\leqslant k_1 < \dots < k_q \leqslant p}}} {\lambda^{q-1}}$}]{};
 \begin{scope}[xshift=-2.5cm]
 			\node (r1) at (31,-1)[rectangle,draw,label=left:{\tiny{$\partial$}}]{\tiny{$\mu_p$}};
 			\node(r21) at(28,1.5)[rectangle, draw]{\tiny{$\nu_{l_1}$}};
 			\node(r22) at(31,1.5)[rectangle,draw]{\tiny{$\nu_{l_j}$}};
 			\node(r23) at(34,1.5)[rectangle,draw]{\tiny{$\nu_{l_q}$}};
 			\draw(r1)--(26,1.25);
 			\draw(r1)--(36,1.25);
 			\draw(r22)--(30,2.5);
 			\draw(r22)--(32,2.5);
 			\draw(r21)--(27,2.5);
 			\draw(r21)--(29,2.5);
 			\draw(r23)--(33,2.5);
 			\draw(r23)--(35,2.5);
 			\draw(r1)--(r21);
 			\draw(r1)--(r22);
 			\draw(r1)--(r23);
 			\draw[dotted,line width=0.5pt](28.6,0.25)--(33.4,0.25);
 			\draw[dotted,line width=0.5pt](27.5,2.25)--(28.5,2.25);
 			\draw(r1)--(29.5,1.5);
 			\draw(r1)--(32.5,1.5);
 			\draw[dotted,line width=0.5pt](30.5,2.25)--(31.5,2.25);
 			\draw[dotted,line width=0.5pt](33.5,2.25)--(34.5,2.25);
         \path[-,font=\scriptsize]
            (r1)  edge node[descr]{{\tiny{$k_1$}}} (r21)
            (r1)  edge node[descr]{{\tiny{$k_j$}}} (r22)
            (r1)  edge node[descr]{{\tiny{$k_q$}}} (r23);
            \end{scope}
 		\end{tikzpicture}\\
 		&&
 		\quad \begin{tikzpicture}
 			[scale=0.43,descr/.style={fill=white}]
 			\tikzstyle{every node}=[thick,minimum size=5pt, inner sep=1pt]
 			\node(a) at(0,0)[minimum size=0pt,label=right:{$+\sum\limits_{\tiny{\substack{l_1+\dots+l_q+p-q=n,\\ 1\leqslant q\leqslant p,p\geqslant 2,1\leqslant j \leqslant q \\ 1\leqslant k_1 < \dots < k_q \leqslant p}}}\lambda^{q-1}$}]{};
 			\node (l1) at (13,-1)[rectangle,draw]{\tiny{$ \mu_p$}};
 			\node(l21) at(10,1.5)[rectangle, draw]{\tiny{$\nu_{l_1}$}};
 			\node(l22) at(13,1.5)[rectangle,draw,label=left:{\tiny{$\partial$}}]{\tiny{$\nu_{l_j}$}};
 			\node(l23) at(16,1.5)[rectangle,draw]{\tiny{$\nu_{l_q}$}};
 			\draw(l1)--(8,1.3);
 			\draw(l1)--(18,1.3);
 			\draw(l21)--(9,2.5);
 			\draw(l21)--(11,2.5);
 			\draw(l22)--(12,2.5);
 			\draw(l22)--(14,2.5);
 			\draw(l23)--(15,2.5);
 			\draw(l23)--(17,2.5);
 			\draw(l1)--(l21);
 			\draw(l1)--(l22);
 			\draw(l1)--(l23);
 			\draw(l1)--(11.5,1.5);
 			\draw(l1)--(14.5,1.5);
 			\draw[dotted,line width=0.5pt](10.5,0.2)--(15.5,0.2);
 			\draw[dotted,line width=0.5pt](9.5,2.25)--(10.5,2.25);
 			\draw[dotted,line width=0.5pt](12.5,2.25)--(13.5,2.25);
 			\draw[dotted,line width=0.5pt](15.5,2.25)--(16.5,2.25);
          \path[-,font=\scriptsize]
            (l1)  edge node[descr]{{\tiny{$k_1$}}} (l21)
            (l1)  edge node[descr]{{\tiny{$k_j$}}} (l22)
            (l1)  edge node[descr]{{\tiny{$k_q$}}} (l23);
 		\end{tikzpicture}\\
 		&&\begin{tikzpicture}[scale=0.43,descr/.style={fill=white}]
 			\tikzstyle{every node}=[thick,minimum size=5pt, inner sep=1pt]
 			\node(a) at (0,0)[minimum size=0pt, label=right:{$=\Big(\sum\limits_{\tiny{\substack{i+r+s-2=n, \\ i,s \geqslant 2, r\geqslant 1, \\ 1\leqslant j < k \leqslant r}}}$}]{};
 \begin{scope}[yshift=-0.5cm]
 			\node(l1) at(8,-1)[rectangle,draw]{\tiny{$\nu_{r}$}};
 			\node(l21) at(6.5,1.5)[rectangle,draw]{\tiny{$\mu_{i}$}};
 			\node(l22) at(9.5,1.5)[rectangle,draw]{\tiny{$\mu_{s}$}};
 			\draw(l1)--(l21);
 			\draw(l1)--(l22);
 			\draw(l21)--(5.5,3);
 			\draw(l21)--(7.5,3);
 			\draw[dotted,line width=0.5pt](6,2.25)--(7,2.25);
 			\draw[dotted,line width=0.5pt](9,2.25)--(10,2.25);
 			\draw(l22)--(8.5,3);
 			\draw(l22)--(10.5,3);
 			\draw(l1)--(5,1.5);
 			\draw(l1)--(11,1.5);
 			\draw[dotted,line width=0.5pt](6.5,0.25)--(9.5,0.25);
 			\draw(l1)--(8,1.5);
 			\path[-,font=\scriptsize]
            (l1)  edge node[descr]{{\tiny{$j$}}} (l21)
 			(l1) edge node[descr]{{\tiny{$k$}}} (l22);
 \end{scope}
 			\begin{scope}[xshift=3cm]
 			\node(b) at(10,0)[minimum size=0pt,label=right:{$\ \ -\sum\limits_{\tiny{\substack{i+r+s-2=n, \\ i,s \geqslant 2, r\geqslant 1, \\ 1\leqslant k < j \leqslant r}}} $}]{};
 \begin{scope}[yshift=-0.5cm]
 			\node (m1) at (17.5,-1)[rectangle,draw]{\tiny{$\nu_r$}};
 			\node(m21) at (16,1.5)[rectangle,draw]{\tiny{$ \mu_s$}};
 			\node(m22) at (19,1.5)[rectangle,draw]{\tiny{$\mu_i$}};
 			\draw(m1)--(m21);
 			\draw(m1)--(m22);
 			\draw(m1)--(17.5,1.5);
 			\draw(m1)--(14.5,1.5);
 			\draw(m1)--(20.5,1.5);
 			\draw[dotted,line width=0.5pt](16,0.25)--(19,0.25);
 			\draw[dotted,line width=0.5pt](15.5,2.25)--(16.5,2.25);
 			\draw[dotted,line width=0.5pt](18.5,2.25)--(19.5,2.25);
 			\draw(m21)--(15,3);
 			\draw(m21)--(17,3);
 			\draw(m22)--(18,3);
 			\draw(m22)--(20,3);
 			\path[-,font=\scriptsize]
            (m1)  edge node[descr]{{\tiny{$k$}}} (m21)
 			(m1) edge node[descr]{{\tiny{$j$}}} (m22);
 			\end{scope}	\end{scope}
 		\begin{scope}[xshift=5cm]	
 			\node(c1) at(20,0)[minimum size=0pt,label=right:{$\ \ + \sum\limits_{\tiny{\substack{i+r+s-2=n\\ i,s \geqslant 2, r\geqslant 1}}} \sum\limits_{\tiny{\substack{1\leqslant k \leqslant r\\ 1\leqslant j' \leqslant s}}}$}]{};
 \begin{scope}[yshift=-1cm]	
 			\node(r1) at(27,-1)[rectangle,draw]{\tiny{$ \nu_r$}};
 			\node(r2) at (27,1)[rectangle,draw]{\tiny{$ \mu_s$}};
 			\node(r3)  at (27,3)[rectangle,draw]{\tiny{$ \mu_i$}};
 			\draw(r1)--(r2);
 			\draw(r2)--(r3);
 			\draw(r1)--(26,0);
 			\draw(r1)--(28,0);
 			\draw(r2)--(26,2);
 			\draw(r2)--(28,2);
 			\draw[dotted,line width=0.5pt](26.5,-0.25)--(27.5,-0.25);
 			\draw[dotted,line width=0.5pt](26.5,1.75)--(27.5,1.75);
 			\draw(r3)--(26,4);
 			\draw(r3)--(28,4);
 			\draw[dotted,line width=0.5pt](26.25,3.75)--(27.75,3.75);
 \path[-,font=\scriptsize]
            (r1)  edge node[descr]{{\tiny{$k$}}} (r2)
 			(r2) edge node[descr]{{\tiny{$j'$}}} (r3);
 \end{scope}
 \end{scope}
 		\end{tikzpicture}\\
&& \ \ \  \begin{tikzpicture}[scale=0.43,descr/.style={fill=white}]
 			\tikzstyle{every node}=[thick,minimum size=5pt, inner sep=1pt]
 			\node(d) at(29,0)[minimum size=0pt,label=right:{$\ -\sum\limits_{2\leqslant i\leqslant n-1\atop p\geqslant 2,1\leqslant q \leqslant p , }\sum\limits_{\tiny{\substack{ l_1+\dots+l_q+p-q=n-i+1 \\ 1\leqslant k_1 < \dots < k_q \leqslant p \\ 1\leqslant t\leqslant q,1\leqslant j'\leqslant l_t }}}\!\!\!\!\!{\lambda^{q-1}}$}]{};
 \begin{scope}[xshift=2cm,yshift=-0.5cm]
 			\node (rr1) at (41,-1)[rectangle,draw]{\tiny{$\mu_p$}};
 			\node(rr21) at(38,1.5)[rectangle, draw]{\tiny{$\nu_{l_1}$}};
 			\node(rr22) at(41,1.25)[rectangle,draw]{\tiny{$\nu_{l_{t}}$}};
 			\node(rr23) at(44,1.5)[rectangle,draw]{\tiny{$\nu_{l_q}$}};
 			\node(rr3) at(41, 3)[rectangle,draw]{\tiny{$\mu_i$}};
 			\draw(rr1)--(37,0.65);
 			\draw[dotted,line width=0.5pt](38.6,0.2)--(43.4,0.2);
 			\draw(rr1)--(45,0.65);
 			\draw(rr21)--(37,2.5);
 			\draw(rr21)--(39,2.5);
 			\draw[dotted,line width=0.5pt](37.5,2.25)--(38.5,2.25);
 			\draw(rr22)--(40,2.5);
 			\draw(rr22)--(42,2.5);
 			\draw[dotted,line width=0.5pt](44.5,2.25)--(43.5,2.25);
 			\draw(rr23)--(43,2.5);
 			\draw(rr23)--(45,2.5);
 			\draw[dotted,line width=0.5pt](40.5,2.25)--(41.5,2.25);
 			\draw[dotted,line width=0.5pt](40.5,3.75)--(41.5,3.75);
 			\draw(rr3)--(40,4);
 			\draw(rr3)--(42,4);
 			\draw(rr1)--(rr21);
 			\draw(rr1)--(rr22);
 			\draw(rr1)--(rr23);
 			\draw(rr22)--(rr3);
 			\draw(rr1)--(39.5,1.5);
 			\draw(rr1)--(42.5,1.5);
 			\path[-,font=\scriptsize]
            (rr1)  edge node[descr]{{\tiny{$k_1$}}} (rr21)
 			(rr1) edge node[descr]{{\tiny{$k_t$}}} (rr22)
            (rr1) edge node[descr]{{\tiny{$k_q$}}} (rr23)
          (rr22) edge node[descr]{{\tiny{$j'$}}} (rr3);
          \end{scope}
 			\begin{scope}[xshift=46cm]
 			\node(a1) at(0,0)[minimum size=0pt,label=right:{$\ \ \ -\sum\limits_{2\leqslant i \leqslant n-1 \atop p\geqslant 2, 1\leqslant q\leqslant p}\sum\limits_{\tiny{\substack{ l_1+\dots+l_q+p-q=n-i+1\\ 1\leqslant k_1 < \dots < k_q \leqslant p\\1\leqslant j' \leqslant p ,j'\neq k_1,\dots, k_q}}}\!\!\!\!{\lambda^{q-1}}$}]{};
 \begin{scope}[xshift=3cm,yshift=-0.5cm]
 			\node (l1) at (12,-1)[rectangle,draw]{\tiny{$\mu_p$}};
 			\node (l21) at(9,1.5)[rectangle, draw]{\tiny{$\nu_{l_1}$}};
 			\node(l22) at(12,2.5)[rectangle,draw]{\tiny{$\mu_{i}$}};
 			\node(l23) at(15,1.5)[rectangle,draw]{\tiny{$\nu_{l_q}$}};
 			\draw(l1)--(8,0.65);
 			\draw(l1)--(16,0.65);
 			\draw[dotted, line width=0.5pt](9.5,0.2)--(14.5,0.2);
 			\draw(l22)--(11,3.5);
 			\draw(l22)--(13,3.5);
 			\draw(l21)--(8,2.5);
 			\draw(l21)--(10,2.5);
 			\draw(l23)--(14,2.5);
 			\draw(l23)--(16,2.5);
 			\draw(l1)--(l21);
 			\draw(l1)--(l22);
 			\draw(l1)--(l23);
 			\draw(l1)--(10.5,1.5);
 			\draw(l1)--(13.5,1.5);
 			\draw[dotted,line width=0.5pt](8.5,2.25)--(9.5,2.25);
 			\draw[dotted,line width=0.5pt](14.5,2.25)--(15.5,2.25);
 			\draw[dotted,line width=0.5pt](11.5,3.25)--(12.5,3.25);
 \path[-,font=\scriptsize]
            (l1)  edge node[descr]{{\tiny{$k_1$}}} (l21)
 			(l1) edge node[descr]{{\tiny{$j'$}}} (l22)
            (l1) edge node[descr]{{\tiny{$k_q$}}} (l23);
            \node(b') at(16,0.75)[minimum size=0pt,label=right:{$\Big)$}]{};
            \end{scope}
 			\end{scope}
 			\end{tikzpicture}\\
 && \quad
\begin{tikzpicture}[scale=0.43,descr/.style={fill=white}]
 			\tikzstyle{every node}=[thick,minimum size=5pt, inner sep=1pt]
 			\node(b) at(16,0)[minimum size=0pt,label=right:{$-\sum\limits_{2\leqslant i \leqslant n,\atop 1\leqslant j \leqslant n-i+1  }\sum\limits_{2\leqslant s\leqslant i-1,\atop 1\leqslant k \leqslant i-s+1} $}]{};
 \begin{scope}[xshift=1cm,yshift=-1cm]
 			\node(m1) at(23,-1)[rectangle,draw]{\tiny{$ \nu_{n-i+1}$}};
 			\node(m2) at (23,1)[rectangle,draw]{\tiny{$ \mu_{i-s+1}$}};
 			\node(m3)  at (23,3)[rectangle,draw]{\tiny{$ \mu_s$}};
 			\draw(m1)--(m2);
 			\draw(m2)--(m3);
 			\draw(m1)--(22,0);
 			\draw(m1)--(24,0);
 			\draw(m2)--(22,2);
 			\draw(m2)--(24,2);
 			\draw(m3)--(22,4);
 			\draw(m3)--(24,4);
 			\draw[dotted,line width=0.5pt](22.5,-0.25)--(23.5,-0.25);
 			\draw[dotted,line width=0.5pt](22.25,1.75)--(23.75,1.75);
 			\draw[dotted,line width=0.5pt](22.25,3.75)--(23.75,3.85);
 			\path[-,font=\scriptsize]
            (m1)  edge node[descr]{{\tiny{$j$}}} (m2)
 			(m2) edge node[descr]{{\tiny{$k$}}} (m3);
 			\begin{scope}[yshift=0.5cm]
 			\node(c) at(25,0)[minimum size=0pt,label=right:{$+ \ \Big(\sum\limits_{\tiny{\substack{p\geqslant 2,  1\leqslant q\leqslant p, \\ l_1+\dots+l_q+p-q=n,\\ 1\leqslant u\leqslant v \leqslant q  }}} \sum\limits_{\tiny{\substack{2\leqslant i \leqslant p-1 \\ 1\leqslant j'\leqslant p-i+1 \\  1\leqslant k'_{u} < \dots < k'_v \leqslant i\\ 1\leqslant k'_1 < \dots <k'_{u-1}< j' < k'_{v+1} < \dots < k'_q \leqslant p-i+1}}} \!\!\!\!\!\!\lambda^{q-1}\ \ \ \ \  \ $}]{};
 \end{scope}     \end{scope}\begin{scope}[xshift=9cm,yshift=-1cm]
 			\node (l1) at (36,-1)[rectangle,draw]{\tiny{$\mu_{p-i+1}$}};
 			\node(l21) at(33.5,1)[rectangle, draw]{\tiny{$\nu_{l_1}$}};
 			\node(l22) at(36,1.5)[rectangle,draw]{\tiny{$\mu_{i}$}};
 			\node(l23) at(38.5,1)[rectangle,draw]{\tiny{$\nu_{l_q}$}};
 			\node(l31) at(35,3.5)[rectangle,draw]{\tiny{$\nu_{l_u}$}};
 			\node(l32) at(37,3.5)[rectangle,draw]{\tiny{$\nu_{l_{v}}$}};
 			\draw(l1)--(l21);
 			\draw(l1)--(l22);
 			\draw(l1)--(l23);
 			\draw(l1)--(32,0.3);
 			\draw(l1)--(40,0.3);
 			\draw(l22)--(l31);
 			\draw(l22)--(l32);
 			\draw(l31)--(34,5);
 			\draw(l31)--(35.8,5);
 			\draw(l32)--(36.2,5);
 			\draw(l32)--(38,5);
 			\draw(l21)--(32.5,2);
 			\draw(l21)--(34.5,2);
 			\draw(l23)--(37.5,2);
 			\draw(l23)--(39.5,2);
 			\draw(l1)--(35,1);
 			\draw(l1)--(37,1);
 			\draw[dotted,line width=0.5](34,0)--(38,0);
 			\draw[dotted,line width=0.5](33,1.75)--(34,1.75);
 			\draw[dotted,line width=0.5](38,1.75)--(39,1.75);
 			\draw(l22)--(33.5,3);
 			\draw(l22)--(38.5,3);
 			\draw[dotted,line width=0.5pt](35,2.5)--(37,2.5);
 			\draw[dotted,line width=0.5pt](34.5,4.5)--(35.5,4.5);
 			\draw[dotted,line width=0.5pt](36.5,4.5)--(37.5,4.5);
 \path[-,font=\scriptsize]
            (l1)  edge node[descr]{{\tiny{$k'_1$}}} (l21)
 			(l1) edge node[descr]{{\tiny{$j'$}}} (l22)
            (l1) edge node[descr]{{\tiny{$k'_q$}}} (l23)
            (l22) edge node[descr]{{\tiny{${k'_{u}}$}}} (l31)
            (l22) edge node[descr]{{\tiny{${k'_{v}}$}}} (l32);
 			\end{scope}
 		\end{tikzpicture}\\
 	&&\ \begin{tikzpicture}[scale=0.4,descr/.style={fill=white}]
 		\tikzstyle{every node}=[thick,minimum size=5pt, inner sep=1pt]
 	\node(a1) at(0,0)[minimum size=0pt,label=right:{$\ \ \ +\sum\limits_{ p\geqslant 2, 1\leqslant q\leqslant p, \atop 2\leqslant i \leqslant p-1 }\sum\limits_{\tiny{\substack{ l_1+\dots+l_q+p-q=n\\ 1\leqslant k'_1 < \dots < k'_q \leqslant p-i+1 \\1\leqslant j' \leqslant p-i+1 ,j'\neq k'_1,\dots, k'_q}}}\!\!\!\!\!\!\!{\lambda^{q-1}}$}]{};
 \begin{scope}[xshift=3cm,yshift=-0.5cm]
 			\node (l1) at (12,-1)[rectangle,draw]{\tiny{$\mu_{p-i+1}$}};
 			\node (l21) at(9,1.5)[rectangle, draw]{\tiny{$\nu_{l_1}$}};
 			\node(l22) at(12,2.5)[rectangle,draw]{\tiny{$\mu_{i}$}};
 			\node(l23) at(15,1.5)[rectangle,draw]{\tiny{$\nu_{l_q}$}};
 			\draw(l1)--(7.5,0.65);
 			\draw(l1)--(16.5,0.65);
 			\draw[dotted, line width=0.5pt](9.5,0.2)--(14.5,0.2);
 			\draw(l22)--(11,3.5);
 			\draw(l22)--(13,3.5);
 			\draw(l21)--(8,2.5);
 			\draw(l21)--(10,2.5);
 			\draw(l23)--(14,2.5);
 			\draw(l23)--(16,2.5);
 			\draw(l1)--(l21);
 			\draw(l1)--(l22);
 			\draw(l1)--(l23);
 			\draw(l1)--(10.5,1.5);
 			\draw(l1)--(13.5,1.5);
 			\draw[dotted,line width=0.5pt](8.5,2.25)--(9.5,2.25);
 			\draw[dotted,line width=0.5pt](14.5,2.25)--(15.5,2.25);
 			\draw[dotted,line width=0.5pt](11.5,3.25)--(12.5,3.25);
 \path[-,font=\scriptsize]
            (l1)  edge node[descr]{{\tiny{$k'_1$}}} (l21)
 			(l1) edge node[descr]{{\tiny{$j'$}}} (l22)
            (l1) edge node[descr]{{\tiny{$k'_q$}}} (l23);
            \end{scope}
            \begin{scope}[xshift=20cm]
            \begin{scope}[yshift=0.75cm]
	\node(a) at(-0.5,-0.5)[minimum size=0pt,label=right:{$\Big)+ \Big( \sum\limits_{\tiny{\substack{p\geqslant 2, 1\leqslant q \leqslant p, 1\leqslant j \leqslant q\\ l_1+\dots+l_q+p-q=n,\\ 1 \leqslant k_1 < \dots < k_q \leqslant p}}}\sum\limits_{ r_j+s_j-1=l_j \atop 1\leqslant t \leqslant r_j  }\!\! \small{\lambda^{q-1}}$}]{};
\end{scope}
\begin{scope}[xshift=3.5cm,yshift=-1cm]
 			\node (r1) at (11.5,-1)[rectangle,draw]{\tiny{$\mu_p$}};
 			\node(r21) at(9,1)[rectangle, draw]{\tiny{$\nu_{l_1}$}};
 			\node(r22) at(11.5,1.5)[rectangle,draw]{\tiny{$\nu_{r_j}$}};
 			\node(r23) at(14,1)[rectangle,draw]{\tiny{$\nu_{l_q}$}};
 			\node(r3) at(11.5,3.5)[rectangle,draw]{\tiny{$\mu_{s_j}$}};
 			\draw(r1)--(r21);
 			\draw(r1)--(r22);
 			\draw(r1)--(r23);
 			\draw(r1)--(7.5,0.3);
 			\draw(r1)--(15.5,0.3);
 			\draw(r22)--(r3);
 			\draw(r21)--(8,2);
 			\draw(r21)--(10,2);
 			\draw(r23)--(13,2);
 			\draw(r23)--(15,2);
 			\draw(r1)--(10.5,1);
 			\draw(r1)--(12.5,1);
 			\draw[dotted,line width=0.5](9.5,0)--(13.5,0);
 			\draw[dotted,line width=0.5](8.5,1.75)--(9.5,1.75);
 			\draw[dotted,line width=0.5](13.5,1.75)--(14.5,1.75);
 			\draw(r22)--(9.5,3);
 			\draw(r22)--(13.5,3);
 			\draw[dotted,line width=0.5pt](10.5,2.5)--(12.5,2.5);
 			\draw[dotted,line width=0.5pt](11,4.5)--(12,4.5);
 			\draw(r3)--(10.5,5);
 			\draw(r3)--(12.7,5);
  \path[-,font=\scriptsize]
            (r1)  edge node[descr]{{\tiny{$k_1$}}} (r21)
 			(r1) edge node[descr]{{\tiny{$k_j$}}} (r22)
            (r1) edge node[descr]{{\tiny{$k_q$}}} (r23)
            (r22) edge node[descr]{{\tiny{$t$}}} (r3);
 \end{scope}
 			\end{scope}
 			 \end{tikzpicture}\\
 		&& \quad \begin{tikzpicture}[scale=0.4,descr/.style={fill=white}]
 			\tikzstyle{every node}=[thick,minimum size=5pt, inner sep=1pt]
 			
 			\node(b1) at(15.5,-0.5)[minimum size=0pt,label=right:{$ -\sum\limits_{\tiny{\substack{p\geqslant 2, 1\leqslant q\leqslant p,1\leqslant j \leqslant q\\l_1+\dots+l_q+p-q=n, \\ 1\leqslant k_1 < \dots < k_q \leqslant p}}}\sum\limits_{\tiny{\substack{ r_j\geqslant 2, 1\leqslant k \leqslant r_j, \\ s_1+\dots+s_k+r_j - k =l_j \\ 1\leqslant t_1 < \dots < t_k \leqslant r_j}}}\lambda^{q+k-2}\ \ $}]{};
 \begin{scope}[yshift=-1.5cm]
 			\node (m1) at (31.5,-1)[rectangle,draw]{\tiny{$\mu_p$}};
 			\node(m21) at(29,1)[rectangle, draw]{\tiny{$\nu_{l_1}$}};
 			\node(m22) at(31.5,1.5)[rectangle,draw]{\tiny{$\mu_{r_j}$}};
 			\node(m23) at(34,1)[rectangle,draw]{\tiny{$\nu_{l_q}$}};
 			\node(m31) at(30.5,3.5)[rectangle,draw]{\tiny{$\nu_{s_1}$}};
 			\node(m32) at(32.5,3.5)[rectangle,draw]{\tiny{$\nu_{s_k}$}};
 			\draw(m1)--(m21);
 			\draw(m1)--(m22);
 			\draw(m1)--(m23);
 			\draw(m1)--(27.5,0.3);
 			\draw(m1)--(35.5,0.3);
 			\draw(m22)--(m31);
 			\draw(m22)--(m32);
 			\draw(m31)--(29.5,5);
 			\draw(m31)--(31.3,5);
 			\draw(m32)--(31.7,5);
 			\draw(m32)--(33.5,5);
 			\draw(m21)--(28,2);
 			\draw(m21)--(30,2);
 			\draw(m23)--(33,2);
 			\draw(m23)--(35,2);
 			\draw(m1)--(30.5,1);
 			\draw(m1)--(32.5,1);
 			\draw[dotted,line width=0.5](29.5,0)--(33.5,0);
 			\draw[dotted,line width=0.5](28.5,1.75)--(29.5,1.75);
 			\draw[dotted,line width=0.5](33.5,1.75)--(34.5,1.75);
 			\draw(m22)--(29.5,3);
 			\draw(m22)--(33.5,3);
 			\draw[dotted,line width=0.5pt](30.5,2.5)--(32.5,2.5);
 			\draw[dotted,line width=0.5pt](30,4.5)--(31,4.5);
 			\draw[dotted,line width=0.5pt](32,4.5)--(33,4.5);
 			\node(b2) at (36,1)[minimum size=0pt,label=right:{$\Big)$}]{};
   \path[-,font=\scriptsize]
            (m1)  edge node[descr]{{\tiny{$k_1$}}} (m21)
 			(m1) edge node[descr]{{\tiny{$k_j$}}} (m22)
            (m1) edge node[descr]{{\tiny{$k_q$}}} (m23)
            (m22) edge node[descr]{{\tiny{$t_1$}}} (m31)
            (m22) edge node[descr]{{\tiny{$t_k$}}} (m32);
            \end{scope}
 		\end{tikzpicture}
 	\end{eqnarray*}
 	Here, the pre-Jacobi identity \meqref{Eq: pre-jacobi} is implicitly used in the last equation above. Notice that  all  trees appear twice with opposite signs. Thus we have $\partial^2(\nu_n)=0$. 

 Finally, note that the maps $\varepsilon$ and $\eta$ are  strict morphisms of homotopy cooperads. This completes the proof.
 \end{proof}

\begin{defn}\mlabel{de:Koszul dual homotopy cooperad}
	The Hadamard product  $\mathscr{S} (\Dif^\ac)\ot_{\mathrm{H}} \cals^{-1}$ is called the \name{Koszul dual homotopy cooperad} of $\Dif$, and is denoted by $\Dif^\ac$.
\end{defn}

To be precise, let $sm_n=\widetilde{m_n}\otimes \varepsilon_n, n\geqslant 1$ and $sd_n=\widetilde{d_n}\otimes \varepsilon_n, n\geqslant 1$. Then  the underlying graded collection of $ \Dif^\ac$ is $ \Dif^\ac(n)=\bfk sm_n\oplus \bfk sd_n$ with $|sm_n|=n-1$ and $|sd_n|=n$.
 The family of operations $\{\Delta_T\}_{T\in \frakt}$ defining its homotopy cooperad structure is given by the following formulas.
\begin{itemize}
	\item[(i)] For the element $\widetilde{m_n}\in\mathscr{S} (\Dif^\ac)(n)$ and $T$ of type $\mathrm{(i)}$, define
	\begin{eqnarray*}
		\begin{tikzpicture}[scale=0.7,descr/.style={fill=white}]
			\tikzstyle{every node}=[thick,minimum size=5pt, inner sep=1pt]
			\node(r) at (0,-0.5)[minimum size=0pt,rectangle]{};
			\node(v-2) at(-2,0.5)[minimum size=0pt, label=left: {$\Delta_T(sm_n)= (-1)^{(j-1)(n-i-j+1)}$}]{};
			\node(v0) at (0,0)[draw,rectangle]{{\small $sm_{n-j+1}$}};
			\node(v1-1) at (-1.5,1){};
			\node(v1-2) at(0,1)[draw,rectangle]{\small$sm_j$};
			\node(v1-3) at(1.5,1){};
			\node(v2-1)at (-1,2){};
			\node(v2-2) at(1,2){};
			\draw(v0)--(v1-1);
			\draw(v0)--(v1-3);
			\draw(v1-2)--(v2-1);
			\draw(v1-2)--(v2-2);
			\draw[dotted](-0.4,1.5)--(0.4,1.5);
			\draw[dotted](-0.5,0.5)--(-0.1,0.5);
			\draw[dotted](0.1,0.5)--(0.5,0.5);
			\path[-,font=\scriptsize]
			(v0) edge node[descr]{{\tiny$i$}} (v1-2);
		\end{tikzpicture}
	\end{eqnarray*}
	\item[(ii)] For the element $\widetilde{d_n}\in\mathscr{S} (\Dif^\ac)(n)$ and  $T$ of type $\mathrm{(i)}$, define
	\begin{eqnarray*}
		\begin{tikzpicture}[scale=0.7,descr/.style={fill=white}]
			\tikzstyle{every node}=[thick,minimum size=5pt, inner sep=1pt]
			\node(r) at (0,-0.5)[minimum size=0pt,rectangle]{};
			\node(va) at(-2,0.5)[minimum size=0pt, label=left:{$\Delta_T(sd_n)=(-1)^{(j-1)(n-j-i+1)}$}]{};
			\node(vb) at(-1.5,0.5)[minimum size=0pt ]{};
			\node(vc) at (0,0)[draw, rectangle]{\small $sd_{n-j+1}$};
			\node(v1) at(0,1)[draw,rectangle]{\small $sm_j$};
            \node(v0) at(-1.5,1){};
            \node(v2) at(1.5,1){};
			\node(v2-1)at (-1,2){};
			\node(v2-2) at(1,2){};
			\node(vd) at(-0.5,-1.5)[minimum size=0, label=left:{$+(-1)^{(j-1)(n-i-j+1)+n-j}$}]{};
			\node(ve) at (2,-2)[draw, rectangle]{\small $sm_{n-j+1}$};
			\node(ve1) at (2,-1)[draw,rectangle]{\small $sd_j$};
			\node(ve2-1) at(1,0){};
            \node(ve0) at (0.5,-1){};
            \node(ve2) at (3.5,-1){};
			\node(ve2-2) at(3,0){};
			\draw(v1)--(v2-1);
			\draw(v1)--(v2-2);
			\draw[dotted](-0.4,1.5)--(0.4,1.5);
            \draw(vc)--(v0);
            \draw(vc)--(v2);
            \draw(ve)--(ve0);
            \draw(ve)--(ve2);
			\draw(ve1)--(ve2-1);
			\draw(ve1)--(ve2-2);
			\draw[dotted](1.6,-0.5)--(2.4,-0.5);
            \draw[dotted](-0.5,0.5)--(0,0.5);
            \draw[dotted](0,0.5)--(0.5,0.5);
            \draw[dotted](1.5,-1.5)--(2,-1.5);
            \draw[dotted](2,-1.5)--(2.5,-1.5);
            \path[-,font=\scriptsize]
			(vc) edge node[descr]{{\tiny$i$}} (v1);
            \path[-,font=\scriptsize]
			(ve) edge node[descr]{{\tiny$i$}} (ve1);
		\end{tikzpicture}
	\end{eqnarray*}
	
	\item[(iii)] For the element $\widetilde{d_n}\in\mathscr{S} (\Dif^\ac)(n)$ and tree $T$ of type $\mathrm{(ii)}$, define
	\begin{eqnarray*}
		\begin{tikzpicture}[scale=0.7,descr/.style={fill=white}]
			\tikzstyle{every node}=[minimum size=4pt, inner sep=1pt]
			\node(v-2) at (-5,2)[minimum size=0pt, label=left:{$\Delta_T(sd_n)=$}]{};
			\node(v-1) at(-5,2)[minimum size=0pt,label=right:{$(-1)^\alpha \lambda^{q-1}$}]{};
			\node(v1-1) at (-2,1){};
			\node(v1-2) at(0,1.2)[rectangle,draw]{\small $sm_p$};
			\node(v1-3) at(2,1){};
			\node(v2-1) at(-1.9,2.6){};
			\node(v2-2) at (-0.9, 2.8)[rectangle,draw]{\small$sd_{l_1}$};
			\node(v2-3) at (0,2.9){};
			\node(v2-4) at(0.9,2.8)[rectangle,draw]{\small $sd_{l_q}$};
			\node(v2-5) at(1.9,2.6){};
			\node(v3-1) at (-1.5,3.5){};
			\node(v3-2) at (-0.3,3.5){};
			\node(v3-3) at (0.3,3.5){};
			\node(v3-4) at(1.5,3.5){};
			\draw(v1-2)--(v2-1);
			\draw(v1-2)--(v2-3);
			\draw(v1-2)--(v2-5);
			\path[-,font=\scriptsize]
			(v1-2) edge node[descr]{{\tiny$k_1$}} (v2-2)
			edge node[descr]{{\tiny$k_{q}$}} (v2-4);
			\draw(v2-2)--(v3-1);
			\draw(v2-2)--(v3-2);
			\draw(v2-4)--(v3-3);
			\draw(v2-4)--(v3-4);
			\draw[dotted](-0.5,2.4)--(-0.1,2.4);
			\draw[dotted](0.1,2.4)--(0.5,2.4);
			\draw[dotted](-1.4,2.4)--(-0.8,2.4);
			\draw[dotted](1.4,2.4)--(0.8,2.4);
			\draw[dotted](-1.1,3.2)--(-0.6,3.2);
			\draw[dotted](1.1,3.2)--(0.6,3.2);
		\end{tikzpicture}
	\end{eqnarray*}
where $$\alpha=\sum_{s=1}^q(l_s-1)(p-k_s)+q(p-1) + \sum_{t=1}^{q-1}(q-t)l_t;$$
	\item[(iv)]All other components of $\Delta_T, T\in \frakt$ vanish.
\end{itemize}

\section{Operad of homotopy differential algebras with weight and minimal model}
\mlabel{sec:homomodel}
Now, we are going to introduce the notion of homotopy differential algebras and their governing dg operad and show that this dg operad is the minimal model of the operad of differential algebras with weight.

\subsection{Operad of homotopy differential algebras with weight}\
\mlabel{ss:operad}

\begin{defn}
The \name{operad  $\Difinfty$ of homotopy differential algebras of weight $\lambda$} is defined to be the cobar construction  $\Omega(\Dif^\ac)$ of  $\Dif^\ac$.
\end{defn}

A direct inspection gives the following description of  $\Difinfty$.
\begin{prop}\mlabel{def: expanded def of homotopy differential algebras}
Let $\mathcal{O}=(\mathcal{O}(1),\dots,\mathcal{O}(n),\dots)$ be the graded collection where  $\mathcal{O}(1)=\bfk d_1$ with $ |d_1|=0$ and for $n\geqslant 2$,  $\mathcal{O}(n)=\bfk d_n\oplus \bfk m_n$ with $ |d_n|=n-1, |m_n|=n-2$.
The operad  $\Difinfty$ of homotopy differential algebras of weight $\lambda$ is the differential graded operad   $(\mathcal{F(O)},\partial)$,  where the underlying free graded  operad is generated by the graded collection $\mathcal{O}$ and the action of the differential $\partial$ on generators is given by the following equations.
     For each $n\geqslant 2$,
\begin{equation}\mlabel{eq:diffgen1}   \partial ({m_n}) =\sum_{j=2}^{n-1}\sum_{i=1}^{n-j+1}(-1)^{i+j(n-i)}m_{n-j+1}\circ_i m_j\end{equation}
and for each $n\geqslant 1$,
{\small\begin{align}\mlabel{eq:diffgen2}  	\partial (d_n)&=-\sum_{j=2}^{n}\sum_{i=1}^{n-j+1}(-1)^{i+j(n-i)}d_{n-j+1}\circ_i m_j\\
\notag &-\hspace{-.7cm}\sum\limits_{\substack{1\leqslant k_1<\dots<k_{q}\leqslant p\\l_1+\dots+l_q+p-q=n\\l_1, \dots, l_q\geqslant 1,  2\leqslant p\leqslant n,1\leqslant q\leqslant p}} \hspace{-.5cm}(-1)^{\xi}\lambda^{q-1}\Big(\cdots\big(\big((m_p\circ_{k_1}d_{l_1})\circ_{k_2+l_1-1}d_{l_2}\big)\circ_{k_3+l_1+l_2-2}d_{l_3}\big)\cdots\Big)\circ_{k_{q}+l_1+\dots+l_{q-1}-q+1} d_{l_q}
\end{align}}
where $\xi:=\sum_{s=1}^q(l_s-1)(p-k_s).$
\end{prop}

\begin{remark}\mlabel{re:diffod}
 \begin{enumerate}
\item When $\lambda=0$, the operad  $\Difinfty$ is precisely the operad $\mathrm{AsDer}_\infty$  constructed by Loday in \mcite{Lod} by using Koszul duality for operads.
\item \label{it:diffod3} It is easy to see that
$\partial(m_2)=0=\partial(d_1)$. So the degree zero part of $\Difinfty$ is just the free non-graded operad generated by $m_2$ and $d_1$.
 This fact will be used in Lemma~\mref{Lem: H0 of Difinfty is Dif}.
\end{enumerate}
    \end{remark}

As free operads are constructed from   planar rooted trees, we use planar rooted  trees to display the dg operad $\Difinfty$ for later applications.
First, we use the following two kinds of corollas to represent generators of $\Difinfty$:
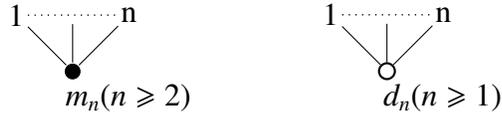
 \begin{figure}[h]
	\begin{tikzpicture}[grow'=up,scale=0.5]
		\tikzstyle{every node}=[thick,minimum size=6pt, inner sep=1pt]
		\node(r)[fill=black,circle,label=below:$\ \ \ \ \ \ \ \ \ \ \ \ \ \ m_n(n\geqslant 2)$]{}
		child{node(1){1}}
		 child{node(i) {}}
		child{node(n){n}};
		\draw [dotted,line width=0.5pt] (1)--(n);
	%	\draw [dotted, line width=0.5pt] (i)--(n);
	\end{tikzpicture}
	\hspace{8mm}
	\begin{tikzpicture}[grow'=up,scale=0.5]
		\tikzstyle{every node}=[thick,minimum size=6pt, inner sep=1pt]
		\node(r)[draw,circle,label=below:$\ \ \ \ \ \ \ \ \ \ \ \ \ \  d_n(n\geqslant 1)$]{}
		child{node(1){1}}
		child{node(i){}}
		child{node(n){n}};
%child{node(0){0}};
		\draw [dotted,line width=0.5pt] (1)--(n);
	%	\draw [dotted, line width=0.5pt] (i)--(n);
% \draw [dotted, line width=0.5pt] (0)--(n);
	\end{tikzpicture}
\end{figure}

For example, the element   $(((m_3\circ_{1}d_4)\circ_3m_4)\circ_9m_2)\circ_{10}d_2$ can be represented by the following tree
{\tiny\begin{figure}[h]
	\begin{tikzpicture}[grow'=up,scale=0.4]
		\tikzstyle{every node}=[level distance=30mm,sibling distance=10em,thick,minimum size=1.5pt, inner sep=2pt]
		\node[fill=black,draw,circle,label=right:$\small m_3$]{}
		child{
			node[draw, circle, label=left:$d_4$](1){}
			child{node(1-1){}}
			child{node(1-2){}}
			child{
				node[fill=black, draw, circle,label=left:$m_4$](1-3){}
				child{node(1-3-1){}}
				child{node(1-3-2){}}
				child{node(1-3-3){}}
				child{node(1-3-4){}}
			}
			child{node(1-4){}}
		}
		child{node(2){}}
		child{
			node[fill=black,draw,circle,label=right:$m_2$](3){}
			child{node(3-1){}}
			child{
				node[circle, draw, label=right:$d_2$](3-2){}
				child{node(3-2-1){}}
				child{node(3-2-2){}}}
		};
	\end{tikzpicture}
\end{figure}
}

  Eqs. \meqref{eq:diffgen1} and \meqref{eq:diffgen2} can be represented by
\begin{eqnarray*}\hspace{-1.5cm}
	\begin{tikzpicture}[scale=0.4]
		\tikzstyle{every node}=[thick,minimum size=6pt, inner sep=1pt]
		\node(a) at (-4,0.5){\begin{large}$\partial$\end{large}};
		\node[circle, fill=black, label=right:$m_n (n\geqslant 2)$] (b0) at (-2,-0.5)  {};
		%  \node (1a) at (-0.5,1.5);
		\node (b1) at (-3.5,1.5)  [minimum size=0pt,label=above:$1$]{};
		\node (b2) at (-2,1.5)  [minimum size=0pt,label=above:$$]{};
		\node (b3) at (-0.5,1.5)  [minimum size=0pt,label=above:$n$]{};
		\draw        (b0)--(b1);
		\draw        (b0)--(b2);
		\draw        (b0)--(b3);
		\draw [dotted,line width=1pt] (-3,1)--(-2.2,1);
		\draw [dotted,line width=1pt] (-1.8,1)--(-1,1);
	\end{tikzpicture}
	&
	\hspace{1mm}
	\begin{tikzpicture}[scale=0.4]
		\node(0){{$= \sum\limits_{j=2}^{n-1}\sum\limits_{i=1}^{n-j+1}(-1)^{i+j(n-i)}$}};
	\end{tikzpicture}
	&
	\begin{tikzpicture}[scale=0.5]
		\tikzstyle{every node}=[thick,minimum size=6pt, inner sep=1pt]
		\node(e0) at (0,-1.5)[circle, fill=black,label=right:$m_{n-j+1}$]{};
		\node(e1) at(-1.5,0){};
		\node(e2-0) at (0,-0.5){{\tiny$i$}};
		\node(e3) at (1.5,0){};
		\node(e2-1) at (0,0.5) [circle,fill=black,label=right: $m_j$]{};
		\node(e2-1-1) at (-1,1.5){};
		\node(e2-1-2) at (1, 1.5){};
		\draw [dotted,line width=1pt] (-0.7,-0.5)--(-0.2,-0.5);
		\draw [dotted,line width=1pt] (0.3,-0.5)--(0.8,-0.5);
		\draw [dotted,line width=1pt] (-0.4,1)--(0.4,1);
		\draw        (e0)--(e1);
		\draw         (e0)--(e3);
		\draw         (e0)--(e2-0);
		\draw         (e2-0)--(e2-1);
		\draw        (e2-1)--(e2-1-1);
		\draw        (e2-1)--(e2-1-2);
	\end{tikzpicture}	
\end{eqnarray*}

\begin{eqnarray*}
	\hspace{1.2cm}
	\begin{tikzpicture}[scale=0.4]
		\tikzstyle{every node}=[thick,minimum size=6pt, inner sep=1pt]
		\node(a) at (-4,0.5){\begin{large}$\partial$\end{large}};
		\node[circle, draw, label=right:$d_n(n\geqslant 1)$] (b0) at (-2,-0.5)  {};
		%  \node (1a) at (-0.5,1.5);
		\node (b1) at (-3.5,1.5)  [minimum size=0pt,label=above:$1$]{};
		\node (b2) at (-2,1.5)  [minimum size=0pt,label=above:$$]{};
		\node (b3) at (-0.5,1.5)  [minimum size=0pt,label=above:$n$]{};
		\draw        (b0)--(b1);
		\draw        (b0)--(b2);
		\draw        (b0)--(b3);
		\draw [dotted,line width=1pt] (-3,1)--(-2.2,1);
		\draw [dotted,line width=1pt] (-1.8,1)--(-1,1);
	\end{tikzpicture}
	&\hspace{-1cm}
	\begin{tikzpicture}[scale=0.4]
		\node(0){{${\Large=}\ \ \ \ \ -\sum\limits_{j=2}^{n}\sum\limits_{i=1}^{n-j+1}(-1)^{i+j(n-i)}$}};
	\end{tikzpicture}
	& \hspace{-1.4cm}
	\begin{tikzpicture}[scale=0.6]
		\tikzstyle{every node}=[thick,minimum size=6pt, inner sep=1pt]
		\node(e0) at (0,-1.5)[circle, draw,label=right:$d_{n-j+1}$]{};
		\node(e1) at(-1.5,-0.3){};
		\node(e2-0) at (0,-0.7){{\tiny$i$}};
		\node(e3) at (1.5,-0.3){};
		\node(e2-1) at (0,0) [draw,circle,fill=black,label=right: $m_j$]{};
		\node(e2-1-1) at (-1,1){};
		\node(e2-1-2) at (1, 1){};
		\draw [dotted,line width=1pt] (-0.7,-0.5)--(-0.2,-0.5);
		\draw [dotted,line width=1pt] (0.3,-0.5)--(0.8,-0.5);
		\draw [dotted,line width=1pt] (-0.2,0.5)--(0.2,0.5);
		\draw        (e0)--(e1);
		\draw         (e0)--(e3);
		\draw         (e0)--(e2-0);
		\draw         (e2-0)--(e2-1);
		\draw        (e2-1)--(e2-1-1);
		\draw        (e2-1)--(e2-1-2);
	\end{tikzpicture}	\\
	&\hspace{1.5cm} \begin{tikzpicture}[scale=0.4]
		\node(0){{$-\sum\limits_{\substack{1\leqslant k_1<\dots<k_{q}\leqslant p\\l_1+\dots+l_q+p-q=n\\l_1, \dots, l_q\geqslant 1,  2\leqslant p\leqslant n,1\leqslant q\leqslant p}} (-1)^{\xi}\lambda^{q-1}$}};
	\end{tikzpicture}&
\hspace{-3mm}
		\begin{tikzpicture}[scale=0.7,descr/.style={fill=white}]
		\tikzstyle{every node}=[thick,minimum size=6pt, inner sep=1pt]
		\node(e0) at(0,0)[circle,draw, fill=black,label=below:$m_p$]{};
		\node(e1-1) at(-2.7,1.5){};
		\node(e1-2) at (-1.5,1.6)[circle, draw, label=right:$d_{l_1}$]{};
		\node(e1-3) at (-0.7,1.5){};
		\node(e1-4) at(0,1.6)[circle, draw, label=right:$d_{l_i}$]{};
		\node(e1-5) at(0.7,1.5){};
		\node(e1-6) at(1.5,1.6)[circle, draw, label=right:$d_{l_q}$]{};
		\node(e1-7) at(2.7,1.5){};
		\node(e2-1) at (-2,2.5){};
		\node(e2-2) at(-1,2.5){};
		\node(e2-3) at(-0.5,2.5){};
		\node(e2-4) at (0.5,2.5){};
		\node(e2-5) at (1,2.5){};
		\node(e2-6) at (2,2.5){};

\draw (e0)--(e1-1);
\draw (e0)--(e1-3);
\draw (e0)--(e1-5);
\draw(e0)--(e1-7);
\draw(e1-2)--(e2-1);
\draw(e1-2)--(e2-2);
\draw(e1-4)--(e2-3);
\draw(e1-4)--(e2-4);
\draw(e1-6)--(e2-5);
\draw(e1-6)--(e2-6);
\draw[dotted,line width=1pt](-1.4,0.9)--(1.4,0.9);
\draw[dotted, line width=1pt](-1.7,2.2)--(-1.3,2.2);
\draw[dotted, line width=1pt](-0.2,2.2)--(0.2,2.2);
\draw[dotted, line width=1pt](1.7,2.2)--(1.3,2.2);
\path[-,font=\scriptsize]
(e0) edge node[descr]{{\tiny$k_1$}} (e1-2)
      edge node[descr]{{\tiny$k_i$}} (e1-4)
      edge node[descr]{{\tiny$k_q$}}(e1-6);
	\end{tikzpicture}
\end{eqnarray*}

\subsection{The minimal model of  the operad of  differential algebras}
\mlabel{sec:model}\

The following result is the main result of this paper.
\begin{thm} \mlabel{thm:difmodel}
	The dg operad $\Difinfty$ is the minimal model of the operad $\Dif$.
\end{thm}
The proof of Theorem~\mref{thm:difmodel} will be carried out in the rest of this subsection.

\begin{lem} \mlabel{Lem: H0 of Difinfty is Dif}
 The  natural surjection
 \begin{equation} \label{eq:proj}
 	p:\Difinfty\rightarrow \Dif
 \end{equation}
induces an isomorphism   $\rmH_0(\Difinfty)\cong \Dif$.
\end{lem}
\begin{proof} By Remark~\mref{re:diffod} \eqref{it:diffod3},
the degree zero part of $\Difinfty$ is the free (non-graded) operad  generated by $m_2$ and $d_1$. By definition, one has
$$\partial(m_3)=-\mu\circ_1\mu+\mu\circ_2\mu,$$
$$\partial(d_2)= d_1\circ_1 m_2-(m_2\circ_1d_1+m_2\circ_2 d_1+\lambda (m_2\circ_1 d_1)\circ_2 d_1). $$
So  the   surjection $p:\Difinfty\rightarrow \Dif$ sending $m_2$ to $\mu $ and $d_1$ to $d$ induces  $\rmH_0(\Difinfty)\cong \Dif$.
\end{proof}

Obviously, the differential $\partial$ of $\Difinfty$ satisfies Definition~\mref{de:model}~\eqref{it:min1} and \eqref{it:min2},  so in order to complete the proof of Theorem~\mref{thm:difmodel}, we only need to show that $p:\Difinfty\rightarrow \Dif$ is a quasi-isomorphism. For this purpose, we are going to build a homotopy map, i.e., a degree $1$ map  $H:\Difinfty\rightarrow \Difinfty$ such that in each  positive degree
$$\partial\circ H+H\circ \partial =\Id.$$

We need the following notion of graded  path-lexicographic ordering on $\Difinfty$.

Each tree monomial gives rise to a path sequence~\cite[Chapter 3]{BD16}.  More precisely,
 to a tree monomial $T$ with  $n$ leaves (written as $\mbox{arity}(T)=n$), we can associate a sequence   $(x_1, \dots, x_n)$  where  $x_i$ is the word formed by   generators of $\Difinfty$ corresponding to the vertices along the unique path from the root of $T$  to its $i$-th leaf.
We define an order on the graded tree monomials as follows.
For graded tree monomials $T$ and $T'$, define $T>T'$ if
 \begin{enumerate}
 	\item either $\mbox{arity}(T)>\mbox{arity}(T')$;
 	\item or $\mbox{arity}(T)=\mbox{arity}(T')$, and $\deg(T)>\deg(T')$, where $\deg(T)$ is the sum of the degrees of all generators of  $\Difinfty$ appearing in the   tree monomial $T$;
 	\item or $\mbox{arity}(T)=\mbox{airty}(T')(=n), \deg(T)=\deg(T')$, and the path sequences   $(x_1,\dots,x_n)$ and $(x'_1,\dots,x_n')$ associated to $T$ and $T'$ satisfy $(x_1,\dots,x_n)>(x_1',\dots,x_n')$ with respect to the length-lexicographic order of words induced by \[m_2<d_1<m_3<d_2<\dots<m_n<d_{n-1}<m_{n+1}<d_n<\dots.\]
 \end{enumerate}
 It is readily seen that $>$ is a well order.

Given a     linear combination of tree monomials, its \textbf{leading monomial} is the largest tree monomial appearing with nonzero coefficient in this  linear combination,  and the coefficient of the leading monomial is called the \textbf{leading coefficient}.

	Let $S$ be a generator of degree $\geqslant 1$ in $ \Difinfty$. Denote the leading monomial of $\partial S$ by $\widehat{S}$ and the  coefficient of $\widehat{S}$ in $\partial S$ by $c_S$.   It is easily seen that the coefficient $c_S$ is always $\pm 1$.
A tree monomial of the form $\widehat{S}$  with the degree of $S$ being positive  is called \textbf{typical}.
More precisely, typical tree monomials are of the following forms:
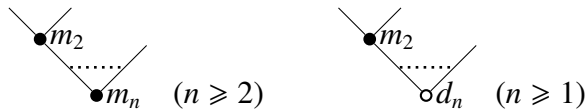
\begin{equation*}
	\begin{tikzpicture}[scale=0.5]
	\tikzstyle{every node}=[thick,minimum size=4pt, inner sep=1pt]
\node(e0) at (0,0)[circle, draw, fill=black, label=right:$m_n\quad (n\geqslant 2)$]{};
\node(e1-1) at (-1.5,1.5)[circle, draw,fill=black, label=right:$m_2$]{};
\node(e1-2) at(1.5,1.5){};
\node(e2-1) at(-2.5,2.5){};
\node(e2-2) at(-0.5,2.5){};
\draw(e0)--(e1-1);
\draw(e0)--(e1-2);
\draw(e1-1)--(e2-1);
\draw(e1-1)--(e2-2);
\draw[dotted,line width=1pt](-0.7,0.75)--(0.7,0.75);
	\end{tikzpicture}
\hspace{8mm}
\begin{tikzpicture}[scale=0.5]
\tikzstyle{every node}=[thick,minimum size=4pt, inner sep=1pt]
\node(e0) at (0,0)[circle, draw,  label=right:$d_n\quad (n\geqslant 1)$]{};
\node(e1-1) at (-1.5,1.5)[circle, draw,fill=black, label=right:$m_2$]{};
\node(e1-2) at(1.5,1.5){};
\node(e2-1) at(-2.5,2.5){};
\node(e2-2) at(-0.5,2.5){};
\draw(e0)--(e1-1);
\draw(e0)--(e1-2);
\draw(e1-1)--(e2-1);
\draw(e1-1)--(e2-2);
\draw[dotted,line width=1pt](-0.7,0.75)--(0.7,0.75);
	
\end{tikzpicture}
	\end{equation*}
 where the first one is the leading monomial of $\partial m_{n+1}, n\geqslant 2$, the second being that of $\partial d_{n+1}, n\geqslant 1$.

\begin{defn}
We call a tree monomial \textbf{ effective}  if  it satisfies the the following conditions:
\begin{enumerate}
	\item There exists a typical divisor $\widehat{S}$ in $T$;
 	\item On the path from  the root $v$ of $\widehat{S}$   to    the leftmost leaf   $l$ of $T$ above  $v$,    there are no other typical divisors and there are no vertices of positive degree except possibly $v$ itself;
 	\item For a leaf $l'$ of $T$ which is located on the left of $l$, there are no vertices of positive degree and no typical divisors on the path from the root of $T$ to $l'$.
\end{enumerate}
The divisor $\widehat{S}$ is unique if exists, and is called the {\bf effective divisor} of $T$ and $l$ is called the {\bf effective leaf} of $T$.
\end{defn}

Intuitively, the effective divisor of a tree monomial $T$ is the left-upper-most typical divisor of $T$.
It follows by the definition that for an effective divisor $T'$ in $T$ with the effective leaf $l$, no vertex in $T'$ belongs to the path from the root of $T$ to any leaf $l'$ located on the left of $l$.

\begin{exam}
	Consider the following three tree monomials with the same underlying tree:
	
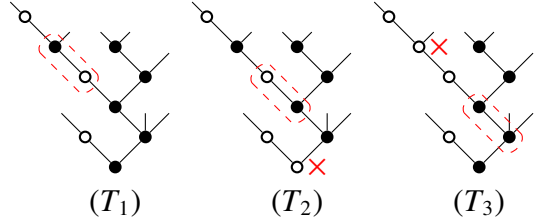
\begin{eqnarray*}
	\begin{tikzpicture}[scale=0.4]
		\tikzstyle{every node}=[thick,minimum size=4pt, inner sep=1pt]
		\node(r) at(0,-0.5)[minimum size=0pt, label=below:$(T_1)$]{};
		\node(0) at(0,0)[circle,draw, fill=black]{};
		\node(1-1)at(-1,1)[circle,draw]{};
		\node(1-2)at(1,1)[circle,draw,fill=black]{};

		\node(2-1) at(-2,2){};
		\node(2-2) at(0,2) [circle,draw,fill=black]{};
		\node(2-3) at(1,2){};
         \node(2-4) at(2,2){};
		\node(3-1) at (-1,3)[circle, draw]{};
		\node(3-2) at(1,3)[circle,draw,fill=black]{};
		\node(4-1) at (-2,4)[circle, draw, fill=black]{};
		\node(4-2) at (0,4)[circle,draw,fill=black]{};
		\node(4-3) at(2,4){};
		\node(5-1) at(-3,5)[circle,draw]{};
		\node(5-2) at(-1.3,4.7){};
		\node(5-3) at(-0.7,4.7){};
		\node(5-4)at (0.7,4.7){};
		\node(6-1) at (-3.7,5.7){};
		\draw(0)--(1-1);
		\draw(0)--(1-2);
		\draw(1-1)--(2-1);
		\draw(1-2)--(2-2);
		\draw(1-2)--(2-3);
        \draw(1-2)--(2-4);
		\draw(2-2)--(3-1);
		\draw(2-2)--(3-2);
		\draw(3-1)--(4-1);
		\draw(3-2)--(4-2);
		\draw(3-2)--(4-3);
		\draw(4-1)--(5-1);
		\draw(4-1)--(5-2);
		\draw(4-2)--(5-3);
		\draw(4-2)--(5-4);
		\draw(5-1)--(6-1);
		\draw[dashed,red](-2.5,3.8) to [in=150, out=120] (-1.9,4.4) ;
		\draw[dashed,red](-2.5,3.8)--(-1.2,2.5);
		\draw[dashed, red](-0.6,3.1)--(-1.9,4.4);
		\draw[dashed,red] (-0.6,3.1)to [in=-30, out=-60] (-1.2,2.5);
	\end{tikzpicture}	
\begin{tikzpicture}[scale=0.4]
	\tikzstyle{every node}=[thick,minimum size=4pt, inner sep=1pt]
		\node(r) at(0,-0.5)[minimum size=0pt, label=below:$(T_2)$]{};
		\node(0) at(0,0)[circle,draw, label=right:${\color{red}\times}$]{};
		\node(1-1)at(-1,1)[circle,draw]{};
		\node(1-2)at(1,1)[circle,draw,fill=black]{};
		\node(2-1) at(-2,2){};
		\node(2-2) at(0,2) [circle,draw,fill=black]{};
		\node(2-3) at(1,2){};
         \node(2-4) at(2,2){};
		\node(3-1) at (-1,3)[circle, draw]{};
		\node(3-2) at(1,3)[circle,draw,fill=black]{};
		\node(4-1) at (-2,4)[circle, draw, fill=black]{};
		\node(4-2) at (0,4)[circle,draw,fill=black]{};
		\node(4-3) at(2,4){};
		\node(5-1) at(-3,5)[circle,draw]{};
		\node(5-2) at(-1.3,4.7){};
		\node(5-3) at(-0.7,4.7){};
		\node(5-4)at (0.7,4.7){};
		\node(6-1) at (-3.7,5.7){};
		\draw(0)--(1-1);
		\draw(0)--(1-2);
		\draw(1-1)--(2-1);
		\draw(1-2)--(2-2);
		\draw(1-2)--(2-3);
\draw(1-2)--(2-4);
	\draw(2-2)--(3-1);
		\draw(2-2)--(3-2);
		\draw(3-1)--(4-1);
		\draw(3-2)--(4-2);
		\draw(3-2)--(4-3);
		\draw(4-1)--(5-1);
		\draw(4-1)--(5-2);
		\draw(4-2)--(5-3);
		\draw(4-2)--(5-4);
		\draw(5-1)--(6-1);
			\draw[dashed,red](-1.5,2.8) to [in=150, out=120] (-0.9,3.4) ;
		\draw[dashed,red](-1.5,2.8)--(-0.2,1.5);
		\draw[dashed, red](0.4,2.1)--(-0.9,3.4);
		\draw[dashed,red] (0.4,2.1)to [in=-30, out=-60] (-0.2,1.5);
	\end{tikzpicture}
	\begin{tikzpicture}[scale=0.4]
		\tikzstyle{every node}=[thick,minimum size=4pt, inner sep=1pt]
		\node(r) at(0,-0.5)[minimum size=0pt, label=below:$(T_3)$]{};
		\node(0) at(0,0)[circle,draw, fill=black]{};
		\node(1-1)at(-1,1)[circle,draw]{};
		\node(1-2)at(1,1)[circle,draw,fill=black]{};
		\node(2-1) at(-2,2){};
		\node(2-2) at(0,2) [circle,draw,fill=black]{};
		\node(2-3) at(1,2){};
         \node(2-4) at(2,2){};
		\node(3-1) at (-1,3)[circle, draw]{};
		\node(3-2) at(1,3)[circle,draw,fill=black]{};
		\node(4-1) at (-2,4)[circle, draw, label=right:$\color{red}\times$]{};
		\node(4-2) at (0,4)[circle,draw,fill=black]{};
		\node(4-3) at(2,4){};
		\node(5-1) at(-3,5)[circle,draw]{};
		\node(5-2) at(-1.3,4.7){};
		\node(5-3) at(-0.7,4.7){};
		\node(5-4)at (0.7,4.7){};
		\node(6-1) at (-3.7,5.7){};
		\draw(0)--(1-1);
		\draw(0)--(1-2);
		\draw(1-1)--(2-1);
		\draw(1-2)--(2-2);
		\draw(1-2)--(2-3);
\draw(1-2)--(2-4);
		\draw(2-2)--(3-1);
		\draw(2-2)--(3-2);
		\draw(3-1)--(4-1);
		\draw(3-2)--(4-2);
		\draw(3-2)--(4-3);
		\draw(4-1)--(5-1);
		\draw(4-1)--(5-2);
		\draw(4-2)--(5-3);
		\draw(4-2)--(5-4);
		\draw(5-1)--(6-1);
			\draw[dashed,red](-0.5,1.8) to [in=150, out=120] (0.1,2.4) ;
		\draw[dashed,red](-0.5,1.8)--(0.8,0.5);
		\draw[dashed, red](1.4,1.1)--(0.1,2.4);
		\draw[dashed,red] (1.4,1.1)to [in=-30, out=-60] (0.8,0.5);
	\end{tikzpicture}
\end{eqnarray*}

	\begin{enumerate}
		\item The tree monomial $T_1$ is effective and the divisor in the red/dashed circle is its effective divisor. Note that there is a vertex of positive degree on the path from the root of the tree to the root of the effective divisor.
		\item The tree monomial $T_2$ is not effective, because  there is a vertex of degree $1$ (the root itself) on the path from the root of the entire tree to the first leaf.
		\item For the tree monomial $T_3$, although there is a typical divisor $m_3\circ_1 m_2$ on the path from the root to the second leaf, there is a vertex  $d_2$ of positive degree on the path from the root of this divisor to the second leaf of $T_3$. Thus $T_3$ is not effective.
	\end{enumerate}
	\end{exam}

\begin{defn}
Let $T$ be an effective tree monomial in $\Difinfty$ and $T'$ be its effective divisor. Assume that $T'=\widehat{S}$, where $S$ is a generator of positive degree. Then define $$\overline{H}(T):=(-1)^\omega \frac{1}{c_S}m_{T', S}(T),$$ where $m_{T',S}(T)$ is the tree monomial obtained from $T$ by replacing the effective divisor $T'$ by $S$, $\omega$ is the sum of degrees of all the vertices on the path from the root of $T'$ to the root of $T$  (except the root vertex of $T'$) and on the left of this path.
\end{defn}

Now, we construct the homotopy map $H$ via the following  inductive  procedure:
\begin{enumerate}
	\item For a non-effective tree monomial $T$, define $H(T)=0$;
	\item For an effective tree monomial $T$, define $H(T)= \overline{H}(T)+H(\overline{T})$, where $\overline{T}$ is obtained from $T$ by replacing $\widehat{S}$ in $T$ by $\widehat{S}-\frac{1}{c_S}\partial S$. Then each tree monomials appearing in $\overline{T}$ is strictly smaller than $T$, and we define $H(\overline{T})$ by taking induction on the leading terms.
\end{enumerate}

We explain more on the definition of $H$. Denote $T$ also by $T_1$ and take $I_1:=\{1\}$ for convenience. Then by the definition above, $H(T)=\overline{H}(T_1)+H(\overline{T_1})$. Since $H$ vanishes on non-effective tree monomials, we have $H(\overline{T}_1)=H(\sum_{i_2\in I_2} T_{i_2})$ where $\{T_{i_2}\}_{i_2\in I_2}$ is the set of effective tree monomials together with their nonzero coefficients appearing in the expansion of $\overline{T_1}$. Then by the definition of $H$, we have $H\Big(\sum_{i_2\in I_2} T_{i_2}\Big)=\overline{H}\Big(\sum_{i_2\in I_2} T_{i_2}\Big)+H\Big(\sum_{i_2\in I_2}\overline{T_{i_2}}\Big),$
giving
$$H(T)=\overline{H}\Big(\sum_{i_1\in I_1} T_{i_1}\Big)+\overline{H}\Big(\sum_{i_2\in I_2} T_{i_2}\Big)+H\Big(\sum_{i_2\in I_2}\overline{T_{i_2}}\Big).$$ An induction on the leading terms shows that $H(T)$ is the following series:
\begin{eqnarray}
	\mlabel{eq:defhomo}
	H(T)=\sum_{k=1}^\infty
	\overline{H}\Big(\sum_{i_k\in I_k} T_{i_k}\Big),
	\end{eqnarray}
where $\{T_{i_k}\}_{i_k\in I_k}, k\geqslant 1,$ is the set of effective tree monomials  appearing in the expansion of $\sum_{i_{k-1}\in I_{k-1}} \overline{T_{i_{k-1}}}$ with nonzero coefficients.
	
\begin{lem}\mlabel{Lem: homotopy well defined}  For an effective tree monomial $T$, the expansion of $H(T)$ in Eq. \meqref{eq:defhomo} is always a finite sum, that is, there exists an integer $n$ such that none of the tree monomials in $\sum_{i_n\in I_n}\overline{T_{i_n}}$ is effective.
\end{lem}
\begin{proof}	 It is easily seen that $\max\{T_{i_k}|i_k\in I_k\}>\max\{T_{i_{k+1}}|i_{k+1}\in I_{k+1}\}$ for all   $k\geqslant 1$, where by convention, $i_1\in I_1=\{1\}$. Then Eq.~\meqref{eq:defhomo} is a finite sum since $>$ is a well order.
\end{proof}

\begin{lem}  \mlabel{Lem: Induction}
For an effective tree monomial $T\in \Difinfty$, we have
$$\partial \overline{H}(T)+H\partial(T-\overline{T})=T-\overline{T}.$$
\end{lem}

\begin{proof}
Express $T$ via partial compositions as follows:
	$$(\cdots(((((\cdots(X_1\circ_{i_1}X_2)\circ_{i_2}\cdots )\circ_{i_{p-1}}X_p)\circ_{i_p} \widehat{S})\circ_{j_1}Y_1)\circ_{j_2}Y_2)\circ_{j_3} \cdots)\circ_{j_q}Y_q,$$
	where $\widehat{S}$ is the effective divisor of $T$ and $X_1,\cdots, X_p$ are generators of $\Difinfty$ corresponding to the vertices that live on the path from the root of $T$ to the root of $\widehat{S}$ (except the root of $\widehat{S}$) and on the left of this path.
	
By definition,
{\small\begin{align*}
		 \partial \overline{H}(T)
=&\frac{1}{c_S}(-1)^{\sum\limits_{j=1}^p|X_j|}\partial \Big((\cdots(( ((\cdots(X_1\circ_{i_1}X_2)\circ_{i_2}\cdots )\circ_{i_{p-1}}X_p)\circ_{i_p} S)\circ_{j_1}Y_1)\circ_{j_2}  \cdots)\circ_{j_q}Y_q\Big)\\
 =&\frac{1}{c_S} \sum\limits_{k=1}^p(-1)^{\sum\limits_{j=1}^p|X_j|+\sum\limits_{j=1}^{k-1}|X_j|}\\
		& (\cdots ((((\cdots ((\cdots (X_1\circ_{i_1}X_2)\circ_{i_2}\cdots)\circ_{i_{k-1}}\partial X_k )\circ_{i_k}\cdots)\circ_{i_{p-1}}X_p)\circ_{i_p} S)\circ_{j_1}Y_1)\circ_{j_2} \cdots)\circ_{j_q}Y_q \\
&+\frac{1}{c_S}(\cdots((((\cdots(X_1\circ_{i_1}X_2)\circ_{i_2}\cdots )\circ_{i_{p-1}}X_p)\circ_{i_p} \partial S)\circ_{j_1}Y_1)\circ_{j_2} \cdots)\circ_{j_q}Y_q\\
		 &+ \frac{1}{c_S}\sum\limits_{k=1}^q (-1)^{|S|+\sum\limits_{j=1}^{k-1}|Y_j|}\\
		 &(\cdots((\cdots ((((\cdots(X_1\circ_{i_1}X_2)\circ_{i_2}\cdots )\circ_{i_{p-1}}X_p)\circ_{i_p} S)\circ_{j_1}Y_1)\circ_{j_2} \cdots \cdots)\circ_{j_k}\partial Y_{k})\circ_{j_{k+1}}\cdots )\circ_{j_q}Y_q
\end{align*}}
and
{\small
\begin{align*}
 &H\partial (T-\overline{T})\\
		 =&\frac{1}{c_S}H\partial \Big((\cdots((((\cdots(X_1\circ_{i_1}X_2)\circ_{i_2}\cdots )\circ_{i_{p-1}}X_p)\circ_{i_p} \partial S)\circ_{j_1}Y_1)\circ_{j_2} \cdots)\circ_{j_q}Y_q\Big)\\
 =&\frac{1}{c_S} \sum\limits_{k=1}^p(-1)^{\sum\limits_{j=1}^p|X_j|+\sum\limits_{j=1}^{k-1}|X_j|}\\
		& H\big((\cdots ((((\cdots ((\cdots (X_1\circ_{i_1}X_2)\circ_{i_2}\cdots)\circ_{i_{k-1}}\partial X_k )\circ_{i_k}\cdots)\circ_{i_{p-1}}X_p)\circ_{i_p} \partial S)\circ_{j_1}Y_1)\circ_{j_2} \cdots)\circ_{j_q}Y_q \big)\\
		 &+\frac{1}{c_S}\sum\limits_{k=1}^q(-1)^{\sum\limits_{j=1}^p|X_j|+|S|-1+\sum\limits_{j=1}^{k-1}|Y_j|}\\
		 &H\Big((\cdots((\cdots ((((\cdots(X_1\circ_{i_1}X_2)\circ_{i_2}\cdots )\circ_{i_{p-1}}X_p)\circ_{i_p} \partial S)\circ_{j_1}Y_1)\circ_{j_2} \cdots)\circ_{j_k}\partial Y_k)\circ_{j_{k+1}}\cdots)\circ_{j_q}Y_q\Big)\\
=&  \sum\limits_{k=1}^p(-1)^{\sum\limits_{j=1}^p|X_j|+\sum\limits_{j=1}^{k-1}|X_j|}\\
		& H\big((\cdots ((((\cdots ((\cdots (X_1\circ_{i_1}X_2)\circ_{i_2}\cdots)\circ_{i_{k-1}}\partial X_k )\circ_{i_k}\cdots)\circ_{i_{p-1}}X_p)\circ_{i_p} \widehat{S})\circ_{j_1}Y_1)\circ_{j_2} \cdots)\circ_{j_q}Y_q \big)\\
 &+  \sum\limits_{k=1}^p(-1)^{\sum\limits_{j=1}^p|X_j|+\sum\limits_{j=1}^{k-1}|X_j|}\\
		& H\big((\cdots ((((\cdots ((\cdots (X_1\circ_{i_1}X_2)\circ_{i_2}\cdots)\circ_{i_{k-1}}\partial X_k )\circ_{i_k}\cdots)\circ_{i_{p-1}}X_p)\circ_{i_p} (\frac{1}{c_S} \partial S-\widehat{S}))\circ_{j_1}Y_1)\circ_{j_2} \cdots)\circ_{j_q}Y_q \big)\\
 &+ \sum\limits_{k=1}^q(-1)^{\sum\limits_{j=1}^p|X_j|+|S|-1+\sum\limits_{j=1}^{k-1}|Y_j|}\\
		 &H\Big((\cdots((\cdots ((((\cdots(X_1\circ_{i_1}X_2)\circ_{i_2}\cdots )\circ_{i_{p-1}}X_p)\circ_{i_p} \widehat{S})\circ_{j_1}Y_1)\circ_{j_2} \cdots)\circ_{j_k}\partial Y_k)\circ_{j_{k+1}}\cdots)\circ_{j_q}Y_q\Big)\\
&+\sum\limits_{k=1}^q(-1)^{\sum\limits_{j=1}^p|X_j|+|S|-1+\sum\limits_{j=1}^{k-1}|Y_j|}\\
		 &H\Big((\cdots((\cdots ((((\cdots(X_1\circ_{i_1}X_2)\circ_{i_2}\cdots )\circ_{i_{p-1}}X_p)\circ_{i_p}(\frac{1}{c_S} \partial S-\widehat{S}))\circ_{j_1}Y_1)\circ_{j_2} \cdots)\circ_{j_k}\partial Y_k)\circ_{j_{k+1}}\cdots)\circ_{j_q}Y_q\Big).
	\end{align*}
}
By the definition of effective divisors in an effective tree monomial, it can be easily seen that each tree monomial in the expansion   of 	
	 $$T_k':=(\cdots ((((\cdots ((\cdots (X_1\circ_{i_1}X_2)\circ_{i_2}\cdots)\circ_{i_{k-1}}\partial X_k )\circ_{i_k}\cdots)\circ_{i_{p-1}}X_p)\circ_{i_p} \widehat{S})\circ_{j_1}Y_1)\circ_{j_2} \cdots)\circ_{j_q}Y_q$$
 and 	$$T_k'':=(\cdots((\cdots ((((\cdots(X_1\circ_{i_1}X_2)\circ_{i_2}\cdots )\circ_{i_{p-1}}X_p)\circ_{i_p} \widehat{S})\circ_{j_1}Y_1)\circ_{j_2} \cdots)\circ_{j_k}\partial Y_k)\circ_{j_{k+1}}\cdots)\circ_{j_q}Y_q$$
	is  still an effective tree monomial with $\widehat{S}$ as the effective divisor.
We thus have
{\small\begin{align*}
		 &H\partial (T-\overline{T})\\
=& \frac{1}{c_S} \sum\limits_{k=1}^p(-1)^{\sum\limits_{j=1}^p|X_j|+\sum\limits_{j=1}^{k-1}|X_j|}  (H(T_k')- H(\overline{T_k'}))+ \sum\limits_{k=1}^q(-1)^{\sum\limits_{j=1}^p|X_j|+|S|-1+\sum\limits_{j=1}^{k-1}|Y_j|} (H(T_k'')- H(\overline{T_k''}))\\
 =&\frac{1}{c_S} \sum\limits_{k=1}^p(-1)^{\sum\limits_{j=1}^p|X_j|+\sum\limits_{j=1}^{k-1}|X_j|}\overline{H}(T_k') + \sum\limits_{k=1}^q(-1)^{\sum\limits_{j=1}^p|X_j|+|S|-1+\sum\limits_{j=1}^{k-1}|Y_j|}   \overline{H}(T_k'') \\
  =&\frac{1}{c_S} \sum\limits_{k=1}^p(-1)^{\sum\limits_{j=1}^p|X_j|+\sum\limits_{j=1}^{k-1}|X_j|}\\
		&  (\cdots ((((\cdots ((\cdots (X_1\circ_{i_1}X_2)\circ_{i_2}\cdots)\circ_{i_{k-1}}\partial X_k )\circ_{i_k}\cdots)\circ_{i_{p-1}}X_p)\circ_{i_p}   S)\circ_{j_1}Y_1)\circ_{j_2} \cdots)\circ_{j_q}Y_q  \\
		  &+\frac{1}{c_S} \sum\limits_{k=1}^q(-1)^{|S|-1+\sum\limits_{j=1}^{k-1}|Y_j|}\\
		 &(\cdots((\cdots ((((\cdots(X_1\circ_{i_1}X_2)\circ_{i_2}\cdots )\circ_{i_{p-1}}X_p)\circ_{i_p} S)\circ_{j_1}Y_1)\circ_{j_2} \cdots \cdots)\circ_{j_k}\partial Y_{k})\circ_{j_{k+1}}\cdots )\circ_{j_q}Y_q.
\end{align*}}
Therefore, adding the expansions of $\partial \overline{H}(T)$ and $H\partial(T-\overline{T})$, we obtain
$$\hspace{4.5cm} \partial \overline{H}(T)+H\partial(T-\overline{T})=T-\overline{T}.
	\hspace{4.5cm} \qedhere$$
\end{proof}

\begin{prop}\mlabel{Prop: homotopy}
The degree $1$ map $H$ defined in Eq.~\meqref{eq:defhomo} satisfies $\partial H+H\partial=\mathrm{id}$ on  $\Difinfty$  in  each positive degree.
\end{prop}
\begin{proof}
We first prove that for a non-effective tree monomial $T$, the equation $\partial H(T)+H\partial(T)=T$  holds. By the definition of $H$, since $T$ is not effective, $H(T)=0$. Thus we just need to check $H\partial(T)=T$. Since $T$ has positive degree, there exists at least one vertex of positive degree. We pick such a vertex $S$ that satisfies the following additional conditions:
	\begin{enumerate}
		\item On the path from $S$ to the leftmost leaf $l$ of $T$ above $S$, there are no other vertices of positive degree;
		\item  For a leaf $l'$ of $T$ located on the left of $l$, the vertices  on the path from the root of $T$ to $l'$ are all of degree 0.
	\end{enumerate}
Such a vertex always exists.
	Then the element in $\Difinfty$ corresponding to $T$ can be written as
	$$(\cdots((((\cdots(X_1\circ_{i_1}X_2)\circ\cdots )\circ_{i_{p-1}}X_p)\circ_{i_p} S)\circ_{j_1}Y_1)\circ_{j_2}Y_2\cdots)\circ_{j_q}Y_q ,$$
	where $X_1,\cdots,X_p$ correspond to the vertices along the path from the root of $T$ to $S$ and the vertices on the left of this path.
Thus, all $X_i, i=1, \dots, p$ are of degree zero.

Then by definition,
	{\small\begin{align*}
		 H\partial T=&H\ \Big( (-1)^{|X_1|+\cdots+|X_p|}(\cdots((((\cdots(X_1\circ_{i_1}X_2)\circ_{i_2}\cdots )\circ_{i_{p-1}}X_p)\circ_{i_p} \partial S)\circ_{j_1}Y_1)\circ_{j_2} \cdots)\circ_{j_q}Y_q\Big)\\
		&+\sum\limits_{k=1}^q(-1)^{\sum\limits_{t=1}^p|X_t|+|S|+\sum\limits_{t=1}^{k-1}|Y_t|} \\
&\quad H\ \Big((\cdots((\cdots((((\cdots(X_1\circ_{i_1}X_2)\circ\cdots )\circ_{i_{p-1}}X_p)\circ_{i_p} S)\circ_{j_1}Y_1)\circ_{j_2} \cdots)\circ_{j_k}\partial Y_{k})\circ_{j_{k+1}}\cdots\circ_{j_q}Y_q\Big).
	\end{align*}}
	By the assumption, the divisor consisting of the path from $S$ to $l$ must be of the following forms
	
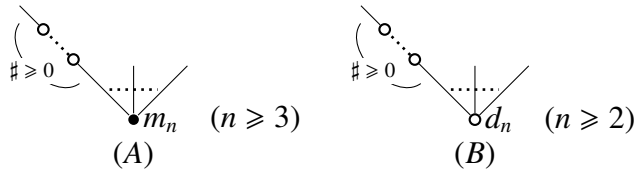
\begin{eqnarray*}
		\begin{tikzpicture}[scale=0.4,descr/.style={fill=white}]
			\tikzstyle{every node}=[thick,minimum size=4pt, inner sep=1pt]
			\node(r) at(0,-0.5)[minimum size=0pt, label=below:$(A)$]{};
			\node(0) at(0,0)[circle, fill=black, label=right:$m_n\quad (n\geqslant 3)$]{};
			\node(1-1) at(-2,2)[circle, draw]{};
			\node(1-2) at(0,2){};
			\node(1-3) at (2,2){};
			\node(2-1) at(-3,3)[circle,draw]{};
			\node(3-1) at(-4,4){};
			\draw(0)--(1-1);
			\draw(0)--(1-2);
			\draw(0)--(1-3);
			\draw[dotted, line width=1pt](-0.8,1)--(0.8,1);
			\draw[dotted, line width=1pt](1-1)--(2-1);
			\draw(2-1)--(3-1);
			\path[-,font=\scriptsize]
			(-1.8,1.2) edge [bend left=80] node[descr]{{\tiny$\sharp\geqslant 0$}} (-3.8,3.2);	
		\end{tikzpicture}
		\hspace{5mm}
		\begin{tikzpicture}[scale=0.4,descr/.style={fill=white}]
			\tikzstyle{every node}=[thick,minimum size=4pt, inner sep=1pt]
			\node(r) at(0,-0.5)[minimum size=0pt, label=below:$(B)$]{};
			\node(0) at(0,0)[circle, draw, label=right:$d_n\quad (n\geqslant 2)$]{};
			\node(1-1) at(-2,2)[circle, draw]{};
			\node(1-2) at(0,2){};
			\node(1-3) at (2,2){};
			\node(2-1) at(-3,3)[circle,draw]{};
			\node(3-1) at(-4,4){};
			\draw(0)--(1-1);
			\draw(0)--(1-2);
			\draw(0)--(1-3);
			\draw[dotted, line width=1pt](-0.8,1)--(0.8,1);
			\draw[dotted, line width=1pt](1-1)--(2-1);
			\draw(2-1)--(3-1);
			\path[-,font=\scriptsize]
			(-1.8,1.2) edge [bend left=80] node[descr]{{\tiny$\sharp\geqslant 0$}} (-3.8,3.2);
		\end{tikzpicture}
	\end{eqnarray*}
	By the assumption that $T$ is not effective and the additional properties of the position of $S$ stated above, the effective tree monomials in $\partial T$ can only appear in the expansion of
	$$(-1)^{|X_1|+\cdots+|X_p|}(\cdots((((\cdots(X_1\circ_{i_1}X_2)\circ\cdots )\circ_{i_{p-1}}X_p)\circ_{i_p} \partial S)\circ_{j_1}Y_1)\circ_{j_2}Y_2\cdots)\circ_{j_q}Y_q.$$
	
	Consider the tree monomial $(\cdots((((\cdots(X_1\circ_{i_1}X_2)\circ\cdots )\circ_{i_{p-1}}X_p)\circ_{i_p} \widehat{ S})\circ_{j_1}Y_1)\circ_{j_2}Y_2\cdots)\circ_{j_q}Y_q$ in $\partial T$. The path connecting the root of $\widehat{S}$ and $l$ must be one of the following forms:
	\begin{eqnarray*}
		\begin{tikzpicture}[scale=0.4,descr/.style={fill=white}]
			\tikzstyle{every node}=[thick,minimum size=4pt, inner sep=1pt]
			\node(r) at(0,-0.5)[minimum size=0pt,label=below:$(A)$]{};
			\node(1) at (0,0)[circle,draw,fill=black,label=right:$\ m_{n-1}\quad (n\geqslant 3)$]{};
			\node(2-1) at(-1,1)[circle,draw,fill=black]{};
			\node(2-2) at (1,1) {};
			\node(3-2) at(0,2){};
			\node(3-1) at (-2,2)[circle,draw]{};
			\node(4-1) at (-3,3)[circle,draw]{};
			\node(5-1) at (-4,4){};
			\draw(1)--(2-1);
			\draw(1)--(2-2);
			\draw(2-1)--(3-1);
			\draw(2-1)--(3-2);
			\draw [dotted,line width=1pt](-0.4,0.5)--(0.4,0.5);
			\draw(2-1)--(3-1);
			\draw[dotted,line width=1pt] (3-1)--(4-1);
			\draw(4-1)--(5-1);
			\path[-,font=\scriptsize]
			(-1.8,1.2) edge [bend left=80] node[descr]{{\tiny$\sharp\geqslant 0$}} (-3.8,3.2);
		\end{tikzpicture}
		\hspace{10mm}
		\begin{tikzpicture}[scale=0.4,descr/.style={fill=white}]
			\tikzstyle{every node}=[thick,minimum size=4pt, inner sep=1pt]
			\node(r) at(0,-0.5)[minimum size=0pt,label=below:$(B)$]{};
			\node(1) at (0,0)[circle,draw,label=right:$\ d_{n-1}\quad (n\geqslant 2)$]{};
			\node(2-1) at(-1,1)[circle,draw,fill=black]{};
			\node(2-2) at (1,1) {};
			\node(3-1) at(-2,2)[circle,draw]{};
			\node(3-2) at(0,2){};
			\node(4-1) at(-3,3)[circle,draw]{};
			\node(5-1) at (-4,4){};
			\draw(1)--(2-1);
			\draw(1)--(2-2);
			\draw(2-1)--(3-1);
			\draw(2-1)--(3-2);
			\draw[dotted,line width=1pt](3-1)--(4-1);
			\draw(4-1)--(5-1);
			\draw [dotted,line width=1pt](-0.4,0.5)--(0.4,0.5);
			\path[-,font=\scriptsize]
			(-1.8,1.2) edge [bend left=80] node[descr]{{\tiny$\sharp\geqslant 0$}} (-3.8,3.2);
		\end{tikzpicture}
	\end{eqnarray*}
By the assumption that $T$ is not effective and the choice of $S$, there exists no effective divisor on the left of the path from the root of $T$ to $l$. So the tree monomial $$(\cdots((((\cdots(X_1\circ_{i_1}X_2)\circ\cdots )\circ_{i_{p-1}}X_p)\circ_{i_p} \widehat{ S})\circ_{j_1}Y_1)\circ_{j_2}Y_2\cdots)\circ_{j_q}Y_q$$
is effective and its effective divisor is simply $\widehat{S}$ itself.
Then we have
{\small
\begin{eqnarray*}&&H\partial T\\
	&=&H\Big((-1)^{|X_1|+\cdots+|X_p|}(\cdots((((\cdots(X_1\circ_{i_1}X_2)\circ_{i_2}\cdots )\circ_{i_{p-1}}X_p)\circ_{i_p} \partial S)\circ_{j_1}Y_1)\circ_{j_2} \cdots)\circ_{j_q}Y_q\Big)\\
	&=&c_SH\Big((-1)^{|X_1|+\cdots+|X_p|}(\cdots((((\cdots(X_1\circ_{i_1}X_2)\circ\cdots )\circ_{i_{p-1}}X_p)\circ_{i_p} \widehat{ S})\circ_{j_1}Y_1)\circ_{j_2}Y_2\cdots)\circ_{j_q}Y_q\Big)\\
	&&+H\Big((-1)^{|X_1|+\cdots+|X_p|}(\cdots((((\cdots(X_1\circ_{i_1}X_2)\circ_{i_2}\cdots )\circ_{i_{p-1}}X_p)\circ_{i_p} (\partial S-c_S\widehat{S}))\circ_{j_1}Y_1)\circ_{j_2} \cdots)\circ_{j_q}Y_q\Big)\\
	&=& c_S\overline{H}\Big((-1)^{|X_1|+\cdots+|X_p|}(\cdots((((\cdots(X_1\circ_{i_1}X_2)\circ_{i_2}\cdots )\circ_{i_{p-1}}X_p)\circ_{i_p} \widehat{ S})\circ_{j_1}Y_1)\circ_{j_2} \cdots)\circ_{j_q}Y_q\Big)\\
	&&+c_SH\Big((-1)^{|X_1|+\cdots+|X_p|}(\cdots((((\cdots(X_1\circ_{i_1}X_2)\circ_{i_2}\cdots )\circ_{i_{p-1}}X_p)\circ_{i_p}  (\widehat{S}-\frac{1}{c_S}\partial S))\circ_{j_1}Y_1)\circ_{j_2}Y_2\cdots)\circ_{j_q}Y_q\Big)\\
	&&+H\Big((-1)^{|X_1|+\cdots+|X_p|}(\cdots((((\cdots(X_1\circ_{i_1}X_2)\circ_{i_2}\cdots )\circ_{i_{p-1}}X_p)\circ_{i_p} (\partial S-c_S\widehat{S}))\circ_{j_1}Y_1)\circ_{j_2}Y_2\cdots)\circ_{j_q}Y_q\Big)\\
	&=& c_S\overline{H}\Big((-1)^{|X_1|+\cdots+|X_p|}(\cdots((((\cdots(X_1\circ_{i_1}X_2)\circ_{i_2}\cdots )\circ_{i_{p-1}}X_p)\circ_{i_p} \widehat{ S})\circ_{j_1}Y_1)\circ_{j_2}Y_2\cdots)\circ_{j_q}Y_q\Big)\\
	&=&(\cdots((((\cdots(X_1\circ_{i_1}X_2)\circ_{i_2}\cdots )\circ_{i_{p-1}}X_p)\circ_{i_p} S)\circ_{j_1}Y_1)\circ_{j_2}Y_2\cdots)\circ_{j_q}Y_q\\
		&=& T .
	\end{eqnarray*}
}	

Let $T$  be an effective tree monomial.
	By Lemma~\mref{Lem: Induction}, we have $\partial \overline{H}(T)+H\partial(T-\overline{T})=T-\overline{T}$. Moreover, since each tree monomial of  $\overline{T}$ is strictly smaller than $T$, by induction, we obtain $H\partial(\overline{T})+\partial H(\overline{T})=\overline{T}$.
	By  $H(T)=\overline{H}(T)+H(\overline{T})$, we get
  $$\hspace{1cm} \partial H(T)+H\partial (T)=\partial \overline{H}(T)+\partial H(\overline{T})+H\partial(T-\overline{T})+H\partial(\overline{T})=T-\overline{T}+\overline{T}=T.
 \hspace{1cm} \qedhere $$
\end{proof}

Combining Lemma~\mref{Lem: H0 of Difinfty is Dif} and Proposition~\mref{Prop: homotopy}, we obtain a surjective  quasi-isomorphism $p: \Difinfty\to\Dif$.
Notice that conditions (i) and (ii) of Definition~\mref{de:model} hold by the construction of the dg operad $\Difinfty$. Thus the proof of Theorem~\mref{thm:difmodel} is completed.

\subsection{Homotopy differential algebras with weight}\
\mlabel{ss:homo}

\begin{defn}\mlabel{de:homodifalg} Let $V=(V, m_1)$ be a complex.  A \name{homotopy differential algebra structure of weight $\lambda$} on $V$ is a homomorphism $\Difinfty\to \mathrm{End}_V$ of unital dg operads. \end{defn}
Definition~\mref{de:homodifalg} translates to the following explicit characterization of homotopy differential algebras.
\begin{prop}\mlabel{de:homodifalg2}  A homotopy differential algebra structure of weight $\lambda$ on a    graded space $V$ is equivalently given by  two family of operators $\{ m_n: V^{\ot n}\to V\}_{n\geqslant 1}$ and $\{d_n: V^{\ot n}\to V\}_{n\geqslant 1}$ with $|m_n|=n-2$ and $|d_n|=n-1$,  fulfilling the following identities: for each $n\geqslant 1$,
	\begin{eqnarray}\mlabel{eq:homodifalg2}
		\sum_{i+j+k=n\atop i, k\geqslant 0, j\geqslant 1}(-1)^{i+jk}m_{i+1+k}\circ(\Id^{\ot i}\ot m_j\ot \Id^{\ot k})=0,
	\end{eqnarray}

\begin{eqnarray}\mlabel{eq:homodifop2}
	 &&\sum_{i+j+k=n\atop i, k\geqslant 0, j\geqslant 1}(-1)^{i+jk}d_{i+1+k}\circ (\Id^{\ot i}\ot m_j\ot \Id^{\ot k})\\
	\notag  &&\quad \quad\quad\quad\quad =\sum_{\substack{l_1+\dots+l_q+j_1+\dots+j_{q+1}=n \\ j_1+\dots+j_{q+1}+q=p\\  j_1, \dots, j_{q+1}\geqslant 0, l_1, \dots, l_{q}\geqslant 1 \\n\geqslant p\geqslant q\geqslant 1}}(-1)^{\eta}\lambda^{q-1}m_p\circ(\Id^{\ot j_1}\ot d_{l_1}\ot \Id^{\ot j_2}\ot \cdots\ot d_{l_q}\ot \Id^{\ot j_{q+1}}),
	\end{eqnarray}
	where
	{\small
$$\eta:=\frac{n(n-1)}{2}+\frac{p(p-1)}{2}+\sum\limits_{k=1}^q\frac{l_k(l_k-1)}{2}+\sum\limits_{k=1}^q\big(l_k-1\big)\big(\sum\limits_{r=1}^kj_r+\sum\limits_{r=1}^{k-1}l_r\big)
=\sum_{k=1}^q(l_k-1)(q-k+\sum_{r=k+1}^{q+1}j_r).
$$}
\end{prop}

Notice that  Eq.~\meqref{eq:homodifalg2} is exactly the Stasheff identity  defining $A_\infty$-algebras \mcite{Sta63}.
As is well known, the homology $\rmH_*(V, m_1)$ endowed with the associative product induced by   $m_2$   is a graded  algebra.

Expanding  Eq.~\meqref{eq:homodifop2} for small $n$'s, one obtains
\begin{enumerate}
	\item when $n=1$, $|d_1|=0$ and $d_1\circ m_1=m_1\circ d_1,$
that is,   $d_1: (V, m_1)\to (V, m_1)$ is a chain map;
	\item when $n=2$,  $|d_2|=1$ and
	\begin{eqnarray*}  d_1\circ m_2-\Big(m_2\circ(d_1\ot \Id)+m_2\circ(\Id\ot d_1)+\lambda m_2\circ(d_1\ot d_1)\Big) \\
= d_2\circ(\Id\ot m_1 +m_1\ot \Id)+
	m_1\circ d_2,	\end{eqnarray*}
which shows that $d_1$ is, up to a homotopy given by $d_2$,  a differential operator of weight $\lambda$ with respect to the ``multiplication" $m_2$.
\end{enumerate}
As a consequence,   the homology $\rmH_*(V, m_1)$, endowed with the associative product induced by $m_2$ and the differential operator induced by  $d_1$, is a differential algebra of weight $\lambda$.
This  indicates that  the notion of  homotopy differential algebras is a  ``higher homotopy" version of  that of differential algebras with weight.

\section{From the  minimal model to the  $L_\infty$-algebra controlling deformations}
  \mlabel{ss:modelinf}

To finish the paper, we will use the homotopy cooperad $ \Dif^{\ac}$ to deduce the  $L_\infty$-structure controlling the deformations of    differential algebras with weight, as promised in Theorem~\mref{th:linfdiff}. As a consequence, homotopy differential algebras can be described as the Maurer-Catan elements in (a reduced version) of this $L_\infty$-algebra.

\subsection{Proof of Theorem~\mref{th:linfdiff}} \mlabel{ss: Linifty}\

By Proposition~\mref{pp:linfhomood} and Proposition-Definition~\mref{Prop: convolution homotopy operad}, for a graded space $V$, the $L_\infty$-algebra from the deformation complex of differential algebras is  $$\mathfrak{C}_{\Dif}(V):=\mathbf{Hom}( \Dif^{\ac}, \End_V)^{\prod}.$$

Now, we determine the $L_{\infty}$-algebra $\frakC_{\Dif}(V)$ explicitly and thus give the promised operadic proof of   Theorem~\mref{th:linfdiff}.

The sign rules in $ \Dif^{\ac}$ are complicated, so we use some transformations. Notice that there is a natural isomorphism of operads: $\mathbf{Hom}(\cals,\End_{sV}) \cong \End_V$. Then we have the following isomorphisms of homotopy operads:
$$\begin{array}{rcl}
\mathbf{Hom}( \Dif^{\ac}, \End_V)& \cong &  \mathbf{Hom}( \Dif^{\ac}, \mathbf{Hom}(\cals,\End_{sV}))  \cong   \mathbf{Hom}( \Dif^{\ac}\otimes_{\mathrm{H}}\cals, \End_{sV})\\
 & =  & \mathbf{Hom}(\mathscr{S} (\Dif^{\ac}), \End_{sV}).
\end{array}$$
Recall that for $n\geqslant 1$, $\mathscr{S} (\Dif^{\ac})(n)=\bfk \widetilde{m_n}\oplus \bfk \widetilde{d_n}$ with $|\widetilde{m_n}|=0$ and $|\widetilde{d_n}|=1$. By definition, we have
$$\begin{array}{rcl}
\mathbf{Hom}(\mathscr{S} (\Dif^{\ac}), \End_{sV})(n) & = & \Hom(\bfk \widetilde{m_n}\oplus \bfk \widetilde{d_n}, \Hom((sV)^{\otimes n}, sV))\\
& \cong & \Hom(\bfk \widetilde{m_n}, \Hom((sV)^{\otimes n}, sV))\oplus \Hom( \bfk \widetilde{d_n}, \Hom((sV)^{\otimes n}, sV)).
\end{array}$$
For a homogeneous element $f\in \Hom((sV)^{\otimes n}, sV)$, define
$$\widehat{f}:\bfk \widetilde{m_n} \to \Hom((sV)^{\otimes n}, sV)), \quad \widetilde{m_n} \mapsto f, $$
 and for a homogeneous element $g\in \Hom((sV)^{\otimes n}, V)$, define
 $$\overline{g}: \bfk \widetilde{d_n} \to \Hom((sV)^{\otimes n}, sV)), \quad  \widetilde{d_n} \mapsto (-1)^{|g|} sg.$$
The resulting bijections,  sending $f$ and $g$  to $\widehat{f}$ and $\overline{g}$ respectively,  allow us to identify $\frakC_{\DA}(V)$  with   $\mathbf{Hom}(\mathscr{S} (\Dif^{\ac}), \End_{sV})$.

Now we   compute the $L_{\infty}$-structure $\{ \ell_n \}_{n\geqslant 1}$ of $\frakC_{\Dif}(V)$.
Recall that for homogeneous elements $x_i\in \frakC_{\Dif}(V),1\leqslant i\leqslant n$,
 $$\ell_n(x_1,\cdots,x_n)=\sum_{\sigma \in \rmS_n} \chi(\sigma; x_1,\dots,x_n) m_n(x_{\sigma(1)}\otimes \cdots \otimes x_{\sigma(n)})$$
 with

 $$m_n(x_1\otimes \cdots \otimes x_n)=\sum_{T\in \frakt, w(T)=n} m_T(x_1\otimes \cdots \otimes x_n),$$
 where $\{m_T\}_{T\in \frakt}$ is given by the homotopy operad structure of $\mathbf{Hom}(\mathscr{S} (\Dif^{\ac}), \End_{sV})$.
One computes the maps $m_n$ as follows.
\begin{itemize}
\item[(i)] For homogeneous elements $f,g\in\End_{sV}$,
$m_2(\widehat{f}\otimes \widehat{g})=\widehat{f\{ g \}};$

\item[(ii)] For homogeneous elements $f,g\in\End_{sV}$,
$m_2(\overline{f}\otimes \widehat{g})=(-1)^{|\widehat{g}|} \overline{f\{ g \}};$

\item[(iii)] For homogeneous elements $f_0,f_1,\dots,f_n\in \End_{sV},n\geqslant 1$,
$$ m_{n+1}(\widehat{f_0}\otimes \overline{f_1}\otimes \cdots \otimes \overline{f_n})=  (-1)^{ (n+1)|\widehat{f_0}| + \sum_{k=1}^{n}(n-k)|\overline{f_k}|} \lambda^{n-1} \overline{f_0\{f_1,\dots, f_n\}};$$

\item[(iv)] All other components of operators $\{m_n \}_{n\geqslant 1}$ vanish.
\end{itemize}
Furthermore, the maps $\{\ell_n\}_{n\geqslant 1}$ are given as follows.
\begin{itemize}
\item[(i)] For homogeneous elements $f,g\in\End_{sV}$,
$$\begin{array}{rcl}
\ell_2(\widehat{f}\otimes \widehat{g})& = & m_2(\widehat{f}\otimes \widehat{g})-(-1)^{|\widehat{f}|| \widehat{g}|} m_2(\widehat{g}\otimes \widehat{f})\\
& = & \widehat{ f\{g\}} -(-1)^{|f||g|}\widehat{ g\{ f \}} \\
 & =&\widehat{[f,g]_G};
\end{array}$$

\item[(ii)] For homogeneous elements $f,g\in\End_{sV}$,
$$\begin{array}{rcl}
\ell_2(\widehat{f}\otimes \overline{g}) & = & m_2(\widehat{f}\otimes\overline{g})-(-1)^{|\widehat{f}|| \overline{g}|} m_2(\overline{g}\otimes \widehat{f})\\
& = & \overline{f\{g\}} -(-1)^{|f||g|}\overline{g\{f\}}
  \\
 & =&   \overline{[f,g]_G};
\end{array}$$

\item[(iii)] For homogeneous elements $f_0,f_1,\dots,f_n\in \End_{sV}$ with $n\geqslant 1$,
$$\begin{aligned}
    \ell_{n+1}(\widehat{f_0}\otimes \overline{f_1}\otimes \cdots \otimes \overline{f_n})
 = & \sum_{\sigma\in \rmS_n} \chi(\sigma;\overline{f_1},\dots, \overline{f_n})m_n(\widehat{f_0} \otimes \overline{f_{\sigma(1)}}\otimes \cdots \otimes \overline{f_{\sigma(n)}})\\
 = &  \sum_{\sigma\in \rmS_n}(-1)^{(n+1)|\widehat{f_0}|+\sum_{k=1}^{n-1}(n-k)|\overline{f_{\sigma(k)}}|}  \lambda^{n-1} \overline{f_0\{ f_{\sigma(1)},\dots ,f_{\sigma(n)}  \}};
\end{aligned}$$

\item[(iv)] For homogeneous elements $f_0,f_1,\dots,f_n\in \End_{sV}$ with $n\geqslant 1$ and $1\leqslant i\leqslant n$,
$$\ell_{n+1}(\overline{f_1}\otimes \cdots \otimes \overline{f_{i}} \otimes \widehat{f_0} \otimes \overline{f_{i+1}} \otimes \cdots \otimes \overline{f_n})=(-1)^{ i+|\widehat{f_0}|(\sum_{k=1}^{i} |\overline{f_k}|)} \ell_{n+1} (\widehat{f_0}\otimes \overline{f_1}\otimes \cdots \otimes \overline{f_n});$$

\item[(v)] All other components of the operators $\{\ell_n \}_{n\geqslant 1}$ vanish.
\end{itemize}

Via the  bijections     $f\mapsto \widehat{f}$ and $g\mapsto \overline{g}$, it is readily verified that the above $L_{\infty}$-structure on  $\mathbf{Hom}(\mathscr{S} (\Dif^{\ac}), \End_{sV})$  is exactly the one on    $\frakC_{\DA}(V)$  defined in Eq~\meqref{eq:linfdiff}, thereby proving Theorem~\mref{th:linfdiff}.

\subsection{Another description of homotopy differential algebras}\
\mlabel{ss: another definition}

We end the paper with a characterization of the key notion of homotopy differential algebras and give a generalization of a work of Kajura and Stasheff~\mcite{KS06}.

By the last paragraph of \S\mref{ss:homocood}, there exists another definition of homotopy differential algebras in terms of the Maurer-Cartan elements in the $L_\infty$-algebra on the deformation complex. Let us make it precise.

 Let $V$ be a graded space. Denote
 $$\overline{T}(V)\coloneqq \oplus_{n\geqslant 1} V^{\ot n},\ \overline{\frakC_\Alg}(V):=\Hom(\overline{T}(sV),sV),\  \overline{\frakC_\DO}(V):= \Hom(\overline{T}(sV),V).  $$
  We consider the subspace
  $$\overline{{\frakC}_{\DA}}(V):=\overline{\frakC_\Alg}(V)\oplus\overline{\frakC_\DO}(V)$$
  of $\frakC_\DA(V)$ in Eq.~\meqref{eq:linfdiff}.
It is easy to see that   $\overline{{\frakC}_{\DA}}(V)$ is an $L_\infty$-subalgebra of $\frakC_{\DA}(V)$.

By the remark at the end of \S\mref{ss:homocood}, Definitions~\mref{de:homodifalg}  can be rephrased as follows.
\begin{prop}\mlabel{de:homodifalg3}
Let $V$ be a graded space. A \name{homotopy differential algebra structure} of weight $\lambda$ on $V$ is defined to be a Maurer-Cartan element in the   $L_\infty$-algebra $\overline{{\frakC}_{\DA}}(V)$.
\end{prop}

Notice that if we take the whole space  $\frakC_{\DA}(V)$ instead of the reduced part $\overline{\frakC_\Alg}(V)$, we get the curved   homotopy  differential algebra structure.

Solving the Maurer-Cartan equation in the $L_\infty$-algebra  $\overline{{\frakC}_{\DA}}(V)$,  Proposition~\mref{de:homodifalg3} can be explicated as follows.

\begin{prop} \mlabel{de:homodifalg4}
A homotopy differential algebra structure of weight $\lambda$ on a graded space $V$ is equivalent to two families of linear maps
$$b_n: (sV)^{\ot n}\rightarrow sV, \quad  R_n:(sV)^{\ot n}\rightarrow V, \quad n\geqslant 1,$$
both of degree $-1$ and subject  to the following equations:
$$\sum_{i+j-1=n\atop i, j\geqslant 1} b_{i}\{b_j\}=0 \ \mathrm{and}\
	\sum_{u+j-1=n\atop u, j\geqslant 1}  sR_{u}\{b_j\}
	=\sum_{l_1+\dots+l_q+p-q=n \atop  {l_1, \dots, l_q\geqslant 1\atop
  1\leqslant q\leqslant p\leqslant n} }\lambda^{q-1} b_p\{sR_{l_1},\dots,sR_{l_q}\}.
	$$
\end{prop}

\begin{remark}
\begin{enumerate}
\item
The equivalence between Definition~\mref{de:homodifalg2} and Proposition~\mref{de:homodifalg4} is given by
 \[m_n:=s^{-1}\circ b_n\circ s^{\ot n} :V^{\ot n}\rightarrow V,\quad  d_n:=R_n\circ s^{\ot n}:V^{\ot n}\rightarrow V,\ n\geqslant 1.\]

  \item Thanks to Proposition~\mref{de:homodifalg4}, Definition~\mref{de:homodifalg2} generalizes the notion of homotopy derivations on $A_\infty$-algebras introduced by
Kajura and Stasheff \mcite{KS06} from zero weight   to nonzero weight.
\end{enumerate}
\end{remark}

\noindent
{\bf Acknowledgments. } This work was   supported  by the National Natural Science Foundation of China (No.  12071137, 11971460), by  Key Laboratory of MEA (Ministry of Education), by  Shanghai Key Laboratory of PMMP  (No.   22DZ2229014),  and by Fundamental Research Funds for the Central Universities.

 The authors would like to express our sincere gratitude to the referee for helpful suggestions which improve much the presentation of  the paper.

\noindent
{\bf Declaration of interests. } The authors have no conflicts of interest to disclose.

\noindent
{\bf Data availability. } No new data were created or analyzed in this study.

\end{document}